\documentclass{amsart}
\usepackage{amsmath,amssymb}
\usepackage{mathrsfs}
\newtheorem{df}{Definition}

\newtheorem{thm}{Theorem}
\newtheorem{lem}{Lemma}
\newtheorem{rem}{Remark}
\newtheorem{con}{Conjecture}
\newtheorem{ex}{Example}
\newcommand{\relmiddle}[1]{\mathrel{}\middle#1\mathrel{}}
\numberwithin{equation}{section}
\numberwithin{figure}{section}
\numberwithin{df}{section}
\numberwithin{prp}{section}
\numberwithin{thm}{section}
\numberwithin{lem}{section}
\numberwithin{rem}{section}
\numberwithin{con}{section}
\numberwithin{ex}{section}
\begin{document}

\title[fixed-point-free involution words]
{Queer Supercrystal Structure for Increasing Factorizations of Fixed-Point-Free Involution Words}
\author{Toya Hiroshima}
\address{OCAMI, Osaka City University,
3-3-138 Sugimoto, Sumiyoshi-ku, Osaka 558-8585, Japan}
\email{toya-hiroshima@outlook.jp}
\date{}

\begin{abstract}
We show that the set of increasing factorizations of fixed-point-free (FPF) involution words has the structure of queer supercrystals.
By exploiting the algorithm of symplectic shifted Hecke insertion recently introduced by Marberg~\cite{Mar}, we establish the one-to-one correspondence between the set of increasing factorizations of fixed-point-free involution words and the set of primed tableau 
(semistandard marked shifted tableaux) and the latter admits the structure of queer supercrystals.
In order to establish the correspondence, we prove that the Coxeter-Knuth related FPF-involution words have the same insertion tableau in the symplectic shifted Hecke insertion, where the insertion tableau is an increasing shifted tableau and the recording tableau is a primed tableau.
\end{abstract}


\maketitle

\section{Introduction}

Recently, Marberg~\cite{Mar} has introduced the symplectic shifted Hecke insertion, which is the symplectic shifted analogue of Hecke insertion introduced by Buch, Kresch, Shimozono, Tamvakis, and Yong~\cite{BKSTY} in their study of Grothendieck polynomials.
By his insertion algorithm, Marberg established a bijection between increasing factorizations of symplectic Hecke words and the pair of shifted Young tableaux of the same shape such that the left is an increasing shifted tableau and the right is a set-valued shifted tableau, which are called insertion and recording tableaux, respectively.
Restricted to reduced symplectic Hecke words, which are called fixed-point-free (FPF) involution words, set-valued shifted tableaux turn out to be primed tableaux (semistandard marked shifted tableaux).
On the other hand, Assaf and Oguz~\cite{AO1,AO2} and the author~\cite{Hir} showed independently that the set of primed tableaux admits the structure of crystals for queer Lie superalgebra or simply queer supercrystals discovered by Grantcharov et al.~\cite{GJKKK1,GJKKK2,GJKKK3,GJKK}.
Combining these results, it is expected that the set of increasing factorizations of FPF-involution words is a queer supercrystal and the affirmative answer is given in this paper.
To do so, we prove that Coxeter-Knuth related (Coxeter braid related in the terminology of \cite{Mar}) FPF-involution words have the same insertion tableau, which ensures the one-to-one correspondence between the set of increasing factorizations of FPF-involution words and the set of primed tableaux. 
 
The paper is organized as follows.
In Section~\ref{sec:queer}, we review crystals for the queer Lie superalgebra.
Definitions of increasing shifted tableaux and primed tableaux are given in Section~\ref{sec:tableaux}.
In Section~\ref{sec:insertion}, we explain the algorithm of FPF-involution Coxeter-Knuth insertion (symplectic shifted Hecke insertion restricted to FPF-involution words) introduced in \cite{Mar} and show that Coxeter-Knuth related FPF-involution words have the same insertion tableau (Theorem~\ref{thm:CK}).
The proof is so lengthy that it is relegated to Appendix A.
In Section~\ref{sec:factorization}, we show that the set of increasing factorizations of FPF-involution words admits the structure of queer supercrystals (Theorem~\ref{thm:q-FPF}).
The odd Kashiwara operators are given in Lemmas~\ref{lem:f0F1} and \ref{lem:e0F1}.
We briefly mention the results of the ``orthogonal'' version without proofs in Appendix B.

\section{Crystals for the queer Lie superalgebra} \label{sec:queer}

Let us start by briefly sketching crystals for the general linear Lie algebra $\mathfrak{gl}(n)$ or simply $\mathfrak{gl}(n)$-crystals~\cite{BS,HK,Kas2}.
Let 
$P=
\bigoplus _{i=1}^{n}
\mathbb{Z\epsilon}_{i}$ 
be the weight lattice and 
$P^{\vee}=
\bigoplus _{i=1}^{n}
\mathbb{Z}k_{i}$ 
be the dual weight lattice with 
$\left\langle \epsilon_{i},k_{j}\right\rangle =\delta_{ij}$ for ($i,j=1,\ldots,n$).
Let 
$\left\{  \alpha_{i}=\epsilon_{i}-\epsilon_{i+1} \mid 1\leq i <n \right\}  $ be the set of simple roots and 
$\left\{  h_{i}=k_{i}-k_{i+1} \mid 1\leq i <n \right\}  $ be the set of simple coroots.
Let
\[
P^{+}=\left\{  \lambda \mid \lambda\in P,\left\langle \lambda,h_{i}\right\rangle \geq0 \; (i=1,\ldots,n-1) \right\}
\] 
be the set of dominant integral weights.

\begin{df} \label{df:crystal}
A set $B$ together with maps
$\mathrm{wt}:B\rightarrow P$ and 
$\Tilde{e_{i}},\Tilde{f_{i}}:B\rightarrow B\sqcup\{\boldsymbol{0}\}$
is called a $\mathfrak{gl}(n)$-crystal if the following properties are satisfied for $i=1,\ldots,n-1$:
when we define 
\[
\varepsilon_{i}(b)=\max\left\{  k\geq0 \relmiddle| \Tilde{e_{i}}^{k}b\in B\right\},
\]
and
\[
\varphi_{i}(b)=\max\left\{  k\geq0 \relmiddle| \Tilde{f_{i}}^{k}b\in B\right\},
\]
for $b\in B$, then
\begin{itemize}
\item[(1)]
$\varepsilon_{i},\varphi_{i}:B\rightarrow\mathbb{Z}_{\geq0}$ 
and 
$\varphi_{i}(b)=\varepsilon_{i}(b)+
\left\langle \mathrm{wt}(b), h_{i} \right\rangle $,
\item[(2)]
if $\Tilde{e_{i}}b\neq \boldsymbol{0}$, then 
$\mathrm{wt}(\Tilde{e_{i}}b)=\mathrm{wt}(b)+\alpha_{i}$, 
$\varepsilon_{i}(\Tilde{e_{i}}b)=\varepsilon_{i}(b)-1$, and 
$\varphi_{i}(\Tilde{e_{i}}b)=\varphi_{i}(b)+1$,
\item[(3)]
if $\Tilde{f_{i}}b\neq \boldsymbol{0}$, then 
$\mathrm{wt}(\Tilde{f_{i}}b)=\mathrm{wt}(b)-\alpha_{i}$, 
$\varepsilon_{i}(\Tilde{f_{i}}b)=\varepsilon_{i}(b)+1$, and 
$\varphi_{i}(\Tilde{f_{i}}b)=\varphi_{i}(b)-1$,
\item[(4)]
for $b,b^{\prime}\in B$, 
$\Tilde{f_{i}}b=b^{\prime}\Longleftrightarrow\Tilde{e_{i}}b^{\prime}=b$.
\end{itemize}
\end{df}
The maps $\Tilde{e_{i}}$ and $\Tilde{f_{i}}$ are called Kashiwara operators 
and $\mathrm{wt}(b)$ is called the weight of $b$.
In Definition~\ref{df:crystal}, $\boldsymbol{0}$ is a formal symbol; $b=\boldsymbol{0}$ implies $b\notin B$.

\begin{rem}
More precisely, crystals in Definition~\ref{df:crystal} are called semiregular~\cite{HK} or seminormal~\cite{BS}.
Throughout this paper, we consider only semiregular or seminormal crystals.
\end{rem}

\begin{df}[tensor product rule] \label{df:tensor}
Let $B_{1}$ and $B_{2}$ be $\mathfrak{gl}(n)$-crystals.
The tensor product $B_{1}\otimes B_{2}$ is defined to be 
the set 
$B_{1}\times B_{2} = 
\left\{  b_{1}\otimes b_{2} \mid b_{1}\in B_{1},b_{2}\in B_{2}\right\}  $
whose crystal structure is defined by
\begin{itemize}
\item[(1)]
$\mathrm{wt}(b_{1}\otimes b_{2})=\mathrm{wt}(b_{1})+\mathrm{wt}(b_{2})$,
\item[(2)]
$\varepsilon_{i}(b_{1}\otimes b_{2})=
\max\left\{  \varepsilon_{i}(b_{2}),\varepsilon _{i}(b_{1})-\left\langle \mathrm{wt}(b_{2}),h_{i} \right\rangle \right\}$,
\item[(3)]
$\varphi_{i}(b_{1}\otimes b_{2})=
\max\left\{ \varphi_{i}(b_{1}), \varphi_{i}(b_{2})+\left\langle \mathrm{wt}(b_{1}), h_{i} \right\rangle \right\}$,
\item[(4)]
$\Tilde{e_{i}}(b_{1}\otimes b_{2})=
\begin{cases}
\Tilde{e_{i}}b_{1}\otimes b_{2} &  if \varphi_{i}(b_{2})<\varepsilon_{i}(b_{1}), \\
b_{1}\otimes\Tilde{e_{i}}b_{2} &  if \varphi_{i}(b_{2})\geq\varepsilon_{i}(b_{1}),
\end{cases}
$
\item[(5)]
$\Tilde{f_{i}}(b_{1}\otimes b_{2})=
\begin{cases}
\Tilde{f_{i}}b_{1}\otimes b_{2} & if \varphi_{i}(b_{2})\leq\varepsilon_{i}(b_{1}), \\
b_{1}\otimes\Tilde{f_{i}}b_{2} & if \varphi_{i}(b_{2})>\varepsilon_{i}(b_{1}),
\end{cases}
$
\end{itemize}
for $i=1,\ldots,n-1$.
\end{df}

\begin{rem}
Note that this definition is different from the one by the convention of Kashiwara.
Since we use Kashiwara operators on primed tableaux which are constructed by the anti-Kashiwara convention for the tensor product rule~\cite{HPS}, we adopt Definition~\ref{df:tensor}.
\end{rem}

Next, let us describe crystals for the queer Lie superalgebra $\mathfrak{q}(n)$ or simply $\mathfrak{q}(n)$-crystals introduced in \cite{GJKKK2,GJKKK3}.

\begin{df} \label{df:queer}
A $\mathfrak{q}(n)$-crystal is a set $B$ together with the maps $\mathrm{wt}:B\rightarrow P$, 
$\varepsilon_{i}, \varphi_{i}:B\rightarrow\mathbb{Z}_{\geq 0}$ and 
$\Tilde{e}_{i},\Tilde{f}_{i}:B\rightarrow B\sqcup\{\boldsymbol{0}\}$ for 
$i\in I:=\{1,\ldots,n-1,\Bar{1}\}$ 
satisfying the following conditions:

\begin{itemize}
\item[(1)]
$B$ is a $\mathfrak{gl}(n)$-crystal with respect to $\mathrm{wt}$, $\varepsilon_{i}$, $\varphi_{i}$, $\Tilde{e}_{i}$ 
and $\Tilde{f}_{i}$ for $i=1,\ldots,n-1$,
\item[(2)]
$\mathrm{wt}(b)\in
\bigoplus _{i=1}^{n}
\mathbb{Z}_{\geq 0}\epsilon_{i}$ 
for $b\in B,$
\item[(3)] $\mathrm{wt}(\Tilde{e}_{\Bar{1}}b)=\mathrm{wt}(b)+\alpha_{1}$, 
$\mathrm{wt}(\Tilde{f}_{\Bar{1}}b)=\mathrm{wt}(b)-\alpha_{1}$ 
for $b\in B$,
\item[(4)]
$\Tilde{f}_{\Bar{1}}b=b^{\prime}\Longleftrightarrow\Tilde{e}_{\Bar{1}}b^{\prime}=b$ for all $b,b^{\prime}\in B$,
\item[(5)] for $3\leq i\leq n-1$, we have 
\begin{itemize}
\item[(i)] the operators $\Tilde{e}_{\Bar{1}}$ and $\Tilde{f}_{\Bar{1}}$ commute with $\Tilde{e}_{i}$ and 
$\Tilde{f}_{i}$,
\item[(ii)] if $\Tilde{e}_{\Bar{1}}b\in B$, then 
$\varepsilon_{i}(\Tilde{e}_{\Bar{1}}b)=\varepsilon_{i}(b)$ and 
$\varphi_{i}(\Tilde{e}_{\Bar{1}}b)=\varphi_{i}(b)$.
\end{itemize}
\end{itemize}
\end{df}

The crystal associated with an irreducible highest weight module $V(\lambda)$ in the category $\mathcal{O}_{int}^{\geq 0}$~\cite{GJKKK1,GJKK} of tensor representations over the $\mathfrak{q}(n)$-quantum group $U_{q}(\mathfrak{q}(n))$ is denoted by $B_{n}(\lambda)$.
We identify the partition $\lambda$ with $\lambda=\sum_{i=1}^{n}\lambda_{i}\epsilon_{i}\in P^{+}$ in a usual way~\cite{HK}.
In $B_{n}(\lambda)$, the relevant partition $\lambda$ is a strict partition $\lambda=(\lambda_{1} > \lambda_{2} > \ldots > \lambda_{l} > \lambda_{l+1}= 0)$, where $l\leq n$.
A crystal $B$ can be viewed as an oriented colored graph with colors $i\in I$ 
when we define $b\stackrel{i}{\longrightarrow} b^{\prime}$ if 
$\Tilde{f_{i}}b=b^{\prime}\; (b,b^{\prime}\in B)$.
This graph is called a crystal graph.
The crystal graph of $B_{n}(\square)$, i.e., $B_{n}(\lambda)$ for $\lambda=\epsilon_{1}\in P^{+}$, is given by
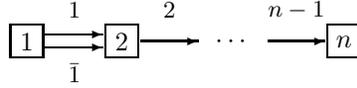
\setlength{\unitlength}{12pt}

\begin{center}
\begin{picture}(11,3)

\put(1.1,1.3){\vector(1,0){1.8}}
\put(1.1,1.7){\vector(1,0){1.8}}
\put(4.1,1.5){\vector(1,0){1.8}}
\put(8.1,1.5){\vector(1,0){1.8}}

\put(6,1){\makebox(2,1){$\cdots$}}

\put(0,1){\framebox(1,1){$1$}}
\put(3,1){\framebox(1,1){$2$}}
\put(10,1){\framebox(1,1){$n$}}

\put(1,0){\makebox(2,1){\small $\Bar{1}$}}
\put(1,2){\makebox(2,1){\small $1$}}
\put(4,2){\makebox(2,1){\small $2$}}
\put(8,2){\makebox(2,1){\small $n-1$}}

\end{picture}.
\end{center}

For $\mathfrak{q}(n)$-crystals $B_{1}$ and $B_{2}$, the tensor product $B_{1}\otimes B_{2}$ is a $\mathfrak{gl}(n)$-crystal, where actions of $\Tilde{e}_{\Bar{1}}$ and $\Tilde{f}_{\Bar{1}}$ on $b_{1}\otimes b_{2}$ ($b_{1}\in B_{1}$, $b_{2}\in B_{2}$) are prescribed as
\[
\Tilde{e}_{\Bar{1}}(b_{1}\otimes b_{2})=
\begin{cases}
b_{1} \otimes \Tilde{e}_{\Bar{1}}b_{2}, & 
if \left\langle \mathrm{wt}(b_{1}), k_{1}\right\rangle =\left\langle \mathrm{wt}(b_{1}), k_{2}\right\rangle=0, \\
\Tilde{e}_{\Bar{1}}b_{1} \otimes b_{2}, & otherwise,
\end{cases}
\]
\[
\Tilde{f}_{\Bar{1}}(b_{1}\otimes b_{2})=
\begin{cases}
b_{1} \otimes \Tilde{f}_{\Bar{1}}b_{2}, & 
if \left\langle \mathrm{wt}(b_{1}), k_{1}\right\rangle =\left\langle \mathrm{wt}(b_{1}), k_{2}\right\rangle=0, \\
\Tilde{f}_{\Bar{1}}b_{1} \otimes b_{2}, & otherwise.
\end{cases}
\]
Under these rules,  $B_{1}\otimes B_{2}$ is also a $\mathfrak{q}(n)$-crystal.

\begin{rem}
These rules are different from the ones given in \cite{CK,GJKKK2,GJKKK3} because we adopt the anti-Kashiwara convention for the tensor product rule.
\end{rem}

Let $B$ be a $\mathfrak{q}(n)$-crystal and suppose that $B$ is in a class of \emph{normal} $\mathfrak{gl}(n)$-crystal~\cite{BS}, i.e., every connected component in $B$ is in one-to-one correspondence with a dominant integral weight.
Here, the \emph{connected components} in $B$ are referred to as the maximal subcrystals of $B$ where all the elements are connected by even Kashiwara operators.
We define the automorphism $S_{i}$ on $B$ by
\[
S_{i}(b)=
\begin{cases}
\Tilde{f}_{i}^{\left\langle \mathrm{wt}(b),h_{i}\right\rangle }b, & 
if \left\langle \mathrm{wt}(b),h_{i}\right\rangle \geq0, \\
\Tilde{e}_{i}^{-\left\langle \mathrm{wt}(b),h_{i}\right\rangle }b, & 
if \left\langle \mathrm{wt}(b),h_{i}\right\rangle <0,
\end{cases}
\]
for $b\in B$ and $i=1,2,\ldots,n-1$.
Let $w$ be an element of Weyl group $W$ of $\mathfrak{gl}(n)$ which is generated by simple reflections $s_{i}$ for $i=1,\ldots,n-1$.
Then, there exists a unique action $S_{w}:B\rightarrow B$ of $W$ on $B$ such that $S_{s_{i}}=S_{i}$ 
for $i=1.\ldots,n-1$~\cite{Kas2}.
Let us set $w_{i}=s_{2}\cdots s_{i}s_{1}\cdots s_{i-1}$.
Then $w_{i}$ is the shortest element in $W$ such that $w_{i}(\alpha_{i})=\alpha_{1}$ 
for $i=2,\ldots,n-1$.
We define new operators by
\[
\Tilde{e}_{\Bar{\imath}} =S_{w_{i}^{-1}}\Tilde{e}_{\Bar{1}}S_{w_{i}},\quad  
\Tilde{f}_{\Bar{\imath}} =S_{w_{i}^{-1}}\Tilde{f}_{\Bar{1}}S_{w_{i}},
\]
where $S_{w_{i}}=S_{2}\cdots S_{i}S_{1}\cdots S_{i-1}$ and similar for $S_{w_{i}^{-1}}$.
These operators together with $\Tilde{e}_{\Bar{1}}$ and $\Tilde{f}_{\Bar{1}}$ are called \emph{odd} Kashiwara operators, while $\Tilde{e}_{i}$ and $\Tilde{f}_{i}$ ($i=1,\ldots,n-1$) are called \emph{even} Kashiwara operators.

\begin{thm}[\cite{GJKKK2,GJKKK3}] 
Let $B_{n}(\lambda)$ be a $\mathfrak{q}(n)$-crystal.
There is a unique element $b\in B_{n}(\lambda)$ such that
$\Tilde{e}_{i}b=\Tilde{e}_{\Bar{\imath}}b=\boldsymbol{0}$ for all $i=1,\ldots,n-1$, 
which is called a \emph{$\mathfrak{q}(n)$-highest weight vector} and 
there is a unique element $b\in B_{n}(\lambda)$ called a \emph{$\mathfrak{q}(n)$-lowest weight vector} such that 
$S_{w_{0}}b$ is a $\mathfrak{q}(n)$-highest weight vector, where $w_{0}$ is the longest element of $W$.
\end{thm}


\section{Shifted tableaux} \label{sec:tableaux}

Let $\mathcal{P}^{+}$ denote the set of strict partitions, 
$\lambda=(\lambda_{1}>\lambda_{2}>\cdots>\lambda_{l}>\lambda_{l+1}=0)$.
For $\lambda\in\mathcal{P}^{+}$, the length $l(\lambda)$ of $\lambda$ is defined as the number of positive parts of $\lambda$. 
The \emph{shifted diagram} of shape $\lambda\in\mathcal{P}^{+}$ is defined to be the set 
\[
S(\lambda)=\{(i,j)\in\mathbb{N}^{2} \mid i\leq j\leq\lambda_{i}+i-1,1\leq i\leq
l(\lambda)\}.
\]
A filling $T$ of $S(\lambda)$ with letters is called a \emph{shifted tableau} where the entry at $(i,j)$-position is denoted by $T_{(i,j)}$.

\begin{df}
An \emph{increasing shifted tableau} $T$ of shape $\lambda$ is a filling of $S(\lambda)$ with letters from the alphabet $\{1,2,\ldots \}$ such that entries are strictly increasing across rows and columns.
\end{df}

The row reading word of an increasing shifted tableau $T$ of shape $\lambda$, denoted by $\mathfrak{row}(T)$, is the sequence of entries, $T_{l(\lambda)}T_{l(\lambda)-1}\cdots T_{1}$, where $T_{i}$ is the sequence of entries of the $i$th row of $T$ read from left to right $(i=1,2,\ldots,l(\lambda))$.

\begin{ex}
The row reading word of
\setlength{\unitlength}{12pt}

\begin{center}
\begin{picture}(6,3)
\put(0,1){\makebox(2,1){$T=$}}
\put(2,2){\line(0,1){1}}
\put(3,1){\line(0,1){2}}
\put(4,0){\line(0,1){3}}
\put(5,0){\line(0,1){3}}
\put(6,2){\line(0,1){1}}
\put(4,0){\line(1,0){1}}
\put(3,1){\line(1,0){2}}
\put(2,2){\line(1,0){4}}
\put(2,3){\line(1,0){4}}

\put(4,0){\makebox(1,1){$6$}}
\put(3,1){\makebox(1,1){$4$}}
\put(4,1){\makebox(1,1){$5$}}
\put(2,2){\makebox(1,1){$1$}}
\put(3,2){\makebox(1,1){$3$}}
\put(4,2){\makebox(1,1){$4$}}
\put(5,2){\makebox(1,1){$5$}}
\end{picture}
\end{center}
is $\mathfrak{row}(T)=6451345$.
\end{ex}

\begin{df}
A \emph{primed tableau} $T$ of shape $\lambda$ is a filling of $S(\lambda)$ with letters from the alphabet, 
$\{1^{\prime}<1<2^{\prime}<2<\cdots<n\}$ such that: 
\begin{itemize}
\item[(1)] the entries are weakly increasing across row and columns,
\item[(2)] each row contains at most one $i^{\prime}$ for $i=1,\ldots,n$,
\item[(3)] each column contains at most one $i$ for $i=1,\ldots,n$, and
\item[(4)] there are no primed letters on the main diagonal.
\end{itemize}
We denote by $\mathrm{PT}_{n}(\lambda)$ the set of primed tableaux of shape $\lambda$. 
\end{df}

A primed tableau is called \emph{standard} if the set of entries consists of $1,2,\ldots,n$ each appearing either primed or unprimed exactly once for some $n$.

\begin{ex}
See Example~\ref{ex:standardization}, where the right tableau is the standard primed tableau and the left is not.
\end{ex}

\begin{rem}
A primed tableau is also called a semistandard marked shifted tableau~\cite{Cho} and is a set-valued shifted tableau~\cite{IN,PP} whose entries are all singleton sets.
\end{rem}

\section{The FPF-involution Coxeter-Knuth insertion} \label{sec:insertion}

Let us start by recalling the definition of reduced words for an element in $\mathfrak{S}_{\infty}$, the symmetric group generated by simple transpositions $s_{i}$ $(i=1,2,\ldots)$.  
We identify the product of simple transpositions with their sequence of indices or the word.
For a word $w$, we denote by $\left\vert w \right\vert$ the number of letters in $w$ and define 
$l(w)=\min \{l \mid \exists i_{1},\ldots,i_{l}, w=s_{i_{1}}\cdots s_{i_{l}} \}$.
A word $w$ with $\left\vert w\right\vert=l(w)$ is referred to as a \emph{reduced word}.
Two reduced words $w$ and $w^{\prime}$ are called \emph{Coxeter-Knuth equivalent}, denoted by $w \stackrel{\mathsf{CK}}{\sim} w^{\prime}$, if $w^{\prime}$ can be obtained from $w$ by a finite sequence of Coxeter-Knuth relations on three consecutive letters $(a+1)a(a+1)\sim a(a+1)a$, $bac\sim bca$, and $acb\sim cab$, where $a<b<c$ \cite{BS}. 
Note that Coxeter-Knuth relations here are called Coxeter braid relation in \cite{Mar} and are not the Coxeter-Knuth relations in \cite{Mar}.
It is confusing but we follow the terminology of \cite{BS} here.

Let $\Theta$ be the infinite product of cycles $(12)(34)(56)\cdots$ and define  
$\mathfrak{F}_{\infty}=\{ \pi ^{-1}\Theta \pi ; \pi \in \mathfrak{S}_{\infty}\}$.
We identify $z \in \mathfrak{F}_{\infty}$ with the word $z_{1}z_{2}\cdots z_{n}$ with $n$ being even as follows: 
$z(i)=z_{i}$ for $i=1,2,\ldots ,n$ and $z(i)=\Theta (i)$ for $i>n$, where $i\neq z_{i}$ for all $i=1,2,\ldots ,n$ (fixed-point-free). 

\begin{ex}
$z=
\begin{pmatrix}
1 & 2 & 3 & 4 & 5 & 6 & 7 & 8 &\ldots \\
5 & 4 & 6 & 2 & 1 & 3 & 8 & 7 & \ldots
\end{pmatrix}
\in \mathfrak{F}_{\infty}
$
is identified with $546213$.
\end{ex}

\begin{df}
A \emph{symplectic Hecke word} for $z \in \mathfrak{F}_{\infty}$ is any word $i_{1}i_{2}\cdots i_{l}$ such that $z=s_{i_{l}}\cdots s_{i_{2}}s_{i_{1}}\Theta s_{i_{1}}s_{i_{2}}\cdots s_{i_{l}}$.
The set of symplectic Hecke words for $z\in \mathfrak{F}_{\infty}$ is denoted by $\mathcal{H}_{\mathsf{Sp}}(z)$.
The reduced words in $\mathcal{H}_{\mathsf{Sp}}(z)$ are called \emph{FPF-involution words}~\cite{HM,HMP18,HMP17}.
The set of FPF-involution words in $\mathcal{H}_{\mathsf{Sp}}(z)$ is denoted by $\hat{\mathcal{R}}_{\mathsf{FPF}}(z)$.
\end{df}

In what follows, when we refer to $z$-related sets for $z\in \mathfrak{F}_{\infty}$, we do not attach the phrase ``for $z\in \mathfrak{F}_{\infty}$''.

\begin{rem}[\cite{Mar}]
Every symplectic Hecke word begins with an even letter.
\end{rem}

Two FPF-involution words, $w$ and $w^{\prime}$ are called \emph{equivalent}, denoted by $w\stackrel{\mathsf{Sp}}{\sim} w^{\prime}$, if $w^{\prime}$ can be obtained from $w$ by a finite sequence of relations on consecutive letters $ab\sim ba$ ($\left\vert a-b\right\vert >1$) and $a(a+1)a\sim (a+1)a(a+1)$ and the relation with $i_{1}(i_{1}-1)\cdots i_{n}\sim i_{1}(i_{1}+1)\cdots i_{n}$ ($i_{1}\geq 2$).

\begin{thm}[\cite{Mar}]
A symplectic Hecke word is an FPF-involution word if and only if its equivalence class contains no words with equal adjacent letters.
\end{thm}


Let $w=u_{1}u_{2}\cdots u_{l}\in \hat{\mathcal{R}}_{\mathsf{FPF}}(z)$.
From $w$, we recursively construct a sequence of pairs of tableaux,
\[
(\emptyset,\emptyset)=(P^{(0)},Q^{(0)}),(P^{(1)},Q^{(1)}),\ldots,(P^{(l)},Q^{(l)})=(P_{\mathsf{Sp}}(w),Q_{\mathsf{Sp}}(w)),
\]
where $P^{(k)}$ is an increasing shifted tableau and $Q^{(k)}$ is a primed tableau.
To obtain the tableau $P^{(k)}$, insert the letter $u_{k}$ into $P^{(k-1)}$ as follows:
First insert $u_{k}$ into the first row of $P^{(k-1)}$.
The rules for inserting $a$ into a row or column, denoted by $L$,  of the increasing shifted  tableau $T$ are as follows:
Find the smallest entry $b$ of $L$ with $a \leq b$.
If no such entry exists, then add $a$ to the end of $L$.
Stop.
Otherwise, 

\begin{enumerate}
\item If $a=b$, then leave $L$ unchanged and insert $a+1$ to the row below if $L$ is a row or to the next column to the right if $L$ is a column. 

\item If $L$ is a row and $b$ is the first entry of $L$ and $a\not\equiv b \pmod{2}$, then leave $L$ unchanged and insert $a+2$ to the next column to the right.

\item In all other cases, replace $b$ by $a$ in $L$ and insert $b$ to the row below if $L$ is a row or to the next column to the right if $L$ is a column or $b$ was on the main diagonal.
\end{enumerate}
To obtain $Q^{(k)}$, we add a box \framebox[12pt]{$k$} (resp. \framebox[12pt]{$k^{\prime}$}) to $Q^{(k-1)}$ at the terminated position if the insertion terminated with row (resp. column) insertion.
Tableaux $P_{\mathsf{Sp}}(w)$ and $Q_{\mathsf{Sp}}(w)$ are referred to as the \emph{insertion tableau} and the \emph{recording tableau}, respectively.
This algorithm is called the \emph{FPF-involution Coxeter-Knuth insertion}~\cite{Mar} and denoted by the map $\mathrm{H}_{\mathsf{Sp}}:w \mapsto (P_{\mathsf{Sp}}(w),Q_{\mathsf{Sp}}(w))$.
This is equivalent to the symplectic shifted Hecke insertion~\cite{Mar} restricted to FPF-involution words and is the shifted analogue of Edelman-Greene insertion~\cite{EG}.
The insertion process of a letter $x$ or a word $u_{1}u_{2}\cdots$ into an increasing shifted tableau $T$ is denoted by $T\stackrel{\mathsf{Sp}}{\leftarrow}x$ or $T\stackrel{\mathsf{Sp}}{\leftarrow}u_{1}u_{2}\cdots$, which also denotes the resulting tableau. 

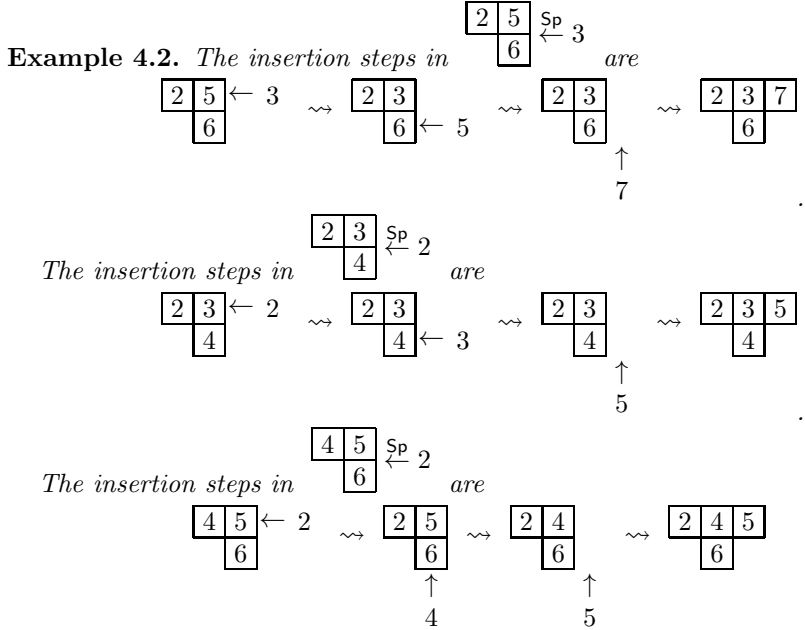
\begin{ex}

The insertion steps in 
\setlength{\unitlength}{12pt}
\begin{picture}(4,2)
\put(0,1){\line(0,1){1}}
\put(1,0){\line(0,1){2}}
\put(2,0){\line(0,1){2}}
\put(1,0){\line(1,0){1}}
\put(0,1){\line(1,0){2}}
\put(0,2){\line(1,0){2}}
\put(0,1){\makebox(1,1){$2$}}
\put(1,0){\makebox(1,1){$6$}}
\put(1,1){\makebox(1,1){$5$}}
\put(2,0.25){\makebox(2,2){$\stackrel{\mathsf{Sp}}{\leftarrow}3$}}
\end{picture}
are
 
\setlength{\unitlength}{12pt}
\begin{center}
\begin{picture}(20,4)
\put(0,3){\line(0,1){1}}
\put(1,2){\line(0,1){2}}
\put(2,2){\line(0,1){2}}
\put(1,2){\line(1,0){1}}
\put(0,3){\line(1,0){2}}
\put(0,4){\line(1,0){2}}

\put(0,3){\makebox(1,1){$2$}}
\put(1,2){\makebox(1,1){$6$}}
\put(1,3){\makebox(1,1){$5$}}

\put(2,3){\makebox(1,1){$\leftarrow$}}

\put(3,3){\makebox(1,1){$3$}}

\put(4,2){\makebox(2,2){$\rightsquigarrow$}}

\put(6,3){\line(0,1){1}}
\put(7,2){\line(0,1){2}}
\put(8,2){\line(0,1){2}}
\put(7,2){\line(1,0){1}}
\put(6,3){\line(1,0){2}}
\put(6,4){\line(1,0){2}}

\put(6,3){\makebox(1,1){$2$}}
\put(7,2){\makebox(1,1){$6$}}
\put(7,3){\makebox(1,1){$3$}}

\put(8,2){\makebox(1,1){$\leftarrow$}}

\put(9,2){\makebox(1,1){$5$}}

\put(10,2){\makebox(2,2){$\rightsquigarrow$}}

\put(12,3){\line(0,1){1}}
\put(13,2){\line(0,1){2}}
\put(14,2){\line(0,1){2}}
\put(13,2){\line(1,0){1}}
\put(12,3){\line(1,0){2}}
\put(12,4){\line(1,0){2}}

\put(12,3){\makebox(1,1){$2$}}
\put(13,2){\makebox(1,1){$6$}}
\put(13,3){\makebox(1,1){$3$}}

\put(14,1){\makebox(1,1){$\uparrow$}}

\put(14,0){\makebox(1,1){$7$}}

\put(15,2){\makebox(2,2){$\rightsquigarrow$}}

\put(17,3){\line(0,1){1}}
\put(18,2){\line(0,1){2}}
\put(19,2){\line(0,1){2}}
\put(20,3){\line(0,1){1}}
\put(18,2){\line(1,0){1}}
\put(17,3){\line(1,0){3}}
\put(17,4){\line(1,0){3}}

\put(17,3){\makebox(1,1){$2$}}
\put(18,2){\makebox(1,1){$6$}}
\put(18,3){\makebox(1,1){$3$}}
\put(19,3){\makebox(1,1){$7$}}

\end{picture}.
\end{center}

The insertion steps in 
\setlength{\unitlength}{12pt}
\begin{picture}(4,2)

\put(0,1){\line(0,1){1}}
\put(1,0){\line(0,1){2}}
\put(2,0){\line(0,1){2}}
\put(1,0){\line(1,0){1}}
\put(0,1){\line(1,0){2}}
\put(0,2){\line(1,0){2}}

\put(0,1){\makebox(1,1){$2$}}
\put(1,0){\makebox(1,1){$4$}}
\put(1,1){\makebox(1,1){$3$}}

\put(2,0.25){\makebox(2,2){$\stackrel{\mathsf{Sp}}{\leftarrow}2$}}

\end{picture}
are
 
\setlength{\unitlength}{12pt}

\begin{center}
\begin{picture}(20,4)

\put(0,3){\line(0,1){1}}
\put(1,2){\line(0,1){2}}
\put(2,2){\line(0,1){2}}
\put(1,2){\line(1,0){1}}
\put(0,3){\line(1,0){2}}
\put(0,4){\line(1,0){2}}

\put(0,3){\makebox(1,1){$2$}}
\put(1,2){\makebox(1,1){$4$}}
\put(1,3){\makebox(1,1){$3$}}

\put(2,3){\makebox(1,1){$\leftarrow$}}

\put(3,3){\makebox(1,1){$2$}}

\put(4,2){\makebox(2,2){$\rightsquigarrow$}}

\put(6,3){\line(0,1){1}}
\put(7,2){\line(0,1){2}}
\put(8,2){\line(0,1){2}}
\put(7,2){\line(1,0){1}}
\put(6,3){\line(1,0){2}}
\put(6,4){\line(1,0){2}}

\put(6,3){\makebox(1,1){$2$}}
\put(7,2){\makebox(1,1){$4$}}
\put(7,3){\makebox(1,1){$3$}}

\put(8,2){\makebox(1,1){$\leftarrow$}}

\put(9,2){\makebox(1,1){$3$}}

\put(10,2){\makebox(2,2){$\rightsquigarrow$}}

\put(12,3){\line(0,1){1}}
\put(13,2){\line(0,1){2}}
\put(14,2){\line(0,1){2}}
\put(13,2){\line(1,0){1}}
\put(12,3){\line(1,0){2}}
\put(12,4){\line(1,0){2}}

\put(12,3){\makebox(1,1){$2$}}
\put(13,2){\makebox(1,1){$4$}}
\put(13,3){\makebox(1,1){$3$}}

\put(14,1){\makebox(1,1){$\uparrow$}}

\put(14,0){\makebox(1,1){$5$}}

\put(15,2){\makebox(2,2){$\rightsquigarrow$}}

\put(17,3){\line(0,1){1}}
\put(18,2){\line(0,1){2}}
\put(19,2){\line(0,1){2}}
\put(20,3){\line(0,1){1}}
\put(18,2){\line(1,0){1}}
\put(17,3){\line(1,0){3}}
\put(17,4){\line(1,0){3}}

\put(17,3){\makebox(1,1){$2$}}
\put(18,2){\makebox(1,1){$4$}}
\put(18,3){\makebox(1,1){$3$}}
\put(19,3){\makebox(1,1){$5$}}

\end{picture}.
\end{center}

The insertion steps in 
\setlength{\unitlength}{12pt}
\begin{picture}(4,2)

\put(0,1){\line(0,1){1}}
\put(1,0){\line(0,1){2}}
\put(2,0){\line(0,1){2}}
\put(1,0){\line(1,0){1}}
\put(0,1){\line(1,0){2}}
\put(0,2){\line(1,0){2}}

\put(0,1){\makebox(1,1){$4$}}
\put(1,0){\makebox(1,1){$6$}}
\put(1,1){\makebox(1,1){$5$}}

\put(2,0.25){\makebox(2,2){$\stackrel{\mathsf{Sp}}{\leftarrow}2$}}

\end{picture}
are
 
\setlength{\unitlength}{12pt}

\begin{center}
\begin{picture}(18,4)

\put(0,3){\line(0,1){1}}
\put(1,2){\line(0,1){2}}
\put(2,2){\line(0,1){2}}
\put(1,2){\line(1,0){1}}
\put(0,3){\line(1,0){2}}
\put(0,4){\line(1,0){2}}

\put(0,3){\makebox(1,1){$4$}}
\put(1,2){\makebox(1,1){$6$}}
\put(1,3){\makebox(1,1){$5$}}

\put(2,3){\makebox(1,1){$\leftarrow$}}

\put(3,3){\makebox(1,1){$2$}}

\put(4,2){\makebox(2,2){$\rightsquigarrow$}}

\put(6,3){\line(0,1){1}}
\put(7,2){\line(0,1){2}}
\put(8,2){\line(0,1){2}}
\put(7,2){\line(1,0){1}}
\put(6,3){\line(1,0){2}}
\put(6,4){\line(1,0){2}}

\put(6,3){\makebox(1,1){$2$}}
\put(7,2){\makebox(1,1){$6$}}
\put(7,3){\makebox(1,1){$5$}}

\put(7,1){\makebox(1,1){$\uparrow$}}

\put(7,0){\makebox(1,1){$4$}}

\put(8,2){\makebox(2,2){$\rightsquigarrow$}}

\put(10,3){\line(0,1){1}}
\put(11,2){\line(0,1){2}}
\put(12,2){\line(0,1){2}}
\put(11,2){\line(1,0){1}}
\put(10,3){\line(1,0){2}}
\put(10,4){\line(1,0){2}}

\put(10,3){\makebox(1,1){$2$}}
\put(11,2){\makebox(1,1){$6$}}
\put(11,3){\makebox(1,1){$4$}}

\put(12,1){\makebox(1,1){$\uparrow$}}
\put(12,0){\makebox(1,1){$5$}}

\put(13,2){\makebox(2,2){$\rightsquigarrow$}}

\put(15,3){\line(0,1){1}}
\put(16,2){\line(0,1){2}}
\put(17,2){\line(0,1){2}}
\put(18,3){\line(0,1){1}}
\put(16,2){\line(1,0){1}}
\put(15,3){\line(1,0){3}}
\put(15,4){\line(1,0){3}}

\put(15,3){\makebox(1,1){$2$}}
\put(16,2){\makebox(1,1){$6$}}
\put(16,3){\makebox(1,1){$4$}}
\put(17,3){\makebox(1,1){$5$}}

\end{picture}.
\end{center}

When we specify a row or column indicated by an arrow, we do not attach the symbol $\mathsf{Sp}$ in the arrow.
This convention is also used in Appendix A.
\end{ex}

\begin{ex} \label{ex:insertion}
The FPF-involution Coxeter-Knuth insertion of an FPF-involution word $6241$ yields

\setlength{\unitlength}{12pt}

\begin{center}
\begin{picture}(19,2)
\put(0,0){\makebox(5,2){$P_{\mathsf{Sp}}(6241)=$}}

\put(5,1){\line(0,1){1}}
\put(6,0){\line(0,1){2}}
\put(7,0){\line(0,1){2}}
\put(8,1){\line(0,1){1}}
\put(6,0){\line(1,0){1}}
\put(5,1){\line(1,0){3}}
\put(5,2){\line(1,0){3}}

\put(5,1){\makebox(1,1){$2$}}
\put(6,0){\makebox(1,1){$6$}}
\put(6,1){\makebox(1,1){$3$}}
\put(7,1){\makebox(1,1){$4$}}

\put(8.5,0){\makebox(2,2){$\text{and}$}}

\put(11,0){\makebox(5,2){$Q_{\mathsf{Sp}}(6241)=$}}

\put(16,1){\line(0,1){1}}
\put(17,0){\line(0,1){2}}
\put(18,0){\line(0,1){2}}
\put(19,1){\line(0,1){1}}
\put(17,0){\line(1,0){1}}
\put(16,1){\line(1,0){3}}
\put(16,2){\line(1,0){3}}

\put(16,1){\makebox(1,1){$1$}}
\put(17,0){\makebox(1,1){$3$}}
\put(17,1){\makebox(1,1){$2^{\prime}$}}
\put(18,1){\makebox(1,1){$4^{\prime}$}}
\end{picture}.
\end{center}

\end{ex}

The following results are due to Marberg~\cite{Mar}.

\begin{thm}[\cite{Mar}]
Suppose that $T$ is an increasing shifted tableau and $a$ is a letter such that $\mathfrak{row}(T)a$ is an FPF-involution word, then $\mathfrak{row}(T\stackrel{\mathsf{Sp}}{\leftarrow}a)\stackrel{\mathsf{Sp}}{\sim}\mathfrak{row}(T)a$.
\end{thm}

\begin{thm}[\cite{Mar}] \label{thm:pathology} 
If $w$ and $w^{\prime}$ are FPF-involution words with $P_{\mathsf{Sp}}(w)=P_{\mathsf{Sp}}(w^{\prime})$, then $w\stackrel{\mathsf{Sp}}{\sim} w^{\prime}$.
\end{thm}

The converse of Theorem~\ref{thm:pathology} does not hold.
For example, we have $26434 \stackrel{\mathsf{Sp}}{\sim} 42346$ but

\setlength{\unitlength}{12pt}

\begin{center}
\begin{picture}(21,2)
\put(0,0){\makebox(6,2){$P_{\mathsf{Sp}}(26434)=$}}

\put(6,1){\line(0,1){1}}
\put(7,0){\line(0,1){2}}
\put(8,0){\line(0,1){2}}
\put(9,0){\line(0,1){2}}
\put(7,0){\line(1,0){2}}
\put(6,1){\line(1,0){3}}
\put(6,2){\line(1,0){3}}

\put(6,1){\makebox(1,1){$2$}}
\put(7,0){\makebox(1,1){$4$}}
\put(7,1){\makebox(1,1){$3$}}
\put(8,0){\makebox(1,1){$6$}}
\put(8,1){\makebox(1,1){$4$}}

\put(9.5,0){\makebox(2,2){$\text{and}$}}

\put(11,0){\makebox(6,2){$P_{\mathsf{Sp}}(42346)=$}}

\put(17,1){\line(0,1){1}}
\put(18,0){\line(0,1){2}}
\put(19,0){\line(0,1){2}}
\put(20,1){\line(0,1){1}}
\put(21,1){\line(0,1){1}}
\put(18,0){\line(1,0){1}}
\put(17,1){\line(1,0){4}}
\put(17,2){\line(1,0){4}}

\put(17,1){\makebox(1,1){$2$}}
\put(18,0){\makebox(1,1){$4$}}
\put(18,1){\makebox(1,1){$3$}}
\put(19,1){\makebox(1,1){$4$}}
\put(20,1){\makebox(1,1){$6$}}
\end{picture}.
\end{center}
Nevertheless, we have

\begin{thm} \label{thm:CK}
If $w$ and $w^{\prime}$ are FPF-involution words such that $w\stackrel{\mathsf{CK}}{\sim}w^{\prime}$, then $P_{\mathsf{Sp}}(w)=P_{\mathsf{Sp}}(w^{\prime})$.
\end{thm}

This is one of our first main results.
The proof of Theorem~\ref{thm:CK} is relegated to Appendix A.
Note that if $w\in \hat{\mathcal{R}}_{\mathsf{FPF}}(z)$, then a word $w^{\prime}$ such that $w\stackrel{\mathsf{CK}}{\sim}w^{\prime}$ is also an FPF-involution word; $w^{\prime}\in \hat{\mathcal{R}}_{\mathsf{FPF}}(z)$.

The FPF-involution Coxeter-Knutk insertion is reversible.
The letter $u_{k}$ in $w=u_{1}u_{2}\cdots u_{l}\in \hat{\mathcal{R}}_{\mathsf{FPF}}(z)$ is extracted from $(P^{(k)},Q^{(k)})$ by the following algorithm:
Let $(i_{k},j_{k})$ be the position of the box containing the largest entry $x_{k}$ of $Q^{(k)}$.
Let $y_{k}$ be the entry at the position $(i_{k},j_{k})$ of $P^{(k)}$.
Remove the box of $Q^{(k)}$ at the position $(i_{k},j_{k})$.
Let $y_{k}$ be the entry of the box of $P^{(k)}$ at the position $(i_{k},j_{k})$.
Remove that box of $P^{(k)}$ and reverse insert $y_{k}$ into the row above if $x_{k}$ in $Q^{(k)}$ is unprimed and into the column to the left if $x_{k}$ is primed.

The rules for the reverse insertion of $y$ into a row or column of $P^{(k)}$ are as follows:

Let $x$ be the largest entry of $L$ with $x<y$.

\begin{enumerate}
\item If $L$ is a column and $x$ is the last entry of $L$ and $x\not\equiv y \pmod{2}$, then leave $L$ unchanged and insert $x-2$ to the row above.

\item If $x+1=y$, then leave $L$ unchanged and insert $x$ the row above if $L$ is a row or to the next column to the right if $L$ is a column.

\item In all other cases, replace $x$ by $y$ in $L$ and insert $x$ to the next column to the left if $L$ is a column or to the row above if $L$ is a row or $x$ was on the main diagonal.
\end{enumerate}

If we are in the first row of $P^{(k)}$ and the last step was a reverse row insertion or we are in the first column and the last step was a reverse column insertion, then $u_{k}$ is the letter bumped out.

\begin{thm}[\cite{Mar}]
Let $z\in \mathfrak{F}_{\infty}$. 
Then, the map $\mathrm{H}_{\mathsf{Sp}}$ gives a bijection between $\hat{\mathcal{R}}_{\mathsf{FPF}}(z)$ and the set of pairs $(P,Q)$, where $P$ is an increasing shifted tableau with $\mathfrak{row}(P)\in \hat{\mathcal{R}}_{\mathsf{FPF}}(z)$ and $Q$ is a standard primed tableau with the same shape as $P$. 
\end{thm}

Given $w\in \hat{\mathcal{R}}_{\mathsf{FPF}}(z)$, an increasing factorization of $w$ is a factorization $w^{1}w^{2}\cdots w^{m}$ such that $w=w^{1}w^{2}\cdots w^{m}$, which is obtained by disregarding the grouping into blocks 
with $\left\vert w \right\vert =\left\vert w^{1} \right\vert + \cdots +\left\vert w^{m} \right\vert $ and each factor $w^{i}$ is strictly increasing.
In the sequel, we consider $w^{1}w^{2}\cdots w^{m}$ as the increasing factorization or the word $w=w^{1}w^{2}\cdots w^{m}$ interchangeably.
We denote by $\mathrm{RF}_{\mathsf{FPF}}^{m}(z)$ the set of all increasing factorizations with $m$ blocks of all FPF-involution words $\hat{\mathcal{R}}_{\mathsf{FPF}}(z)$.
We say that two factorizations $w^{1}w^{2}\cdots w^{m}$ and $\Tilde{w}^{1}\Tilde{w}^{2}\cdots \Tilde{w}^{m}$, both of which are elements of $\mathrm{RF}_{\mathsf{FPF}}^{m}(z)$, are equivalent if two words $w=w^{1}w^{2}\cdots w^{m}$ and $\Tilde{w}=\Tilde{w}^{1}\Tilde{w}^{2}\cdots \Tilde{w}^{m}$ are equivalent; $w\stackrel{\mathsf{Sp}}{\sim}w^{\prime}$.

For an increasing factorization $w^{1}w^{2}\cdots w^{m}$ of $w\in \hat{\mathcal{R}}_{\mathsf{FPF}}(z)$, we construct $Q_{\mathsf{Sp}}(w^{1}w^{2}\cdots w^{m})$ from $Q_{\mathsf{Sp}}(w)$ by the following rule: 
Let $x$ be the entry in $Q_{\mathsf{Sp}}(w)$, which appears when a letter in $w^{i}$ is inserted.
Then, $x$ is replaced by $i$ if $x$ is unprimed and by $i^{\prime}$ if $x$ is primed.
We apply this procedure for all entries in $Q_{\mathsf{Sp}}(w)$.

\begin{ex}
For an increasing factorization $(6)(24)(1)$ of the FPF-involution word $6241$ in Example~\ref{ex:insertion}, we have
\setlength{\unitlength}{12pt}

\begin{center}
\begin{picture}(10,2)
\put(0,0){\makebox(7,2){$Q_{\mathsf{Sp}}((6)(24)(1))=$}}

\put(7,1){\line(0,1){1}}
\put(8,0){\line(0,1){2}}
\put(9,0){\line(0,1){2}}
\put(10,1){\line(0,1){1}}
\put(8,0){\line(1,0){1}}
\put(7,1){\line(1,0){3}}
\put(7,2){\line(1,0){3}}

\put(7,1){\makebox(1,1){$1$}}
\put(8,0){\makebox(1,1){$2$}}
\put(8,1){\makebox(1,1){$2^{\prime}$}}
\put(9,1){\makebox(1,1){$3^{\prime}$}}
\end{picture}.
\end{center}

\end{ex}

The algorithm of constructing the pair of tableaux $(P_{\mathsf{Sp}}(w),Q_{\mathsf{Sp}}(w^{1}w^{2}\cdots w^{m}))$ from the increasing factorization $w^{1}w^{2}\cdots w^{m}$ of $w\in \hat{\mathcal{R}}_{\mathsf{FPF}}(z)$ is called the \emph{semistandard FPF-involution Coxeter-Knuth insertion} and denoted by the map $\mathrm{H}_{\mathsf{Sp}}^{\prime}:w^{1}w^{2}\cdots w^{m} \mapsto (P_{\mathsf{Sp}}(w),Q_{\mathsf{Sp}}(w^{1}w^{2}\cdots w^{m}))$.
Again, tableaux $P_{\mathsf{Sp}}(w)$ and $Q_{\mathsf{Sp}}(w^{1}w^{2}\cdots w^{m})$ are referred to as the insertion tableau and the recording tableau, respectively.

The semistandard FPF-involution Coxeter-Knuth insertion is also reversible.
To reverse the insertion, we first standardize the recording tableau and apply the reverse FPF-involution Coxeter-Knuth insertion.
The rest procedure is obvious.
The \emph{standardization} of a primed tableau $T$, denoted by $\mathfrak{st}(T)$, is given by the following procedure: 
We first replace all $1$s appearing $T$, read from left to right, by $1,2,\ldots,i$, where $i$ is the number of $1$s in $T$.
Then, replace all $2^{\prime}$s appearing $T$, read from top to bottom, by the primed numbers $(i+1)^{\prime},(i+2)^{\prime},\ldots,(i+j_{1})^{\prime}$, where $j_{1}$ is the number of $2^{\prime}$s in $T$.
Then replace all $2$s appearing $T$, read from left to right, by the unprimed numbers $(i+j_{1}+1),(i+j_{1}+2),\ldots,(i+j_{1}+j_{2})$, where $j_{2}$ is the number of $2$s in $T$, 
and so on.

\begin{ex} \label{ex:standardization}
Let $\lambda =(5,3,1)$.

\setlength{\unitlength}{12pt}

\begin{center}
\begin{picture}(16,3)
\put(0,1){\makebox(2,1){$T=$}}
\put(2,2){\line(0,1){1}}
\put(3,1){\line(0,1){2}}
\put(4,0){\line(0,1){3}}
\put(5,0){\line(0,1){3}}
\put(6,1){\line(0,1){2}}
\put(7,2){\line(0,1){1}}
\put(4,0){\line(1,0){1}}
\put(3,1){\line(1,0){3}}
\put(2,2){\line(1,0){5}}
\put(2,3){\line(1,0){5}}

\put(4,0){\makebox(1,1){$4$}}
\put(3,1){\makebox(1,1){$2$}}
\put(4,1){\makebox(1,1){$3^{\prime}$}}
\put(5,1){\makebox(1,1){$4$}}
\put(2,2){\makebox(1,1){$1$}}
\put(3,2){\makebox(1,1){$1$}}
\put(4,2){\makebox(1,1){$2^{\prime}$}}
\put(5,2){\makebox(1,1){$3^{\prime}$}}
\put(6,2){\makebox(1,1){$4^{\prime}$}}

\put(8,1){\makebox(3,1){$\mathfrak{st}(T)=$}}
\put(11,2){\line(0,1){1}}
\put(12,1){\line(0,1){2}}
\put(13,0){\line(0,1){3}}
\put(14,0){\line(0,1){3}}
\put(15,1){\line(0,1){2}}
\put(16,2){\line(0,1){1}}
\put(13,0){\line(1,0){1}}
\put(12,1){\line(1,0){3}}
\put(11,2){\line(1,0){5}}
\put(11,3){\line(1,0){5}}

\put(13,0){\makebox(1,1){$8$}}
\put(12,1){\makebox(1,1){$4$}}
\put(13,1){\makebox(1,1){$6^{\prime}$}}
\put(14,1){\makebox(1,1){$9$}}
\put(11,2){\makebox(1,1){$1$}}
\put(12,2){\makebox(1,1){$2$}}
\put(13,2){\makebox(1,1){$3^{\prime}$}}
\put(14,2){\makebox(1,1){$5^{\prime}$}}
\put(15,2){\makebox(1,1){$7^{\prime}$}}

\end{picture}.
\end{center}

\end{ex}

\begin{thm}[\cite{Mar}] \label{thm:bijection2}
Let $z\in \mathfrak{F}_{\infty}$. 
Then, the map $\mathrm{H}_{\mathsf{Sp}}^{\prime}$ gives a bijection between $\mathrm{RF}_{\mathsf{FPF}}^{m}(z)$ and the set of pairs $(P,Q)$, where $P$ is an increasing shifted tableau with $\mathfrak{row}(P)\in \hat{\mathcal{R}}_{\mathsf{FPF}}(z)$ and $Q$ is a primed tableau with the same shape as $P$. 
\end{thm}

\section{Queer supercrystal structure for increasing factorizations of FPF-involution words} \label{sec:factorization}

It is established that $\mathrm{PT}_{m}(\lambda)$ is a $\mathfrak{q}(m)$-crystal~\cite{AO1,AO2,Hir}.
The $\mathfrak{gl}(m)$-crystal structure of $\mathrm{PT}_{m}(\lambda)$ together with (even) Kashiwara operators $\Tilde{e}_{i}^{P}$ and $\Tilde{f}_{i}^{P}$ ($i=1,\ldots,m-1$) is given by Hawkes, Paramonov, and Schilling~\cite{HPS} and by Assaf and Oguz~\cite{AO1,AO2}.
The actual algorithm is quite involved.
For details we refer the reader to \cite{HPS} and \cite{AO1,AO2}.
The odd Kashiwara operators $\Tilde{f}_{\Bar{1}}^{P}$ and $\Tilde{e}_{\Bar{1}}^{P}$ on $\mathrm{PT}_{m}(\lambda)$ are given by the following two lemmas~\cite{AO1,AO2,Hir}.

\begin{lem} \label{lem:eP}
For $T\in\mathrm{PT}_{m}(\lambda)$, the odd Kashiwara operator $\Tilde{e}_{\Bar{1}}^{P}$ on $T$ is given by the following rule: 

If $T_{(1,1)}=2$, change $T_{(1,1)}$ to $1$.
If $T_{(1,i)}=2^{\prime}$ ($i\geq 2$), change $T_{(1,i)}$ to $1$.
Otherwise, $\Tilde{e}_{\Bar{1}}^{P}T=\boldsymbol{0}$.
\end{lem}

\begin{lem} \label{lem:fP}
For $T\in\mathrm{PT}_{m}(\lambda)$, the odd Kashiwara operator $\Tilde{f}_{\Bar{1}}^{P}$ on $T$ is given by the following rule:

If the first row of $T$ does not contain letters $1$, then $\Tilde{f}_{\Bar{1}}^{P}T=\boldsymbol{0}$.
Let $(1,i)$ be the position of the rightmost $1$ in the first row.
\begin{enumerate}
\item $i=1$.
If $T_{(1,2)}=2^{\prime}$, then $\Tilde{f}_{\Bar{1}}^{P}T=\boldsymbol{0}$.
Otherwise, change $T_{(1,1)}=1$ to $2$.
\item $i\geq 2$.
If $T_{(1,i+1)}=2^{\prime}$, then $\Tilde{f}_{\Bar{1}}^{P}T=\boldsymbol{0}$.
Otherwise, change $T_{(1,i)}=1$ to $2^{\prime}$.
\end{enumerate}
\end{lem}

Since $\mathrm{H}_{\mathsf{Sp}}^{\prime}$ is a bijection (Theorem~\ref{thm:bijection2}), we can translate the $\mathfrak{gl}(m)$-crystal structure of $\mathrm{PT}_{m}(\lambda)$ to that of $\mathrm{RF}^{m}_{\mathsf{FPF}}(z)$.
That is, Kashiwara operators $\Tilde{e}_{i}^{F}$ and $\Tilde{f}_{i}^{F}$ on $\mathrm{RF}^{m}_{\mathsf{FPF}}(z)$ are given by

\begin{equation} \label{eq:crystaliso}
\Tilde{e}_{i} =(\mathrm{H}_{\mathsf{Sp}}^{\prime})^{-1}\circ (\mathrm{id},\Tilde{e}_{i}^{P})\circ \mathrm{H}_{\mathsf{Sp}}^{\prime}, \quad
\Tilde{f}_{i} =(\mathrm{H}_{\mathsf{Sp}}^{\prime})^{-1}\circ (\mathrm{id},\Tilde{f}_{i}^{P})\circ \mathrm{H}_{\mathsf{Sp}}^{\prime},
\end{equation}
for a fixed insertion tableau $P_{\mathsf{Sp}}(w)$ ($i=1,\ldots,m-1,\Bar{1}$).

First, we construct odd Kashiwara operators.
Let $w^{1}w^{2}\cdots w^{m}\in \mathrm{RF}^{m}_{\mathsf{FPF}}(z)$.
Let us denote by $\mathrm{cont}(w^{i})$ the set of letters appearing in $w^{i}$. 

\begin{lem} \label{lem:f0F1}
Let $w^{1}w^{2}\cdots w^{m}\in \mathrm{RF}_{\mathsf{FPF}}^{m}(z)$ be an increasing factorization of $w\in \hat{\mathcal{R}}_{\mathsf{FPF}}(z)$ for $z\in \mathfrak{F}_{\infty}$.
The action of the odd Kashiwara operator $\Tilde{f}_{\Bar{1}}^{F}$ on $w^{1}w^{2}\cdots w^{m}$ is given by the following rule:

$\Tilde{f}_{\Bar{1}}^{F}$ always changes the first two factors if $\Tilde{f}_{\Bar{1}}^{F}(w^{1}w^{2}\cdots w^{m})\neq \boldsymbol{0}$. 
Suppose that $\left\vert w^{2}\right\vert\neq 0$.

\begin{enumerate}
\item $\left\vert w^{1} \right\vert =0$.

$\Tilde{f}_{\Bar{1}}^{F}(w^{1}w^{2}\cdots w^{m})= \boldsymbol{0}$.

\item $\left\vert w^{1} \right\vert =1$.

If $w^{1}=u_{1}<\min (\mathrm{cont}(w^{2}))$, then $\Tilde{f}_{\Bar{1}}^{F}(w^{1}w^{2}\cdots w^{m})=\Tilde{w}^{1}\Tilde{w}^{2}\cdots w^{m}$, where $\Tilde{w}^{1}=()$ and $\mathrm{cont}(\Tilde{w}^{2})=\mathrm{cont}(w^{2})\cup \{u_{1}\}$.
Otherwise, $\Tilde{f}_{\Bar{1}}^{F}(w^{1}w^{2}\cdots w^{m})= \boldsymbol{0}$.

\item $\left\vert w^{1} \right\vert \geq 2$.

Let us write $w^{1}=u_{1}u_{2}\cdots$.
If $u_{2}>u_{1}+1$ and $u_{1}<\min (\mathrm{cont}(w^{2}))$, then $\Tilde{f}_{\Bar{1}}^{F}(w^{1}w^{2}\cdots w^{m})=\Tilde{w}^{1}\Tilde{w}^{2}\cdots w^{m}$, where $\mathrm{cont}(\Tilde{w}^{1})=\mathrm{cont}(w^{1})\backslash \{u_{1}\}$ and $\mathrm{cont}(\Tilde{w}^{2})=\mathrm{cont}(w^{2})\cup \{u_{1}\}$.
If $u_{2}=u_{1}+1$ and $u_{1}<\min (\mathrm{cont}(w^{2}))$, then $\Tilde{f}_{\Bar{1}}^{F}(w^{1}w^{2}\cdots w^{m})=\Tilde{w}^{1}\Tilde{w}^{2}\cdots w^{m}$, where $\mathrm{cont}(\Tilde{w}^{1})=\mathrm{cont}(w^{1})\backslash \{u_{1}+1\}$ and $\mathrm{cont}(\Tilde{w}^{2})=\mathrm{cont}(w^{2})\cup \{u_{1}-1\}$.
Otherwise, $\Tilde{f}_{\Bar{1}}^{F}(w^{1}w^{2}\cdots w^{m})= \boldsymbol{0}$.
\end{enumerate}

If $\left\vert w^{2} \right\vert =0$, then just drop the condition $u_{1}< \min (\mathrm{cont}(w^{2}))$. 
\end{lem}

\begin{lem} \label{lem:e0F1}
Let $w^{1}w^{2}\cdots w^{m}\in \mathrm{RF}_{\mathsf{FPF}}^{m}(z)$ be an increasing factorization of $w\in \hat{\mathcal{R}}_{\mathsf{FPF}}(z)$ for $z\in \mathfrak{F}_{\infty}$.
The action of the odd Kashiwara operator $\Tilde{e}_{\Bar{1}}^{F}$ on $w^{1}w^{2}\cdots w^{m}$ is given by the following rule:

$\Tilde{e}_{\Bar{1}}^{F}$ always changes the first two factors if $\Tilde{e}_{\Bar{1}}^{F}(w^{1}w^{2}\cdots w^{m})\neq \boldsymbol{0}$. 

\begin{enumerate}

\item $\left\vert w^{2} \right\vert =0$.

$\Tilde{e}_{\Bar{1}}^{F}(w^{1}w^{2}\cdots w^{m})= \boldsymbol{0}$.

\item $\left\vert w^{1} \right\vert =0$ and $\left\vert w^{2} \right\vert \neq 0$.

$\Tilde{e}_{\Bar{1}}^{F}(w^{1}w^{2}\cdots w^{m})=\Tilde{w}^{1}\Tilde{w}^{2}\cdots w^{m}$, where $\Tilde{w}^{1}=v_{1}$ and $\mathrm{cont}(\Tilde{w}^{2})=\mathrm{cont}(w^{2})\backslash \{v_{1}\}$ with $v_{1}$ being the first entry of $w_{2}$.

\item $\left\vert w^{1} \right\vert \neq 0$ and $\left\vert w^{2} \right\vert \neq 0$.

Let us write $w^{1}=u_{1}u_{2}\cdots$ and $w^{2}=v_{1}\cdots$.
If $u_{1}>v_{1}+1$, then $\Tilde{e}_{\Bar{1}}^{F}(w^{1}w^{2}\cdots w^{m})=\Tilde{w}^{1}\Tilde{w}^{2}\cdots w^{m}$, where $\mathrm{cont}(\Tilde{w}^{1})=\mathrm{cont}(w^{1})\cup \{v_{1}\}$ and $\mathrm{cont}(\Tilde{w}^{2})=\mathrm{cont}(w^{2})\backslash \{v_{1}\}$.
If $u_{1}=v_{1}+1$, then $\Tilde{e}_{\Bar{1}}^{F}(w^{1}w^{2}\cdots w^{m})=\Tilde{w}^{1}\Tilde{w}^{2}\cdots w^{m}$, where $\mathrm{cont}(\Tilde{w}^{1})=\mathrm{cont}(w^{1})\cup \{u_{1}+1\}$ and $\mathrm{cont}(\Tilde{w}^{2})=\mathrm{cont}(w^{2})\backslash \{u_{1}-1\}$.
Otherwise, $\Tilde{e}_{\Bar{1}}^{F}(w^{1}w^{2}\cdots w^{m})= \boldsymbol{0}$.
 
\end{enumerate}
\end{lem}

We give the proof of Lemma~\ref{lem:f0F1} only.
The proof of Lemma~\ref{lem:e0F1} is similar.

\begin{proof}[Proof of Lemma~\ref{lem:f0F1}]

For cases (1) and (2), the statement is obvious.
Let us assume that $\left\vert w^{1} \right\vert \geq 2$ and $\left\vert w^{2}\right\vert\neq 0$ and write $w^{1}=u_{1}u_{2}\cdots u_{l}$ and $w^{2}=v_{1}\cdots$.
Let $Q$ be the recording tableau for $w^{1}w^{2}\cdots w^{m}$.

If $u_{1}>v_{1}$, then the configuration of the recording tableau up to the insertion of the letter $v_{1}$ has the configuration,
\setlength{\unitlength}{12pt}

\begin{center}
\begin{picture}(5,2)

\put(0,0){\line(0,1){1}}
\put(1,0){\line(0,1){1}}
\put(3,0){\line(0,1){1}}
\put(4,0){\line(0,1){1}}
\put(5,0){\line(0,1){1}}
\put(0,0){\line(1,0){5}}
\put(0,1){\line(1,0){5}}

\put(0,0){\makebox(1,1){$1$}}
\put(1,0){\makebox(2,1){$\cdots$}}
\put(3,0){\makebox(1,1){$1$}}
\put(4,0){\makebox(1,1){$2^{\prime}$}}

\put(3,1){\makebox(1,1){\small $l$}}
\end{picture}.
\end{center}
This portion does not change under the subsequent insertions.
By Lemma~\ref{lem:fP}, $\Tilde{f}_{\Bar{1}}^{P}(Q)=\boldsymbol{0}$ so that $\Tilde{f}_{\Bar{1}}^{F}(w^{1}w^{2}\cdots w^{m})=\boldsymbol{0}$.
It is easy to see that $u_{1}\neq v_{1}$.
Indeed, if $u_{2}>u_{1}+1$, then $u_{1}u_{2}\cdots u_{l}v_{1} \stackrel{\mathsf{Sp}}{\sim}u_{1}v_{1}u_{2}\cdots u_{l}$, which is not an FPF-involution word and if $u_{2}=u_{1}+1$, then $u_{1}u_{2}\cdots u_{l}v_{1}\stackrel{\mathsf{Sp}}{\sim}u_{1}(u_{1}+1)v_{1}\cdots u_{l}\stackrel{\mathsf{Sp}}{\sim}(u_{1}+1)u_{1}(u_{1}+1)\cdots u_{l}$, which is also not an FPF-involution word because $u_{1}+1$ is an odd letter. 
If $u_{1}<v_{1}$, then the first row of $Q$ has one of the following two,

\setlength{\unitlength}{12pt}

\begin{center}
\begin{picture}(16,2)

\put(0,0){\line(0,1){1}}
\put(1,0){\line(0,1){1}}
\put(4,0){\line(0,1){1}}
\put(5,0){\line(0,1){1}}
\put(0,0){\line(1,0){5}}
\put(0,1){\line(1,0){5}}

\put(0,0){\makebox(1,1){$1$}}
\put(1,0){\makebox(3,1){$\cdots$}}
\put(4,0){\makebox(1,1){$1$}}

\put(4,1){\makebox(1,1){\small $l$}}

\put(8,0){\line(0,1){1}}
\put(9,0){\line(0,1){1}}
\put(12,0){\line(0,1){1}}
\put(13,0){\line(0,1){1}}
\put(14,0){\line(0,1){1}}
\put(8,0){\line(1,0){8}}
\put(8,1){\line(1,0){8}}

\put(8,0){\makebox(1,1){$1$}}
\put(9,0){\makebox(3,1){$\cdots$}}
\put(12,0){\makebox(1,1){$1$}}
\put(13,0){\makebox(1,1){$x$}}
\put(14,0){\makebox(2,1){$\cdots$}}

\put(12,1){\makebox(1,1){\small $l$}}
\end{picture}
\end{center}
with $x\geq 2$.
Note that column insertions never happen in the insertion of $w^{2}$ to the increasing tableau,
\setlength{\unitlength}{12pt}

\begin{center}
\begin{picture}(5,1)

\put(0,0){\line(0,1){1}}
\put(1,0){\line(0,1){1}}
\put(4,0){\line(0,1){1}}
\put(5,0){\line(0,1){1}}
\put(0,0){\line(1,0){5}}
\put(0,1){\line(1,0){5}}

\put(0,0){\makebox(1,1){$u_{1}$}}
\put(1,0){\makebox(3,1){$\cdots$}}
\put(4,0){\makebox(1,1){$u_{l}$}}

\end{picture}.
\end{center}
By Lemma~\ref{lem:fP}, the configuration of the first row of $\Tilde{f}_{\Bar{1}}^{P}(Q)$ is either 
\setlength{\unitlength}{12pt}

\begin{center}
\begin{picture}(16,2)

\put(0,0){\line(0,1){1}}
\put(1,0){\line(0,1){1}}
\put(3,0){\line(0,1){1}}
\put(4,0){\line(0,1){1}}
\put(5,0){\line(0,1){1}}
\put(0,0){\line(1,0){5}}
\put(0,1){\line(1,0){5}}

\put(0,0){\makebox(1,1){$1$}}
\put(1,0){\makebox(2,1){$\cdots$}}
\put(3,0){\makebox(1,1){$1$}}
\put(4,0){\makebox(1,1){$2^{\prime}$}}

\put(4,1){\makebox(1,1){\small $l$}}

\put(6,0){\makebox(1,1){$\text{or}$}}

\put(8,0){\line(0,1){1}}
\put(9,0){\line(0,1){1}}
\put(11,0){\line(0,1){1}}
\put(12,0){\line(0,1){1}}
\put(13,0){\line(0,1){1}}
\put(14,0){\line(0,1){1}}
\put(8,0){\line(1,0){8}}
\put(8,1){\line(1,0){8}}

\put(8,0){\makebox(1,1){$1$}}
\put(9,0){\makebox(2,1){$\cdots$}}
\put(11,0){\makebox(1,1){$1$}}
\put(12,0){\makebox(1,1){$2^{\prime}$}}
\put(13,0){\makebox(1,1){$x$}}
\put(14,0){\makebox(2,1){$\cdots$}}

\put(12,1){\makebox(1,1){\small $l$}}
\end{picture}.
\end{center}
If $u_{2}>u_{1}+1$, then the increasing factorization that is equivalent to $w^{1}w^{2}$ and gives the above configuration is $(u_{2}\cdots u_{l})(u_{1}v_{1}\cdots)$.
If $u_{2}=u_{1}+1$, then the increasing factorization which is equivalent to $w^{1}w^{2}$ and gives the above configuration is $(u_{1}u_{3}\cdots u_{l})((u_{1}-1)v_{1}\cdots)$.
This is because $u_{1}u_{3}\cdots u_{l}(u_{1}-1)v_{1} \stackrel{\mathsf{Sp}}{\sim} u_{1}(u_{1}-1)\cdots u_{l}v_{1} \stackrel{\mathsf{Sp}}{\sim} u_{1}(u_{1}+1)\cdots u_{l}v_{1}$.

We omit the proof for the case when $\left\vert w^{2}\right\vert =0$; it is much simpler.
This completes the proof. 

\end{proof}

By construction, two vertices connected by $\Tilde{f}_{\Bar{1}}^{F}$ or $\Tilde{e}_{\Bar{1}}^{F}$ in the same connected component of $\mathfrak{q}(m)$-crystal $\mathrm{RF}^{m}_{\mathsf{FPF}}(z)$ have the same insertion tableau $P_{\mathsf{Sp}}(w)$ and $\Tilde{f}_{\Bar{1}}^{F}$ or $\Tilde{e}_{\Bar{1}}^{F}$ are independent of the choice of $P_{\mathsf{Sp}}(w)$.
Note that the first letter of $\Tilde{f}_{\Bar{1}}^{F}(w^{1}w^{2}\cdots w^{m})$ is always an even letter if $\Tilde{f}_{\Bar{1}}^{F}(w^{1}w^{2}\cdots w^{m})\neq \boldsymbol{0}$ for $w^{1}w^{2}\cdots w^{m} \in\mathrm{RF}^{m}_{\mathsf{FPF}}(z)$.

The set $\mathrm{RF}^{m}_{\mathsf{FPF}}(z)$ is a $\mathfrak{gl}(m)$-crystal with Kashiwara operators $\Tilde{e}_{i}^{F}$ and $\Tilde{f}_{i}^{F}$ given by Morse and Schilling~\cite{BS,MS}.
Before giving the verification of this claim, we restate Morse-Schilling's construction of their Kashiwara operators with appropriate alterations in our setting.

The crystal operators $\Tilde{e}^{F}_{i}$ and $\Tilde{f}^{F}_{i}$ only act on the block $w^{i}w^{i+1}$ $(i=1,\ldots,m-1)$.
We define the pairing of $w^{i}$ and $w^{i+1}$ as follows: 
Pair the largest $b\in \mathrm{cont}(w^{i+1})$ with the smallest $a>b$ in $\mathrm{cont}(w^{i})$ and if no such $a$ exists, then $b$ is unpaired.
We proceed in decreasing order on letters in $w^{i+1}$, ignoring previously paired letters in $w^{i}$.
Define
\[
R_{i}=\left\{b\in \mathrm{cont}(w^{i+1}) \relmiddle| b \text{ is unpaired in the } w^{i}w^{i+1}\text{-pairing} \right\}
\] 
and
\[
L_{i}=\left\{b\in \mathrm{cont}(w^{i}) \relmiddle| b \text{ is unpaired in the } w^{i}w^{i+1}\text{-pairing} \right\}.
\] 
Then $\Tilde{f}^{F}_{i}(w^{1}\cdots w^{m})$ is defined by replacing blocks $w^{i}w^{i+1}$ by $\Tilde{w}^{i}\Tilde{w}^{i+1}$ such that $\mathrm{cont}(\Tilde{w}^{i})=\mathrm{cont}(w^{i})\backslash \{c\}$ and $\mathrm{cont}(\Tilde{w}^{i+1})=\mathrm{cont}(w^{i+1})\cup \{c+s\}$ for $c=\max L_{i}$ and $s=\min \left\{j\geq 0 \relmiddle| c+j+1 \notin \mathrm{cont}(w^{i}) \right\}$.
If $L_{i}=\emptyset$, then $\Tilde{f}^{F}_{i}(w^{1}\cdots w^{m})=\boldsymbol{0}$.
Similarly, $\Tilde{e}^{F}_{i}(w^{1}\cdots w^{m})$ is defined by replacing blocks $w^{i}w^{i+1}$ by $\Tilde{w}^{i}\Tilde{w}^{i+1}$ such that $\mathrm{cont}(\Tilde{w}^{i})=\mathrm{cont}(w^{i})\cup \{c-t\}$ and $\mathrm{cont}(\Tilde{w}^{i+1})=\mathrm{cont}(w^{i+1})\backslash \{c\}$ for $c=\min R_{i}$ and $t=\min \left\{j\geq 0 \relmiddle| c-j-1 \notin \mathrm{cont}(w^{i+1}) \right\}$.
If $R_{i}=\emptyset$, then $\Tilde{e}^{F}_{i}(w^{1}\cdots w^{m})=\boldsymbol{0}$.

\begin{ex}
This example is taken from the top part of Fig.~\ref{fig:factorization}.
\[
w^{1}w^{2}=(234)(3) \stackrel{1}{\longrightarrow}\Tilde{w}^{1}\Tilde{w}^{2}=(24)(34).
\]
In this example, $L_{1}=\{2,3\}$, $c=\max L_{1}=3$, and $s=1$ so that $\mathrm{cont}(\Tilde{w}^{2})=\{3,4\}$.
\end{ex}

The weight of $w^{1}\cdots w^{m}$ is defined to be 
\[
\mathrm{wt}(w^{1}\cdots w^{m})=\sum_{i=1}^{m}\left\vert w^{i}\right\vert \epsilon_{i}.
\]

\begin{lem}
Let $w^{1}w^{2}\cdots w^{m}\in \mathrm{RF}^{m}_{\mathsf{FPF}}(z)$ for $z\in \mathfrak{F}_{\infty}$.
If $\Tilde{f}_{i}^{F}(w^{1}w^{2}\cdots w^{m})\neq \boldsymbol{0}$ for some $i\in\{1,\ldots,m-1\}$, then $\Tilde{f}_{i}^{F}(w^{1}w^{2}\cdots w^{m})\in \mathrm{RF}^{m}_{\mathsf{FPF}}(z)$.
\end{lem}

\begin{proof}
Let $\Tilde{f}_{i}^{F}(w^{1}w^{2}\cdots w^{m})$ be an increasing factorization of an FPF-involution word $\Tilde{w}$.
Then, $w=w^{1}w^{2}\cdots w^{m}\stackrel{\mathsf{CK}}{\sim}\Tilde{w}$~\cite{BS,MS} so that $\Tilde{w}\in \hat{\mathcal{R}}_{\mathsf{FPF}}(z)$ and $\Tilde{f}_{i}^{F}(w^{1}w^{2}\cdots w^{m})\in \mathrm{RF}^{m}_{\mathsf{FPF}}(z)$.
\end{proof}

Thus, our claim has been verified.

\begin{lem} \label{lem:const1}
Let $u^{1}u^{2}\cdots u^{m}$ and $v^{1}v^{2}\cdots v^{m}$ be two vertices in the same connected component of the $\mathfrak{gl}(m)$-crystal $\mathrm{RF}^{m}_{\mathsf{FPF}}(z)$ for $z\in \mathfrak{F}_{\infty}$.
Then, two words $u=u^{1}u^{2}\cdots u^{m}$ and $v=v^{1}v^{2}\cdots v^{m}$ have the same insertion tableau; $P_{\mathsf{Sp}}(u)=P_{\mathsf{Sp}}(v)$.
\end{lem}

\begin{proof}
Two words $u$ and $v$ are Coxeter-Knuth equivalent; $u\stackrel{\mathsf{CK}}{\sim}v$~\cite{BS,MS}.
Hence, by Theorem~\ref{thm:CK}, the assertion of Lemma~\ref{lem:const1} follows.
\end{proof}

This property is compatible with Eq~\eqref{eq:crystaliso}, i.e, the $\mathfrak{gl}(m)$-crystal structure of $\mathrm{RF}_{\mathsf{FPF}}^{m}(z)$ is in one-to-one correspondence with that of $\mathrm{PT}_{m}(\lambda)$, and $\Tilde{e}_{i}^{F}$ and $\Tilde{f}_{i}^{F}$ ($i=1,\ldots,m-1,\Bar{1}$) satisfy conditions in Definition~\ref{df:queer}.
Thus, we can adopt Kashiwara operators above as the proper even Kashiwara operators on $\mathrm{RF}^{m}_{\mathsf{FPF}}(z)$.

\begin{thm} \label{thm:q-FPF}
Let $z\in \mathfrak{F}_{\infty}$.
Then, the set $\mathrm{RF}_{\mathsf{FPF}}^{m}(z)$ admits a $\mathfrak{q}(m)$-crystal structure.
The even Kashiwara operators are $\Tilde{e}_{i}^{F}$ and $\Tilde{f}_{i}^{F}$ ($i=1,2,\ldots,m-1$), whereas the odd Kashiwara operators are given in Lemma~\ref{lem:f0F1} and \ref{lem:e0F1}.
\end{thm}

This is our second main result.

\begin{con}
The even Kashiwara operators given by Eq~\eqref{eq:crystaliso} agree with those given by Morse and Schilling.
\end{con}

\setlength{\unitlength}{10pt}

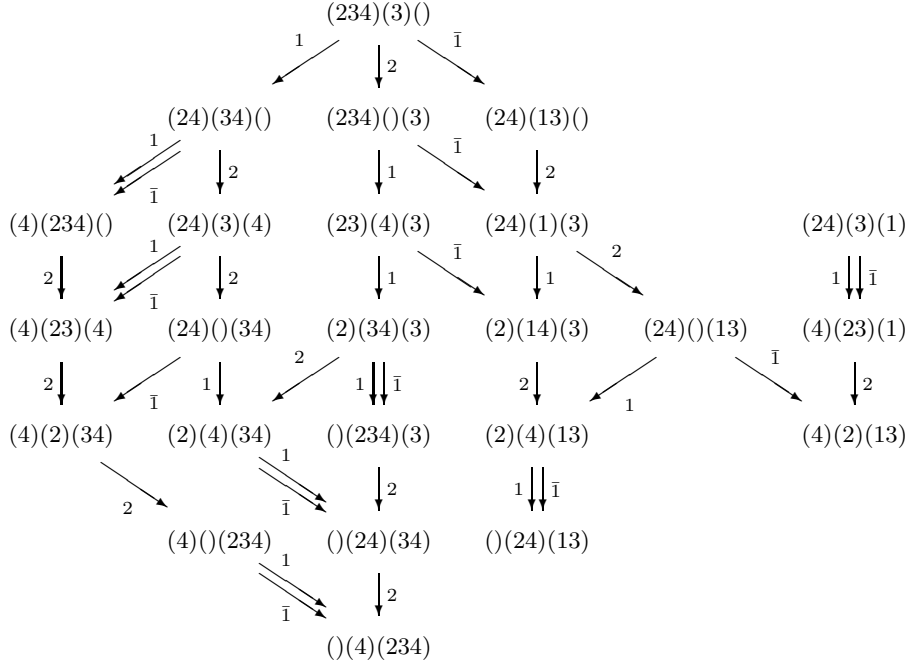
\begin{figure}
\begin{center}
\begin{picture}(33,26)

\put(1.5,11.8){\vector(0,-1){1.6}}
\put(1.5,15.8){\vector(0,-1){1.6}}

\put(7.5,11.8){\vector(0,-1){1.6}}
\put(7.5,15.8){\vector(0,-1){1.6}}
\put(7.5,19.8){\vector(0,-1){1.6}}

\put(13.5,3.8){\vector(0,-1){1.6}}
\put(13.5,7.8){\vector(0,-1){1.6}}
\put(13.3,11.8){\vector(0,-1){1.6}}
\put(13.7,11.8){\vector(0,-1){1.6}}
\put(13.5,15.8){\vector(0,-1){1.6}}
\put(13.5,19.8){\vector(0,-1){1.6}}
\put(13.5,23.8){\vector(0,-1){1.6}}

\put(19.3,7.8){\vector(0,-1){1.6}}
\put(19.7,7.8){\vector(0,-1){1.6}}
\put(19.5,11.8){\vector(0,-1){1.6}}
\put(19.5,15.8){\vector(0,-1){1.6}}
\put(19.5,19.8){\vector(0,-1){1.6}}

\put(31.5,11.8){\vector(0,-1){1.6}}
\put(31.3,15.8){\vector(0,-1){1.6}}
\put(31.7,15.8){\vector(0,-1){1.6}}

\put(3,8){\vector(3,-2){2.5}}
\put(6,12){\vector(-3,-2){2,5}}
\put(6,15.8){\vector(-3,-2){2,5}}
\put(6,16.2){\vector(-3,-2){2,5}}
\put(6,19.8){\vector(-3,-2){2,5}}
\put(6,20.2){\vector(-3,-2){2,5}}

\put(9,3.8){\vector(3,-2){2.5}}
\put(9,4.2){\vector(3,-2){2.5}}
\put(9,7.8){\vector(3,-2){2.5}}
\put(9,8.2){\vector(3,-2){2.5}}
\put(12,12){\vector(-3,-2){2,5}}
\put(12,24){\vector(-3,-2){2,5}}

\put(15,16){\vector(3,-2){2.5}}
\put(15,20){\vector(3,-2){2.5}}
\put(15,24){\vector(3,-2){2.5}}

\put(21,16){\vector(3,-2){2.5}}
\put(24,12){\vector(-3,-2){2,5}}

\put(27,12){\vector(3,-2){2.5}}

\put(0.5,10.5){\makebox(1,1){\scriptsize $2$}}
\put(0.5,14.5){\makebox(1,1){\scriptsize $2$}}

\put(6.5,10.5){\makebox(1,1){\scriptsize $1$}}
\put(7.5,14.5){\makebox(1,1){\scriptsize $2$}}
\put(7.5,18.5){\makebox(1,1){\scriptsize $2$}}

\put(13.5,2.5){\makebox(1,1){\scriptsize $2$}}
\put(13.5,6.5){\makebox(1,1){\scriptsize $2$}}
\put(12.3,10.5){\makebox(1,1){\scriptsize $1$}}
\put(13.7,10.5){\makebox(1,1){\scriptsize $\Bar{1}$}}
\put(13.5,14.5){\makebox(1,1){\scriptsize $1$}}
\put(13.5,18.5){\makebox(1,1){\scriptsize $1$}}
\put(13.5,22.5){\makebox(1,1){\scriptsize $2$}}

\put(18.3,6.5){\makebox(1,1){\scriptsize $1$}}
\put(19.7,6.5){\makebox(1,1){\scriptsize $\Bar{1}$}}
\put(18.5,10.5){\makebox(1,1){\scriptsize $2$}}
\put(19.5,14.5){\makebox(1,1){\scriptsize $1$}}
\put(19.5,18.5){\makebox(1,1){\scriptsize $2$}}

\put(31.5,10.5){\makebox(1,1){\scriptsize $2$}}
\put(30.3,14.5){\makebox(1,1){\scriptsize $1$}}
\put(31.7,14.5){\makebox(1,1){\scriptsize $\Bar{1}$}}

\put(3.5,5.8){\makebox(1,1){\scriptsize $2$}}
\put(4.5,9.8){\makebox(1,1){\scriptsize $\Bar{1}$}}
\put(4.5,13.6){\makebox(1,1){\scriptsize $\Bar{1}$}}
\put(4.5,15.7){\makebox(1,1){\scriptsize $1$}}
\put(4.5,17.6){\makebox(1,1){\scriptsize $\Bar{1}$}}
\put(4.5,19.7){\makebox(1,1){\scriptsize $1$}}

\put(9.5,1.8){\makebox(1,1){\scriptsize $\Bar{1}$}}
\put(9.5,3.8){\makebox(1,1){\scriptsize $1$}}
\put(9.5,5.8){\makebox(1,1){\scriptsize $\Bar{1}$}}
\put(9.5,7.8){\makebox(1,1){\scriptsize $1$}}
\put(10,11.5){\makebox(1,1){\scriptsize $2$}}
\put(10,23.5){\makebox(1,1){\scriptsize $1$}}

\put(16,15.5){\makebox(1,1){\scriptsize $\Bar{1}$}}
\put(16,19.5){\makebox(1,1){\scriptsize $\Bar{1}$}}
\put(16,23.5){\makebox(1,1){\scriptsize $\Bar{1}$}}

\put(22.5,9.7){\makebox(1,1){\scriptsize $1$}}
\put(22,15.5){\makebox(1,1){\scriptsize $2$}}

\put(28,11.5){\makebox(1,1){\scriptsize $\Bar{1}$}}

\put(0,8){\makebox(3,2){\small $(4)(2)(34)$}}
\put(0,12){\makebox(3,2){\small $(4)(23)(4)$}}
\put(0,16){\makebox(3,2){\small $(4)(234)()$}}

\put(6,4){\makebox(3,2){\small $(4)()(234)$}}
\put(6,8){\makebox(3,2){\small $(2)(4)(34)$}}
\put(6,12){\makebox(3,2){\small $(24)()(34)$}}
\put(6,16){\makebox(3,2){\small $(24)(3)(4)$}}
\put(6,20){\makebox(3,2){\small $(24)(34)()$}}

\put(12,0){\makebox(3,2){\small $()(4)(234)$}}
\put(12,4){\makebox(3,2){\small $()(24)(34)$}}
\put(12,8){\makebox(3,2){\small $()(234)(3)$}}
\put(12,12){\makebox(3,2){\small $(2)(34)(3)$}}
\put(12,16){\makebox(3,2){\small $(23)(4)(3)$}}
\put(12,20){\makebox(3,2){\small $(234)()(3)$}}
\put(12,24){\makebox(3,2){\small $(234)(3)()$}}

\put(18,4){\makebox(3,2){\small $()(24)(13)$}}
\put(18,8){\makebox(3,2){\small $(2)(4)(13)$}}
\put(18,12){\makebox(3,2){\small $(2)(14)(3)$}}
\put(18,16){\makebox(3,2){\small $(24)(1)(3)$}}
\put(18,20){\makebox(3,2){\small $(24)(13)()$}}

\put(24,12){\makebox(3,2){\small $(24)()(13)$}}

\put(30,8){\makebox(3,2){\small $(4)(2)(13)$}}
\put(30,12){\makebox(3,2){\small $(4)(23)(1)$}}
\put(30,16){\makebox(3,2){\small $(24)(3)(1)$}}

\end{picture}
\end{center}
\caption{An example of $\mathfrak{q}(3)$-crystal structure of increasing factorizations of FPF-involution words, $2143$, $2343$, $2413$, $2431$, $2434$, $4213$, $4231$, and $4234$.
The set of these words is $\hat{\mathcal{R}}_{\mathsf{FPF}}(546213)$.}
 \label{fig:factorization}
\end{figure}

It is not hard to check that the $\mathfrak{q}(m)$-crystal graph of $\mathrm{RF}_{\mathsf{FPF}}^{m}(z)$ satisfies the local queer axioms introduced by Assaf and  Oguz~\cite{AO1,AO2} and Gillespie, Hawkes, Poh, and Schilling~\cite{GHPS}.
We omit the details

\appendix
\section{}

In this Appendix, $T\stackrel{\mathsf{Sp}}{\leftarrow}w$ is taken to be the resulting tableau by the insertion unless stated otherwise, where $T$ is an increasing tableau and $w$ is a letter or word.

\begin{proof}[Proof of Theorem~\ref{thm:CK}]
The assertion is the direct consequence of Lemmas~\ref{lem:bac}, \ref{lem:acb}, and \ref{lem:braid}.
\end{proof}

\begin{lem} \label{lem:even}
Let $T$ be an increasing shifted tableau of shape $\lambda$ such that $\mathfrak{row}(T)$ is an FPF-involution word.
Then entries in the main diagonal of $T$ are even letters.
\end{lem}

\begin{proof}
Let $T^{\prime}_{i}$ be the portion of $T$ below its $i$th row and let $T_{(i,i)}=d_{i}$ ($i=1,2,\ldots,l(\lambda)-1$).
Since all entries of $T^{\prime}_{i}$ are strictly greater than $d_{i}$+1, $\mathfrak{row}(T^{\prime}_{i})d_{i}\stackrel{\mathsf{Sp}}{\sim}d_{i}\mathfrak{row}(T^{\prime}_{i})$ so that $d_{i}$ must be an even letter.
It is clear that $d_{l(\lambda)}$ is an even letter.
This is also obvious from the algorithm of FPF-involution Coxeter-Knuth insertion.
\end{proof}

\begin{lem} \label{lem:T1}
Let $T$ be an increasing shifted tableau and $a$ be a letter such that $\mathfrak{row}(T)a$ is an FPF-involution word.
Let $T_{(1,1)}=x$ and suppose that $a<x$.
If $a \not\equiv x \pmod{2}$, then $T_{(1,1)}=a+1$, $T_{(1,3)}\geq a+4$ and the insertion $T\stackrel{\mathsf{Sp}}{\leftarrow}a$ changes the first row of $T$ only.
If $a\equiv x$ ($\mathrm{mod}$ $2$), then the insertion $T\stackrel{\mathsf{Sp}}{\leftarrow}a$ changes the first row of $T$ only.
\end{lem}

\begin{proof}
Let $T_{1}$ be the first row of $T$ and $T^{\prime}$ be the portion of $T$ below the first row.

\begin{enumerate}
\item $a\not\equiv x \pmod{2}$.

Let us write $T_{1}$ as
\setlength{\unitlength}{12pt}
\begin{picture}(4,1)
\put(0,0){\line(0,1){1}}
\put(1,0){\line(0,1){1}}
\put(2,0){\line(0,1){1}}
\put(4,0){\line(0,1){1}}
\put(0,0){\line(1,0){4}}
\put(0,1){\line(1,0){4}}
\put(0,0){\makebox(1,1){$x$}}
\put(1,0){\makebox(1,1){$y$}}
\put(2,0){\makebox(2,1){$A$}}
\end{picture}.
The insertion $T\stackrel{\mathsf{Sp}}{\leftarrow}a$ replace $y$ by $a+2$ so that $x=a+1$.

First, let us consider the case when $y=a+2$.
Then, the entry $a+3$ exists in $T^{\prime}$.
Otherwise, all the entries in $T^{\prime}$ are greater than or equal to $a+4$ so that
\begin{align*}
\mathfrak{row}(T)a&=\mathfrak{row}(T^{\prime})T_{(1,1)}T_{(1,2)}\mathfrak{row}(A)a \\
&\stackrel{\mathsf{Sp}}{\sim}(a+1)(a+2)\mathfrak{row}(T^{\prime})\mathfrak{row}(A)a \\
&\stackrel{\mathsf{Sp}}{\sim}(a+1)a\mathfrak{row}(T^{\prime})\mathfrak{row}(A)a \\
&\stackrel{\mathsf{Sp}}{\sim}(a+1)\mathfrak{row}(T^{\prime})\mathfrak{row}(A)aa,
\end{align*}
which implies that $\mathfrak{row}(T)a$ is not an FPF-involution word and contradicts the assumption of Lemma~\ref{lem:T1}.
The entry $a+3$ exists only in the position $(2,2)$.
If $T_{(1,3)}=a+3$, then 
\begin{align*}
\mathfrak{row}(T)a&=\mathfrak{row}(T^{\prime})(a+1)(a+2)(a+3)\cdots a \\
&\stackrel{\mathsf{Sp}}{\sim}(a+1)a\mathfrak{row}(T^{\prime})(a+2)(a+2)\cdots \\
&\stackrel{\mathsf{Sp}}{\sim}(a+1)(a+2)\mathfrak{row}(T^{\prime})(a+2)(a+3)\cdots \\
&\stackrel{\mathsf{Sp}}{\sim}(a+1)\cdots (a+2)T_{(2,2)}(a+2)(a+3)\cdots \\
&\stackrel{\mathsf{Sp}}{\sim}(a+1)\cdots (a+3)(a+2)(a+3)(a+3)\cdots
\end{align*}
so that $\mathfrak{row}(T)a$ is not an FPF-involution word.
Hence, $T_{(1,3)}\geq a+4$, which implies that the insertion $T\stackrel{\mathsf{Sp}}{\leftarrow}a$ yields no bumped letter to be inserted to the second row; the first row of $T\stackrel{\mathsf{Sp}}{\leftarrow}a$ is
\setlength{\unitlength}{12pt}
\begin{picture}(8,1)
\put(0,0){\line(0,1){1}}
\put(2,0){\line(0,1){1}}
\put(4,0){\line(0,1){1}}
\put(6,0){\line(0,1){1}}
\put(8,0){\line(0,1){1}}
\put(0,0){\line(1,0){8}}
\put(0,1){\line(1,0){8}}
\put(0,0){\makebox(2,1){$a+1$}}
\put(2,0){\makebox(2,1){$a+2$}}
\put(4,0){\makebox(2,1){$a+3$}}
\put(6,0){\makebox(2,1){$A$}}
\end{picture}.

Next, let us consider the case when $y>a+2$.
It is obvious that the insertion $T\stackrel{\mathsf{Sp}}{\leftarrow}a$ yields no bumped letter to be inserted to the second row; the first row of $T\stackrel{\mathsf{Sp}}{\leftarrow}a$ is
\setlength{\unitlength}{12pt}
\begin{picture}(7,1)
\put(0,0){\line(0,1){1}}
\put(2,0){\line(0,1){1}}
\put(4,0){\line(0,1){1}}
\put(5,0){\line(0,1){1}}
\put(7,0){\line(0,1){1}}
\put(0,0){\line(1,0){7}}
\put(0,1){\line(1,0){7}}
\put(0,0){\makebox(2,1){$a+1$}}
\put(2,0){\makebox(2,1){$a+2$}}
\put(4,0){\makebox(1,1){$y$}}
\put(5,0){\makebox(2,1){$A$}}
\end{picture}.

\item $a\equiv x \pmod{2}$.

Let us write $T_{1}$ as
\setlength{\unitlength}{12pt}
\begin{picture}(3,1)
\put(0,0){\line(0,1){1}}
\put(1,0){\line(0,1){1}}
\put(3,0){\line(0,1){1}}
\put(0,0){\line(1,0){3}}
\put(0,1){\line(1,0){3}}
\put(0,0){\makebox(1,1){$x$}}
\put(1,0){\makebox(2,1){$A$}}
\end{picture}.
It is obvious that the insertion $T\stackrel{\mathsf{Sp}}{\leftarrow}a$ yields no bumped letter to be inserted to the second row; the first row of $T\stackrel{\mathsf{Sp}}{\leftarrow}a$ is
\setlength{\unitlength}{12pt}
\begin{picture}(4,1)
\put(0,0){\line(0,1){1}}
\put(1,0){\line(0,1){1}}
\put(2,0){\line(0,1){1}}
\put(4,0){\line(0,1){1}}
\put(0,0){\line(1,0){4}}
\put(0,1){\line(1,0){4}}
\put(0,0){\makebox(1,1){$a$}}
\put(1,0){\makebox(1,1){$x$}}
\put(2,0){\makebox(2,1){$A$}}
\end{picture}.

\end{enumerate}
\end{proof}

\begin{lem} \label{lem:single_bac}
Let $T$ be a tableau 
\setlength{\unitlength}{12pt}
\begin{picture}(1,1)
\put(0,0){\line(0,1){1}}
\put(1,0){\line(0,1){1}}
\put(0,0){\line(1,0){1}}
\put(0,1){\line(1,0){1}}

\put(0,0){\makebox(1,1){$p$}}
\end{picture}
such that $pbac$ is an FPF-involution word, where $a<b<c$.
Then, $T\stackrel{\mathsf{Sp}}{\leftarrow}bac=T\stackrel{\mathsf{Sp}}{\leftarrow}bca$.
\end{lem}

\begin{proof}
By direct computation, we have

\begin{enumerate}
\item $p<a$.

\setlength{\unitlength}{12pt}
\begin{center}
\begin{picture}(12,2)

\put(0,0){\makebox(9,2){$T\stackrel{\mathsf{Sp}}{\leftarrow}bac=T\stackrel{\mathsf{Sp}}{\leftarrow}bca=$}}

\put(9,1){\line(0,1){1}}
\put(10,0){\line(0,1){2}}
\put(11,0){\line(0,1){2}}
\put(12,1){\line(0,1){1}}
\put(10,0){\line(1,0){1}}
\put(9,1){\line(1,0){3}}
\put(9,2){\line(1,0){3}}

\put(9,1){\makebox(1,1){$p$}}
\put(10,0){\makebox(1,1){$b$}}
\put(10,1){\makebox(1,1){$a$}}
\put(11,1){\makebox(1,1){$c$}}
\end{picture}.
\end{center}

\item $a<p<b$.

\begin{enumerate}
\item $p=a+1$.

\setlength{\unitlength}{12pt}
\begin{center}
\begin{picture}(15,1)
\put(0,0){\makebox(9,1){$T\stackrel{\mathsf{Sp}}{\leftarrow}bac=T\stackrel{\mathsf{Sp}}{\leftarrow}bca=$}}
\put(9,0){\line(0,1){1}}
\put(11,0){\line(0,1){1}}
\put(13,0){\line(0,1){1}}
\put(14,0){\line(0,1){1}}
\put(15,0){\line(0,1){1}}
\put(9,0){\line(1,0){6}}
\put(9,1){\line(1,0){6}}
\put(9,0){\makebox(2,1){$a+1$}}
\put(11,0){\makebox(2,1){$a+2$}}
\put(13,0){\makebox(1,1){$b$}}
\put(14,0){\makebox(1,1){$c$}}
\end{picture}.
\end{center}

\item $p\equiv a \pmod{2}$.

\setlength{\unitlength}{12pt}
\begin{center}
\begin{picture}(13,1)
\put(0,0){\makebox(9,1){$T\stackrel{\mathsf{Sp}}{\leftarrow}bac=T\stackrel{\mathsf{Sp}}{\leftarrow}bca=$}}
\put(9,0){\line(0,1){1}}
\put(10,0){\line(0,1){1}}
\put(11,0){\line(0,1){1}}
\put(12,0){\line(0,1){1}}
\put(13,0){\line(0,1){1}}
\put(9,0){\line(1,0){4}}
\put(9,1){\line(1,0){4}}
\put(9,0){\makebox(1,1){$a$}}
\put(10,0){\makebox(1,1){$p$}}
\put(11,0){\makebox(1,1){$b$}}
\put(12,0){\makebox(1,1){$c$}}
\end{picture}.
\end{center}

\end{enumerate}

\item $b<p<c$.

\begin{enumerate}
\item $b=a+1$ and $p=a+2$.

\setlength{\unitlength}{12pt}
\begin{center}
\begin{picture}(16,1)

\put(0,0){\makebox(9,1){$T\stackrel{\mathsf{Sp}}{\leftarrow}bac=T\stackrel{\mathsf{Sp}}{\leftarrow}bca=$}}

\put(9,0){\line(0,1){1}}
\put(11,0){\line(0,1){1}}
\put(13,0){\line(0,1){1}}
\put(15,0){\line(0,1){1}}
\put(16,0){\line(0,1){1}}
\put(9,0){\line(1,0){7}}
\put(9,1){\line(1,0){7}}

\put(9,0){\makebox(2,1){$a+1$}}
\put(11,0){\makebox(2,1){$a+2$}}
\put(13,0){\makebox(2,1){$a+3$}}
\put(15,0){\makebox(1,1){$c$}}
\end{picture}.
\end{center}

\item $b=a+1$ and $p>a+2$.

\setlength{\unitlength}{12pt}
\begin{center}
\begin{picture}(15,1)

\put(0,0){\makebox(9,1){$T\stackrel{\mathsf{Sp}}{\leftarrow}bac=T\stackrel{\mathsf{Sp}}{\leftarrow}bca=$}}

\put(9,0){\line(0,1){1}}
\put(11,0){\line(0,1){1}}
\put(13,0){\line(0,1){1}}
\put(14,0){\line(0,1){1}}
\put(15,0){\line(0,1){1}}
\put(9,0){\line(1,0){6}}
\put(9,1){\line(1,0){6}}

\put(9,0){\makebox(2,1){$a+1$}}
\put(11,0){\makebox(2,1){$a+2$}}
\put(13,0){\makebox(1,1){$p$}}
\put(14,0){\makebox(1,1){$c$}}
\end{picture}.
\end{center}

\item $a\equiv b \pmod{2}$.

\setlength{\unitlength}{12pt}
\begin{center}
\begin{picture}(13,1)

\put(0,0){\makebox(9,1){$T\stackrel{\mathsf{Sp}}{\leftarrow}bac=T\stackrel{\mathsf{Sp}}{\leftarrow}bca=$}}

\put(9,0){\line(0,1){1}}
\put(10,0){\line(0,1){1}}
\put(11,0){\line(0,1){1}}
\put(12,0){\line(0,1){1}}
\put(13,0){\line(0,1){1}}
\put(9,0){\line(1,0){4}}
\put(9,1){\line(1,0){4}}

\put(9,0){\makebox(1,1){$a$}}
\put(10,0){\makebox(1,1){$b$}}
\put(11,0){\makebox(1,1){$p$}}
\put(12,0){\makebox(1,1){$c$}}
\end{picture}.
\end{center}

\end{enumerate}

\item $c<p$.

\begin{enumerate}

\item $b=a+1$.
 
\setlength{\unitlength}{12pt}
\begin{center}
\begin{picture}(15,1)

\put(0,0){\makebox(9,1){$T\stackrel{\mathsf{Sp}}{\leftarrow}bac=T\stackrel{\mathsf{Sp}}{\leftarrow}bca=$}}

\put(9,0){\line(0,1){1}}
\put(11,0){\line(0,1){1}}
\put(13,0){\line(0,1){1}}
\put(14,0){\line(0,1){1}}
\put(15,0){\line(0,1){1}}
\put(9,0){\line(1,0){6}}
\put(9,1){\line(1,0){6}}

\put(9,0){\makebox(2,1){$a+1$}}
\put(11,0){\makebox(2,1){$a+2$}}
\put(13,0){\makebox(1,1){$p$}}
\put(14,0){\makebox(1,1){$c$}}
\end{picture}.
\end{center}

\item $a\equiv b \pmod{2}$.

\setlength{\unitlength}{12pt}
\begin{center}
\begin{picture}(13,1)

\put(0,0){\makebox(9,1){$T\stackrel{\mathsf{Sp}}{\leftarrow}bac=T\stackrel{\mathsf{Sp}}{\leftarrow}bca=$}}

\put(9,0){\line(0,1){1}}
\put(10,0){\line(0,1){1}}
\put(11,0){\line(0,1){1}}
\put(12,0){\line(0,1){1}}
\put(13,0){\line(0,1){1}}
\put(9,0){\line(1,0){4}}
\put(9,1){\line(1,0){4}}

\put(9,0){\makebox(1,1){$a$}}
\put(10,0){\makebox(1,1){$b$}}
\put(11,0){\makebox(1,1){$p$}}
\put(12,0){\makebox(1,1){$c$}}
\end{picture}.
\end{center}
 
\end{enumerate}

\end{enumerate}

The word $pbac$ is not an FPF-involution word in other cases.
\end{proof}

\begin{lem}
Let $T$ be a tableau 
\setlength{\unitlength}{12pt}
\begin{picture}(1,1)
\put(0,0){\line(0,1){1}}
\put(1,0){\line(0,1){1}}
\put(0,0){\line(1,0){1}}
\put(0,1){\line(1,0){1}}

\put(0,0){\makebox(1,1){$p$}}
\end{picture}
such that $pacb$ is an FPF-involution word, where $a<b<c$.
Then, $T\stackrel{\mathsf{Sp}}{\leftarrow}acb=T\stackrel{\mathsf{Sp}}{\leftarrow}cab$.
\end{lem}

\begin{proof}

By direct computation, we have

\begin{enumerate}

\item $p<a$.

\setlength{\unitlength}{12pt}
\begin{center}
\begin{picture}(12,2)

\put(0,0){\makebox(9,2){$T\stackrel{\mathsf{Sp}}{\leftarrow}acb=T\stackrel{\mathsf{Sp}}{\leftarrow}cab=$}}

\put(9,1){\line(0,1){1}}
\put(10,0){\line(0,1){2}}
\put(11,0){\line(0,1){2}}
\put(12,1){\line(0,1){1}}
\put(10,0){\line(1,0){1}}
\put(9,1){\line(1,0){3}}
\put(9,2){\line(1,0){3}}

\put(9,1){\makebox(1,1){$p$}}
\put(10,0){\makebox(1,1){$c$}}
\put(10,1){\makebox(1,1){$a$}}
\put(11,1){\makebox(1,1){$b$}}
\end{picture}.
\end{center}

\item $a<p<b$.

\begin{enumerate}

\item $p=a+1$.

\setlength{\unitlength}{12pt}
\begin{center}
\begin{picture}(14,2)

\put(0,0){\makebox(9,2){$T\stackrel{\mathsf{Sp}}{\leftarrow}acb=T\stackrel{\mathsf{Sp}}{\leftarrow}cab=$}}

\put(9,1){\line(0,1){1}}
\put(11,0){\line(0,1){2}}
\put(13,0){\line(0,1){2}}
\put(14,1){\line(0,1){1}}
\put(11,0){\line(1,0){2}}
\put(9,1){\line(1,0){5}}
\put(9,2){\line(1,0){5}}

\put(9,1){\makebox(2,1){$a+1$}}
\put(11,0){\makebox(2,1){$c$}}
\put(11,1){\makebox(2,1){$a+2$}}
\put(13,1){\makebox(1,1){$b$}}
\end{picture}.
\end{center}

\item $p\equiv a \pmod{2}$.

\setlength{\unitlength}{12pt}
\begin{center}
\begin{picture}(12,2)

\put(0,0){\makebox(9,2){$T\stackrel{\mathsf{Sp}}{\leftarrow}acb=T\stackrel{\mathsf{Sp}}{\leftarrow}cab=$}}

\put(9,1){\line(0,1){1}}
\put(10,0){\line(0,1){2}}
\put(11,0){\line(0,1){2}}
\put(12,1){\line(0,1){1}}
\put(10,0){\line(1,0){1}}
\put(9,1){\line(1,0){3}}
\put(9,2){\line(1,0){3}}

\put(9,1){\makebox(1,1){$a$}}
\put(10,0){\makebox(1,1){$c$}}
\put(10,1){\makebox(1,1){$p$}}
\put(11,1){\makebox(1,1){$b$}}
\end{picture}.
\end{center}

\end{enumerate}

\item $b<p<c$.

\setlength{\unitlength}{12pt}
\begin{center}
\begin{picture}(12,2)

\put(0,0){\makebox(9,2){$T\stackrel{\mathsf{Sp}}{\leftarrow}acb=T\stackrel{\mathsf{Sp}}{\leftarrow}cab=$}}

\put(9,1){\line(0,1){1}}
\put(10,0){\line(0,1){2}}
\put(11,0){\line(0,1){2}}
\put(12,1){\line(0,1){1}}
\put(10,0){\line(1,0){1}}
\put(9,1){\line(1,0){3}}
\put(9,2){\line(1,0){3}}

\put(9,1){\makebox(1,1){$a$}}
\put(10,0){\makebox(1,1){$p$}}
\put(10,1){\makebox(1,1){$b$}}
\put(11,1){\makebox(1,1){$c$}}
\end{picture}.
\end{center}

\item $c<p$.

\begin{enumerate}

\item $p=c+1$.

\setlength{\unitlength}{12pt}
\begin{center}
\begin{picture}(14,2)

\put(0,0){\makebox(9,2){$T\stackrel{\mathsf{Sp}}{\leftarrow}acb=T\stackrel{\mathsf{Sp}}{\leftarrow}cab=$}}

\put(9,1){\line(0,1){1}}
\put(10,0){\line(0,1){2}}
\put(12,0){\line(0,1){2}}
\put(14,1){\line(0,1){1}}
\put(10,0){\line(1,0){2}}
\put(9,1){\line(1,0){5}}
\put(9,2){\line(1,0){5}}

\put(9,1){\makebox(1,1){$a$}}
\put(10,0){\makebox(2,1){$c+1$}}
\put(10,1){\makebox(2,1){$b$}}
\put(12,1){\makebox(2,1){$c+2$}}
\end{picture}.
\end{center}

\item $p\equiv c \pmod{2}$.

\setlength{\unitlength}{12pt}
\begin{center}
\begin{picture}(12,2)

\put(0,0){\makebox(9,2){$T\stackrel{\mathsf{Sp}}{\leftarrow}acb=T\stackrel{\mathsf{Sp}}{\leftarrow}cab=$}}

\put(9,1){\line(0,1){1}}
\put(10,0){\line(0,1){2}}
\put(11,0){\line(0,1){2}}
\put(12,1){\line(0,1){1}}
\put(10,0){\line(1,0){1}}
\put(9,1){\line(1,0){3}}
\put(9,2){\line(1,0){3}}

\put(9,1){\makebox(1,1){$a$}}
\put(10,0){\makebox(1,1){$c$}}
\put(10,1){\makebox(1,1){$b$}}
\put(11,1){\makebox(1,1){$p$}}
\end{picture}.
\end{center}

\end{enumerate}

\end{enumerate}

The word $pacb$ is not an FPF-involution word in other cases.
\end{proof}

\begin{lem}
Let $T$ be a tableau 
\setlength{\unitlength}{12pt}
\begin{picture}(1,1)
\put(0,0){\line(0,1){1}}
\put(1,0){\line(0,1){1}}
\put(0,0){\line(1,0){1}}
\put(0,1){\line(1,0){1}}

\put(0,0){\makebox(1,1){$p$}}
\end{picture}
such that $pa(a+1)a$ is an FPF-involution word.
Then, $T\stackrel{\mathsf{Sp}}{\leftarrow}a(a+1)a=T\stackrel{\mathsf{Sp}}{\leftarrow}(a+1)a(a+1)$.
\end{lem}

\begin{proof}

By direct computation, we have

\begin{enumerate}

\item $p=a-1$.

\setlength{\unitlength}{12pt}
\begin{center}
\begin{picture}(21,2)

\put(0,0){\makebox(15,2){$T\stackrel{\mathsf{Sp}}{\leftarrow}a(a+1)a=T\stackrel{\mathsf{Sp}}{\leftarrow}(a+1)a(a+1)=$}}

\put(15,1){\line(0,1){1}}
\put(17,0){\line(0,1){2}}
\put(19,0){\line(0,1){2}}
\put(21,1){\line(0,1){1}}
\put(17,0){\line(1,0){2}}
\put(15,1){\line(1,0){6}}
\put(15,2){\line(1,0){6}}

\put(15,1){\makebox(2,1){$a-1$}}
\put(17,0){\makebox(2,1){$a+1$}}
\put(17,1){\makebox(2,1){$a$}}
\put(19,1){\makebox(2,1){$a+1$}}
\end{picture}.
\end{center}

\item $p=a+2$.

\setlength{\unitlength}{12pt}
\begin{center}
\begin{picture}(20,2)

\put(0,0){\makebox(15,2){$T\stackrel{\mathsf{Sp}}{\leftarrow}a(a+1)a=T\stackrel{\mathsf{Sp}}{\leftarrow}(a+1)a(a+1)=$}}

\put(15,1){\line(0,1){1}}
\put(16,0){\line(0,1){2}}
\put(18,0){\line(0,1){2}}
\put(20,1){\line(0,1){1}}
\put(16,0){\line(1,0){2}}
\put(15,1){\line(1,0){5}}
\put(15,2){\line(1,0){5}}

\put(15,1){\makebox(1,1){$a$}}
\put(16,0){\makebox(2,1){$a+2$}}
\put(16,1){\makebox(2,1){$a+1$}}
\put(18,1){\makebox(2,1){$a+3$}}
\end{picture}.
\end{center}

\end{enumerate}

The word $pa(a+1)a$ is not an FPF-involution word in other cases.
\end{proof}

\begin{lem} \label{lem:bac}
Let $T$ be an increasing shifted tableau of shape $\lambda=(\lambda_{1},\lambda_{2},\ldots)$ such that $\mathfrak{row}(T)bac$ is an FPF-involution word, where $a<b<c$.
Then, $T\stackrel{\mathsf{Sp}}{\leftarrow}bac=T\stackrel{\mathsf{Sp}}{\leftarrow}bca$.
\end{lem}

\begin{proof}

Let $T_{1}$ be the first row of $T$ and $T^{\prime}$ be a portion of $T$ below the first row.

\begin{flushleft}
\textbf{Case 1:} $T_{(1,\lambda_1)}<a$.
\end{flushleft}

\setlength{\unitlength}{12pt}
\begin{center}
\begin{picture}(16,3)
\put(0,1){\makebox(8,1){$T\stackrel{\mathsf{Sp}}{\leftarrow}bac=T\stackrel{\mathsf{Sp}}{\leftarrow}bca=$}}
\put(8,2){\line(0,1){1}}
\put(9,1){\line(0,1){1}}
\put(12,2){\line(0,1){1}}
\put(13,2){\line(0,1){1}}
\put(14,2){\line(0,1){1}}
\put(9,1){\line(1,0){1}}
\put(8,2){\line(1,0){6}}
\put(8,3){\line(1,0){6}}

\put(14,1){\makebox(2,1){$\leftarrow b$}}

\put(10,0){\makebox(1,1){$\ddots$}}
\put(11,0){\makebox(1,2){$T^{\prime}$}}
\put(8,2){\makebox(4,1){$T_{1}$}}
\put(12,2){\makebox(1,1){$a$}}
\put(13,2){\makebox(1,1){$c$}}
\end{picture}.
\end{center}

\begin{flushleft}
\textbf{Case 2:} $T_{(1,\lambda_1)}=a$.
\end{flushleft}

\setlength{\unitlength}{12pt}
\begin{center}
\begin{picture}(17,3)
\put(0,1){\makebox(8,1){$T\stackrel{\mathsf{Sp}}{\leftarrow}bac=T\stackrel{\mathsf{Sp}}{\leftarrow}bca=$}}
\put(8,2){\line(0,1){1}}
\put(9,1){\line(0,1){1}}
\put(12,2){\line(0,1){1}}
\put(13,2){\line(0,1){1}}
\put(14,2){\line(0,1){1}}
\put(9,1){\line(1,0){1}}
\put(8,2){\line(1,0){6}}
\put(8,3){\line(1,0){6}}

\put(14,1){\makebox(3,1){$\leftarrow a+1$}}

\put(10,0){\makebox(1,1){$\ddots$}}
\put(11,0){\makebox(1,2){$T^{\prime}$}}
\put(8,2){\makebox(4,1){$T_{1}$}}
\put(12,2){\makebox(1,1){$b$}}
\put(13,2){\makebox(1,1){$c$}}
\end{picture}.
\end{center}

\begin{flushleft}
\textbf{Case 3:} $a<T_{(1,\lambda_{1})}<b$.
\end{flushleft}

\begin{enumerate}

\item $T_{(1,1)}\leq a$.

Let $a^{\prime}$ be the smallest entry in $T_{1}$ such that $a< a^{\prime}$ (if it exists).
Then,

\setlength{\unitlength}{12pt}
\begin{center}
\begin{picture}(16,3)
\put(0,1){\makebox(8,1){$T\stackrel{\mathsf{Sp}}{\leftarrow}bac=T\stackrel{\mathsf{Sp}}{\leftarrow}bca=$}}
\put(8,2){\line(0,1){1}}
\put(9,1){\line(0,1){1}}
\put(12,2){\line(0,1){1}}
\put(13,2){\line(0,1){1}}
\put(14,2){\line(0,1){1}}
\put(9,1){\line(1,0){1}}
\put(8,2){\line(1,0){6}}
\put(8,3){\line(1,0){6}}

\put(14,1){\makebox(2,1){$\leftarrow \Tilde{a}$}}

\put(10,0){\makebox(1,1){$\ddots$}}
\put(11,0){\makebox(1,2){$T^{\prime}$}}
\put(8,2){\makebox(4,1){$T_{1}^{\prime}$}}
\put(12,2){\makebox(1,1){$b$}}
\put(13,2){\makebox(1,1){$c$}}
\end{picture},
\end{center}
where $\Tilde{a}$ is either $a+1$ or $a^{\prime}$ and $T_{1}^{\prime}=T_{1}$ if $\Tilde{a}=a+1$ or $T_{1}^{\prime}$ is obtained from $T_{1}$ by replacing $a^{\prime}$ by $a$ if $\Tilde{a}=a^{\prime}$.

\item $T_{(1,1)}>a$ and $T_{(1,1)}\not\equiv a \pmod{2}$.

Let us write $T_{1}$ as
\setlength{\unitlength}{12pt}
\begin{picture}(4,1)
\put(0,0){\line(0,1){1}}
\put(1,0){\line(0,1){1}}
\put(2,0){\line(0,1){1}}
\put(4,0){\line(0,1){1}}
\put(0,0){\line(1,0){4}}
\put(0,1){\line(1,0){4}}
\put(0,0){\makebox(1,1){$x$}}
\put(1,0){\makebox(1,1){$y$}}
\put(2,0){\makebox(2,1){$A$}}
\end{picture}.
By Lemma~\ref{lem:T1}, $x=a+1$ and 

\setlength{\unitlength}{12pt}
\begin{center}
\begin{picture}(16,3)
\put(0,1){\makebox(8,1){$T\stackrel{\mathsf{Sp}}{\leftarrow}bac=T\stackrel{\mathsf{Sp}}{\leftarrow}bca=$}}
\put(8,2){\line(0,1){1}}
\put(10,1){\line(0,1){2}}
\put(12,2){\line(0,1){1}}
\put(14,2){\line(0,1){1}}
\put(15,2){\line(0,1){1}}
\put(16,2){\line(0,1){1}}
\put(10,1){\line(1,0){2}}
\put(8,2){\line(1,0){8}}
\put(8,3){\line(1,0){8}}
\put(12,0){\makebox(1,1){$\ddots$}}
\put(13,0){\makebox(1,2){$T^{\prime}$}}
\put(8,2){\makebox(2,1){$a+1$}}
\put(10,2){\makebox(2,1){$a+2$}}
\put(12,2){\makebox(2,1){$A^{\prime}$}}
\put(14,2){\makebox(1,1){$b$}}
\put(15,2){\makebox(1,1){$c$}}
\end{picture},
\end{center}
where 
\setlength{\unitlength}{12pt}
\begin{picture}(2,1)
\put(0,0){\line(0,1){1}}
\put(2,0){\line(0,1){1}}
\put(0,0){\line(1,0){2}}
\put(0,1){\line(1,0){2}}
\put(0,0){\makebox(2,1){$A^{\prime}$}}
\end{picture} is 
\setlength{\unitlength}{12pt}
\begin{picture}(4,1)
\put(0,0){\line(0,1){1}}
\put(2,0){\line(0,1){1}}
\put(4,0){\line(0,1){1}}
\put(0,0){\line(1,0){4}}
\put(0,1){\line(1,0){4}}
\put(0,0){\makebox(2,1){$a+3$}}
\put(2,0){\makebox(2,1){$A$}}
\end{picture} if $y=a+2$ and 
\setlength{\unitlength}{12pt}
\begin{picture}(3,1)
\put(0,0){\line(0,1){1}}
\put(1,0){\line(0,1){1}}
\put(3,0){\line(0,1){1}}
\put(0,0){\line(1,0){3}}
\put(0,1){\line(1,0){3}}
\put(0,0){\makebox(1,1){$y$}}
\put(1,0){\makebox(2,1){$A$}}
\end{picture} if $y>a+2$.

\item $T_{(1,1)}>a$ and $T_{(1,1)}\equiv a \pmod{2}$.

Let us write $T_{1}$ as
\setlength{\unitlength}{12pt}
\begin{picture}(3,1)
\put(0,0){\line(0,1){1}}
\put(1,0){\line(0,1){1}}
\put(3,0){\line(0,1){1}}
\put(0,0){\line(1,0){3}}
\put(0,1){\line(1,0){3}}
\put(0,0){\makebox(1,1){$x$}}
\put(1,0){\makebox(2,1){$A$}}
\end{picture}.
Then,

\setlength{\unitlength}{12pt}
\begin{center}
\begin{picture}(14,3)
\put(0,1){\makebox(8,1){$T\stackrel{\mathsf{Sp}}{\leftarrow}bac=T\stackrel{\mathsf{Sp}}{\leftarrow}bca=$}}
\put(8,2){\line(0,1){1}}
\put(9,1){\line(0,1){2}}
\put(10,2){\line(0,1){1}}
\put(12,2){\line(0,1){1}}
\put(13,2){\line(0,1){1}}
\put(14,2){\line(0,1){1}}
\put(9,1){\line(1,0){1}}
\put(8,2){\line(1,0){6}}
\put(8,3){\line(1,0){6}}
\put(10,0){\makebox(1,1){$\ddots$}}
\put(11,0){\makebox(1,2){$T^{\prime}$}}
\put(8,2){\makebox(1,1){$a$}}
\put(9,2){\makebox(1,1){$x$}}
\put(10,2){\makebox(2,1){$A$}}
\put(12,2){\makebox(1,1){$b$}}
\put(13,2){\makebox(1,1){$c$}}
\end{picture},
\end{center}
  
\end{enumerate}

\begin{flushleft}
\textbf{Case 4:} $b\leq T_{(1,\lambda_{1})}< c$.
\end{flushleft}

\begin{enumerate}

\item $T_{(1,1)}=b$ and $a\not\equiv b \pmod{2}$.

\setlength{\unitlength}{12pt}
\begin{center}
\begin{picture}(13,3)
\put(0,1){\makebox(5,1){$T\stackrel{\mathsf{Sp}}{\leftarrow}b=\Tilde{T}=$}}
\put(5,2){\line(0,1){1}}
\put(6,1){\line(0,1){2}}
\put(10,2){\line(0,1){1}}
\put(6,1){\line(1,0){1}}
\put(5,2){\line(1,0){5}}
\put(5,3){\line(1,0){5}}
\put(10,1){\makebox(3,1){$\leftarrow b+1$}}
\put(7,0){\makebox(1,1){$\ddots$}}
\put(8,0){\makebox(1,2){$T^{\prime}$}}
\put(5,2){\makebox(1,1){$b$}}
\end{picture}.
\end{center}

\setlength{\unitlength}{12pt}
\begin{center}
\begin{picture}(11,4)
\put(0,2){\makebox(3,1){$\Tilde{T}\stackrel{\mathsf{Sp}}{\leftarrow}a=$}}
\put(3,3){\line(0,1){1}}
\put(4,2){\line(0,1){2}}
\put(8,3){\line(0,1){1}}
\put(4,2){\line(1,0){1}}
\put(3,3){\line(1,0){5}}
\put(3,4){\line(1,0){5}}
\put(8,2){\makebox(3,1){$\leftarrow b+1$}}
\put(4,1){\makebox(1,1){$\uparrow$}}
\put(3,0){\makebox(3,1){$a+2$}}
\put(5,1){\makebox(1,1){$\ddots$}}
\put(6,1){\makebox(1,2){$T^{\prime}$}}
\put(3,3){\makebox(1,1){$b$}}
\end{picture},
\end{center}
where the row insertion of $b+1$ is followed by the column insertion of $a+2$.
From this configuration, we have that $b=a+1$.
Let us write the first row of $T$ as

\setlength{\unitlength}{12pt}
\begin{center}
\begin{picture}(8,2)
\put(0,0){\line(0,1){1}}
\put(2,0){\line(0,1){1}}
\put(4,0){\line(0,1){1}}
\put(6,0){\line(0,1){1}}
\put(8,0){\line(0,1){1}}
\put(0,0){\line(1,0){8}}
\put(0,1){\line(1,0){8}}
\put(0,0){\makebox(2,1){$a+1$}}
\put(2,0){\makebox(2,1){$\cdots$}}
\put(4,0){\makebox(2,1){$a+l$}}
\put(6,0){\makebox(2,1){$A$}}
\put(4,1){\makebox(2,1){$\small{l}$}}
\end{picture},
\end{center}
where the leftmost entry of $A$ is strictly greater than $a+l+1$.
Then,

\setlength{\unitlength}{12pt}
\begin{center}
\begin{picture}(22,3)
\put(0,1){\makebox(8,1){$T\stackrel{\mathsf{Sp}}{\leftarrow}bac=T\stackrel{\mathsf{Sp}}{\leftarrow} bca=$}}
\put(8,2){\line(0,1){1}}
\put(10,1){\line(0,1){2}}
\put(12,2){\line(0,1){1}}
\put(16,2){\line(0,1){1}}
\put(18,2){\line(0,1){1}}
\put(19,2){\line(0,1){1}}
\put(10,1){\line(1,0){1}}
\put(8,2){\line(1,0){11}}
\put(8,3){\line(1,0){11}}
\put(19,1){\makebox(3,1){$\leftarrow a+2$}}
\put(11,0){\makebox(1,1){$\ddots$}}
\put(12,0){\makebox(1,2){$T^{\prime}$}}
\put(8,2){\makebox(2,1){$a+1$}}
\put(10,2){\makebox(2,1){$\cdots$}}
\put(12,2){\makebox(4,1){$a+l+1$}}
\put(16,2){\makebox(2,1){$A$}}
\put(18,2){\makebox(1,1){$c$}}
\end{picture}.
\end{center}

\item $T_{(1,1)}=b$ and $a\equiv b \pmod{2}$.

\setlength{\unitlength}{12pt}
\begin{center}
\begin{picture}(16,3)
\put(0,1){\makebox(8,1){$T\stackrel{\mathsf{Sp}}{\leftarrow}bac=T\stackrel{\mathsf{Sp}}{\leftarrow}bca=$}}
\put(8,2){\line(0,1){1}}
\put(9,1){\line(0,1){2}}
\put(12,2){\line(0,1){1}}
\put(13,2){\line(0,1){1}}
\put(9,1){\line(1,0){1}}
\put(8,2){\line(1,0){5}}
\put(8,3){\line(1,0){5}}
\put(13,1){\makebox(3,1){$\leftarrow b+1$}}
\put(10,0){\makebox(1,1){$\ddots$}}
\put(11,0){\makebox(1,2){$T^{\prime}$}}
\put(8,2){\makebox(1,1){$a$}}
\put(9,2){\makebox(3,1){$T_{1}$}}
\put(12,2){\makebox(1,1){$c$}}
\end{picture}.
\end{center}

\item $T_{(1,1)}>b$ and $T_{(1,1)}\not\equiv b \pmod{2}$.

By Lemma~\ref{lem:T1}, $T_{(1,1)}=b+1$.
It is easy to see that $a\not\equiv b \pmod{2}$ and 

\setlength{\unitlength}{12pt}
\begin{center}
\begin{picture}(16,3)
\put(0,1){\makebox(8,1){$T\stackrel{\mathsf{Sp}}{\leftarrow}bac=T\stackrel{\mathsf{Sp}}{\leftarrow}bca=$}}
\put(8,2){\line(0,1){1}}
\put(9,1){\line(0,1){2}}
\put(11,2){\line(0,1){1}}
\put(13,2){\line(0,1){1}}
\put(15,2){\line(0,1){1}}
\put(16,2){\line(0,1){1}}
\put(9,1){\line(1,0){2}}
\put(8,2){\line(1,0){8}}
\put(8,3){\line(1,0){8}}
\put(11,0){\makebox(1,1){$\ddots$}}
\put(12,0){\makebox(1,2){$T^{\prime}$}}
\put(8,2){\makebox(1,1){$a$}}
\put(9,2){\makebox(2,1){$b+1$}}
\put(11,2){\makebox(2,1){$b+2$}}
\put(13,2){\makebox(2,1){$\cdots$}}
\put(15,2){\makebox(1,1){$c$}}
\end{picture}.
\end{center}

\item $T_{(1,1)}>b$ and $T_{(1,1)}\equiv b \pmod{2}$.

It is easy to see that $a\equiv b \pmod{2}$ and

\setlength{\unitlength}{12pt}
\begin{center}
\begin{picture}(14,3)
\put(0,1){\makebox(8,1){$T\stackrel{\mathsf{Sp}}{\leftarrow}bac=T\stackrel{\mathsf{Sp}}{\leftarrow}bca=$}}
\put(8,2){\line(0,1){1}}
\put(9,2){\line(0,1){1}}
\put(9,1){\line(0,1){2}}
\put(10,2){\line(0,1){1}}
\put(13,2){\line(0,1){1}}
\put(14,2){\line(0,1){1}}
\put(9,1){\line(1,0){1}}
\put(8,2){\line(1,0){6}}
\put(8,3){\line(1,0){6}}
\put(10,0){\makebox(1,1){$\ddots$}}
\put(11,0){\makebox(1,2){$T^{\prime}$}}
\put(8,2){\makebox(1,1){$a$}}
\put(9,2){\makebox(1,1){$b$}}
\put(10,2){\makebox(3,1){$T_{1}$}}
\put(13,2){\makebox(1,1){$c$}}
\end{picture}.
\end{center}

\end{enumerate}

\begin{flushleft}
\textbf{Case 5:} $T_{(1,\lambda_{1})}\geq c$.
\end{flushleft}

Let $b^{\prime}$ be the smallest entry in $T_{1}$ such that $b< b^{\prime}$.
Let us denote $T\stackrel{\mathsf{Sp}}{\leftarrow}b$ by $\Tilde{T}$ and let $a^{\prime}$ (resp. $c^{\prime}$) be the smallest entry in $(\Tilde{T})_{1}$, the first row of $\Tilde{T}$, such that $a< a^{\prime}$ (resp. $c< c^{\prime}$).
Since $b$ exists in $(\Tilde{T})_{1}$, $a^{\prime}<c^{\prime}$.

\setlength{\unitlength}{12pt}
\begin{center}
\begin{picture}(11,3)
\put(0,1){\makebox(4,1){$T\stackrel{\mathsf{Sp}}{\leftarrow}bac=$}}
\put(4,2){\line(0,1){1}}
\put(5,1){\line(0,1){1}}
\put(9,2){\line(0,1){1}}
\put(5,1){\line(1,0){1}}
\put(4,2){\line(1,0){5}}
\put(4,3){\line(1,0){5}}
\put(9,1){\makebox(2,1){$\leftarrow qpr$}}
\put(6,0){\makebox(1,1){$\ddots$}}
\put(7,0){\makebox(1,2){$T^{\prime}$}}
\put(4,2){\makebox(5,1){$T_{1}^{\prime}$}}
\end{picture},
\end{center}
and
\setlength{\unitlength}{12pt}
\begin{center}
\begin{picture}(11,3)
\put(0,1){\makebox(4,1){$T\stackrel{\mathsf{Sp}}{\leftarrow}bca=$}}
\put(4,2){\line(0,1){1}}
\put(5,1){\line(0,1){1}}
\put(9,2){\line(0,1){1}}
\put(5,1){\line(1,0){1}}
\put(4,2){\line(1,0){5}}
\put(4,3){\line(1,0){5}}
\put(9,1){\makebox(2,1){$\leftarrow qrp$}}
\put(6,0){\makebox(1,1){$\ddots$}}
\put(7,0){\makebox(1,2){$T^{\prime}$}}
\put(4,2){\makebox(5,1){$T_{1}^{\prime}$}}
\end{picture},
\end{center}
where $p=a^{\prime}$ or $a+1$, $q=b^{\prime}$ or $b+1$, $r=c^{\prime}$ or $c+1$ and $p<q<r$.
Suppose that $T^{\prime}\stackrel{\mathsf{Sp}}{\leftarrow}qpr=T^{\prime}\stackrel{\mathsf{Sp}}{\leftarrow}qrp$ with $p<q<r$.
This assumption is satisfied when $T^{\prime}$ is a tableau of a single box (Lemma~\ref{lem:single_bac}).
We claim that $T\stackrel{\mathsf{Sp}}{\leftarrow}bac=T\stackrel{\mathsf{Sp}}{\leftarrow}bca$ by induction.

A column insertion is called the \emph{insertion to the top} when the inserting letter $x$ is greater then or equal to $x^{\prime}$, the first entry of this column, or the column insertion to the empty column.
If it is not the column insertion to the empty column and $x<x^{\prime}$, then we call such an insertion the \emph{strong insertion to the top}.

Suppose that an insertion to the top, which is not the column insertion to the empty column, begins at some column position in the course of insertion $T\stackrel{\mathsf{Sp}}{\leftarrow}x$.
This implies the strong insertion to the top begins at the same column position in the course of insertion $T^{\prime}\stackrel{\mathsf{Sp}}{\leftarrow}x^{\prime}$, where $x^{\prime}$ is the inserting letter corresponding to $x$.
Such a column insertion is called the \emph{initial insertion to the top} and the inserting letter is called the \emph{initial column inserting letter}.
The initial insertion to the top triggers the subsequent insertions to the top.
If an insertion to the empty column appears, then it is the last column insertion.
If such a column insertion is not the subsequent insertion followed by some insertion to the top, then it is called the \emph{initial empty column insertion} and the inserting letter is also called the initial column inserting letter but the initial empty column insertion does not trigger the subsequent insertions to the top. 
The initial empty column insertion must appear earlier than any other initial insertions to the top and once it appears no further initial empty column insertions appear. 
It is obvious that the sequence of initial column inserting letters in $T\stackrel{\mathsf{Sp}}{\leftarrow}bac$ (resp. $T\stackrel{\mathsf{Sp}}{\leftarrow}bca$) and that of $T^{\prime}\stackrel{\mathsf{Sp}}{\leftarrow}qpr$ (resp. $T^{\prime}\stackrel{\mathsf{Sp}}{\leftarrow}qrp$) are identical.

Let us suppose that the sequence of initial column inserting letters is $xy$ in $T\stackrel{\mathsf{Sp}}{\leftarrow}bac$ with $x<y$ and that the corresponding insertions are not the initial empty column insertions.
The insertions to the top go as follows, where we assume that $x<x^{\prime}$ and $y<y^{\prime}$.
\setlength{\unitlength}{12pt}
\begin{center}
\begin{picture}(25,7)
\put(1,4){\line(0,1){2}}
\put(2,4){\line(0,1){2}}
\put(0,4){\line(1,0){6}}
\put(0,5){\line(1,0){6}}
\put(0,6){\line(1,0){6}}
\put(1,0){\makebox(1,1){$x$}}
\put(1,1){\makebox(1,1){$\uparrow$}}
\put(2,2){\makebox(2,2){$\vdots$}}
\put(1,4){\makebox(1,1){$x^{\prime\prime}$}}
\put(1,5){\makebox(1,1){$x^{\prime}$}}
\put(1,6){\makebox(1,1){$\small{i}$}}
\put(6,4){\makebox(2,2){$\rightsquigarrow$}}
\put(9,4){\line(0,1){2}}
\put(10,4){\line(0,1){2}}
\put(11,5){\line(0,1){1}}
\put(13,4){\line(0,1){2}}
\put(14,4){\line(0,1){2}}
\put(8,4){\line(1,0){7}}
\put(8,5){\line(1,0){7}}
\put(8,6){\line(1,0){7}}
\put(13,0){\makebox(1,1){$y$}}
\put(13,1){\makebox(1,1){$\uparrow$}}
\put(11,2){\makebox(1,2){$\vdots$}}
\put(9,4){\makebox(1,1){$x^{\prime\prime}$}}
\put(9,5){\makebox(1,1){$x$}}
\put(9,6){\makebox(1,1){$\small{i}$}}
\put(10,5){\makebox(1,1){$x^{\prime}$}}
\put(13,4){\makebox(1,1){$y^{\prime\prime}$}}
\put(13,5){\makebox(1,1){$y^{\prime}$}}
\put(13,6){\makebox(1,1){$\small{j}$}}
\put(15,4){\makebox(2,2){$\rightsquigarrow$}}
\put(18,4){\line(0,1){2}}
\put(19,4){\line(0,1){2}}
\put(20,5){\line(0,1){1}}
\put(22,4){\line(0,1){2}}
\put(23,4){\line(0,1){2}}
\put(24,5){\line(0,1){1}}
\put(17,4){\line(1,0){8}}
\put(17,5){\line(1,0){8}}
\put(17,6){\line(1,0){8}}
\put(20,2){\makebox(1,2){$\vdots$}}
\put(18,4){\makebox(1,1){$x^{\prime\prime}$}}
\put(18,5){\makebox(1,1){$x$}}
\put(18,6){\makebox(1,1){$\small{i}$}}
\put(19,5){\makebox(1,1){$x^{\prime}$}}
\put(22,4){\makebox(1,1){$y^{\prime\prime}$}}
\put(22,5){\makebox(1,1){$y$}}
\put(22,6){\makebox(1,1){$\small{j}$}}
\put(23,5){\makebox(1,1){$y^{\prime}$}}
\end{picture}.
\end{center}
The corresponding insertions to the top in $T^{\prime}\stackrel{\mathsf{Sp}}{\leftarrow}qpr$ go as follows.

\setlength{\unitlength}{12pt}
\begin{center}
\begin{picture}(25,6)
\put(1,4){\line(0,1){1}}
\put(2,4){\line(0,1){1}}
\put(0,4){\line(1,0){6}}
\put(0,5){\line(1,0){6}}
\put(1,0){\makebox(1,1){$x$}}
\put(1,1){\makebox(1,1){$\uparrow$}}
\put(2,2){\makebox(2,2){$\vdots$}}
\put(1,4){\makebox(1,1){$x^{\prime\prime}$}}
\put(1,5){\makebox(1,1){$\small{i}$}}
\put(6,4){\makebox(2,1){$\rightsquigarrow$}}
\put(9,4){\line(0,1){1}}
\put(10,4){\line(0,1){1}}
\put(11,4){\line(0,1){1}}
\put(13,4){\line(0,1){1}}
\put(14,4){\line(0,1){1}}
\put(8,4){\line(1,0){7}}
\put(8,5){\line(1,0){7}}
\put(13,0){\makebox(1,1){$y$}}
\put(13,1){\makebox(1,1){$\uparrow$}}
\put(11,2){\makebox(1,2){$\vdots$}}
\put(9,4){\makebox(1,1){$x$}}
\put(9,5){\makebox(1,1){$\small{i}$}}
\put(10,4){\makebox(1,1){$x^{\prime}$}}
\put(13,4){\makebox(1,1){$y^{\prime\prime}$}}
\put(13,5){\makebox(1,1){$\small{j}$}}
\put(15,4){\makebox(2,1){$\rightsquigarrow$}}
\put(18,4){\line(0,1){1}}
\put(19,4){\line(0,1){1}}
\put(20,4){\line(0,1){1}}
\put(22,4){\line(0,1){1}}
\put(23,4){\line(0,1){1}}
\put(24,4){\line(0,1){1}}
\put(17,4){\line(1,0){8}}
\put(17,5){\line(1,0){8}}
\put(20,2){\makebox(1,2){$\vdots$}}
\put(18,4){\makebox(1,1){$x$}}
\put(18,5){\makebox(1,1){$\small{i}$}}
\put(19,4){\makebox(1,1){$x^{\prime\prime}$}}
\put(22,4){\makebox(1,1){$y$}}
\put(22,5){\makebox(1,1){$\small{j}$}}
\put(23,4){\makebox(1,1){$y^{\prime\prime}$}}
\end{picture}.
\end{center}
Since $T^{\prime}\stackrel{\mathsf{Sp}}{\leftarrow}qpr=T^{\prime}\stackrel{\mathsf{Sp}}{\leftarrow}qrp$ by the assumption of induction, the insertions to the top in $T^{\prime}\stackrel{\mathsf{Sp}}{\leftarrow}qrp$ must be the same as above or 

\setlength{\unitlength}{12pt}
\begin{center}
\begin{picture}(25,6)
\put(1,4){\line(0,1){1}}
\put(2,4){\line(0,1){1}}
\put(4,4){\line(0,1){1}}
\put(5,4){\line(0,1){1}}
\put(0,4){\line(1,0){6}}
\put(0,5){\line(1,0){6}}
\put(4,0){\makebox(1,1){$y$}}
\put(4,1){\makebox(1,1){$\uparrow$}}
\put(2,2){\makebox(2,2){$\vdots$}}
\put(1,4){\makebox(1,1){$x^{\prime\prime}$}}
\put(1,5){\makebox(1,1){$\small{i}$}}
\put(4,4){\makebox(1,1){$y^{\prime\prime}$}}
\put(4,5){\makebox(1,1){$\small{j-1}$}}
\put(6,4){\makebox(2,1){$\rightsquigarrow$}}
\put(9,4){\line(0,1){1}}
\put(10,4){\line(0,1){1}}
\put(12,4){\line(0,1){1}}
\put(13,4){\line(0,1){1}}
\put(14,4){\line(0,1){1}}
\put(8,4){\line(1,0){7}}
\put(8,5){\line(1,0){7}}
\put(9,0){\makebox(1,1){$x$}}
\put(9,1){\makebox(1,1){$\uparrow$}}
\put(11,2){\makebox(1,2){$\vdots$}}
\put(9,4){\makebox(1,1){$x^{\prime\prime}$}}
\put(9,5){\makebox(1,1){$\small{i}$}}
\put(12,4){\makebox(1,1){$y$}}
\put(13,4){\makebox(1,1){$y^{\prime\prime}$}}
\put(13,5){\makebox(1,1){$\small{j}$}}
\put(15,4){\makebox(2,1){$\rightsquigarrow$}}
\put(18,4){\line(0,1){1}}
\put(19,4){\line(0,1){1}}
\put(20,4){\line(0,1){1}}
\put(22,4){\line(0,1){1}}
\put(23,4){\line(0,1){1}}
\put(24,4){\line(0,1){1}}
\put(17,4){\line(1,0){8}}
\put(17,5){\line(1,0){8}}
\put(20,2){\makebox(1,2){$\vdots$}}
\put(18,4){\makebox(1,1){$x$}}
\put(18,5){\makebox(1,1){$\small{i}$}}
\put(19,4){\makebox(1,1){$x^{\prime\prime}$}}
\put(22,4){\makebox(1,1){$y$}}
\put(22,5){\makebox(1,1){$\small{j}$}}
\put(23,4){\makebox(1,1){$y^{\prime\prime}$}}
\end{picture}.
\end{center}
In this case, the insertions to the top in $T\stackrel{\mathsf{Sp}}{\leftarrow}bca$ go as follows.

\setlength{\unitlength}{12pt}
\begin{center}
\begin{picture}(25,7)
\put(1,4){\line(0,1){2}}
\put(2,4){\line(0,1){2}}
\put(4,4){\line(0,1){2}}
\put(5,4){\line(0,1){2}}
\put(0,4){\line(1,0){6}}
\put(0,5){\line(1,0){6}}
\put(0,6){\line(1,0){6}}
\put(4,0){\makebox(1,1){$y$}}
\put(4,1){\makebox(1,1){$\uparrow$}}
\put(2,2){\makebox(2,2){$\vdots$}}
\put(1,4){\makebox(1,1){$x^{\prime\prime}$}}
\put(1,5){\makebox(1,1){$x^{\prime}$}}
\put(1,6){\makebox(1,1){$\small{i}$}}
\put(4,4){\makebox(1,1){$y^{\prime\prime}$}}
\put(4,5){\makebox(1,1){$y^{\prime}$}}
\put(4,6){\makebox(1,1){$\small{j-1}$}}
\put(6,4){\makebox(2,2){$\rightsquigarrow$}}

\put(9,4){\line(0,1){2}}
\put(10,4){\line(0,1){2}}
\put(12,4){\line(0,1){2}}
\put(13,4){\line(0,1){2}}
\put(14,4){\line(0,1){2}}
\put(8,4){\line(1,0){7}}
\put(8,5){\line(1,0){7}}
\put(8,6){\line(1,0){7}}
\put(9,0){\makebox(1,1){$x$}}
\put(9,1){\makebox(1,1){$\uparrow$}}
\put(11,2){\makebox(1,2){$\vdots$}}
\put(9,4){\makebox(1,1){$x^{\prime\prime}$}}
\put(9,5){\makebox(1,1){$x^{\prime}$}}
\put(9,6){\makebox(1,1){$\small{i}$}}
\put(12,4){\makebox(1,1){$y^{\prime\prime}$}}
\put(12,5){\makebox(1,1){$y$}}
\put(13,5){\makebox(1,1){$y^{\prime}$}}
\put(13,6){\makebox(1,1){$\small{j}$}}
\put(15,4){\makebox(2,2){$\rightsquigarrow$}}

\put(18,4){\line(0,1){2}}
\put(19,4){\line(0,1){2}}
\put(20,5){\line(0,1){1}}
\put(22,4){\line(0,1){2}}
\put(23,4){\line(0,1){2}}
\put(24,5){\line(0,1){1}}
\put(17,4){\line(1,0){8}}
\put(17,5){\line(1,0){8}}
\put(17,6){\line(1,0){8}}
\put(20,2){\makebox(1,2){$\vdots$}}
\put(18,4){\makebox(1,1){$x^{\prime\prime}$}}
\put(18,5){\makebox(1,1){$x$}}
\put(18,6){\makebox(1,1){$\small{i}$}}
\put(19,5){\makebox(1,1){$x^{\prime}$}}
\put(22,4){\makebox(1,1){$y^{\prime\prime}$}}
\put(22,5){\makebox(1,1){$y$}}
\put(22,6){\makebox(1,1){$\small{j}$}}
\put(23,5){\makebox(1,1){$y^{\prime}$}}
\end{picture}.
\end{center}
As a result, $T\stackrel{\mathsf{Sp}}{\leftarrow}bac=T\stackrel{\mathsf{Sp}}{\leftarrow}bca$.
The verifications of other cases are similar.
We omit the details.

\end{proof}

\begin{ex}
We have that $327\stackrel{\mathsf{CK}}{\sim}372$, 

\setlength{\unitlength}{12pt}
\begin{center}
\begin{picture}(22,2)
\put(0,1){\line(0,1){1}}
\put(1,0){\line(0,1){2}}
\put(2,0){\line(0,1){2}}
\put(3,0){\line(0,1){2}}
\put(4,1){\line(0,1){1}}
\put(1,0){\line(1,0){2}}
\put(0,1){\line(1,0){4}}
\put(0,2){\line(1,0){4}}
\put(0,1){\makebox(1,1){$2$}}
\put(1,0){\makebox(1,1){$6$}}
\put(1,1){\makebox(1,1){$4$}}
\put(2,0){\makebox(1,1){$7$}}
\put(2,1){\makebox(1,1){$5$}}
\put(3,1){\makebox(1,1){$8$}}
\put(4,0.25){\makebox(3,2){$\stackrel{\mathsf{Sp}}{\leftarrow}327$}}
\put(7,0){\makebox(1,2){$=$}}
\put(8,1){\line(0,1){1}}
\put(9,0){\line(0,1){2}}
\put(10,0){\line(0,1){2}}
\put(11,0){\line(0,1){2}}
\put(12,1){\line(0,1){1}}
\put(9,0){\line(1,0){2}}
\put(8,1){\line(1,0){4}}
\put(8,2){\line(1,0){4}}
\put(8,1){\makebox(1,1){$2$}}
\put(9,0){\makebox(1,1){$6$}}
\put(9,1){\makebox(1,1){$3$}}
\put(10,0){\makebox(1,1){$7$}}
\put(10,1){\makebox(1,1){$5$}}
\put(11,1){\makebox(1,1){$7$}}
\put(12,0){\makebox(3,1){$\leftarrow 438$}}
\put(15,0){\makebox(1,2){$=$}}
\put(16,1){\line(0,1){1}}
\put(17,0){\line(0,1){2}}
\put(18,0){\line(0,1){2}}
\put(19,0){\line(0,1){2}}
\put(20,0){\line(0,1){2}}
\put(21,1){\line(0,1){1}}
\put(22,1){\line(0,1){1}}
\put(17,0){\line(1,0){3}}
\put(16,1){\line(1,0){6}}
\put(16,2){\line(1,0){6}}
\put(16,1){\makebox(1,1){$2$}}
\put(17,0){\makebox(1,1){$4$}}
\put(17,1){\makebox(1,1){$3$}}
\put(18,0){\makebox(1,1){$6$}}
\put(18,1){\makebox(1,1){$5$}}
\put(19,0){\makebox(1,1){$8$}}
\put(19,1){\makebox(1,1){$6$}}
\put(20,1){\makebox(1,1){$7$}}
\put(21,1){\makebox(1,1){$8$}}
\end{picture},
\end{center}
and
\setlength{\unitlength}{12pt}
\begin{center}
\begin{picture}(22,2)

\put(0,1){\line(0,1){1}}
\put(1,0){\line(0,1){2}}
\put(2,0){\line(0,1){2}}
\put(3,0){\line(0,1){2}}
\put(4,1){\line(0,1){1}}
\put(1,0){\line(1,0){2}}
\put(0,1){\line(1,0){4}}
\put(0,2){\line(1,0){4}}

\put(0,1){\makebox(1,1){$2$}}
\put(1,0){\makebox(1,1){$6$}}
\put(1,1){\makebox(1,1){$4$}}
\put(2,0){\makebox(1,1){$7$}}
\put(2,1){\makebox(1,1){$5$}}
\put(3,1){\makebox(1,1){$8$}}

\put(4,0.25){\makebox(3,2){$\stackrel{\mathsf{Sp}}{\leftarrow}372$}}

\put(7,0){\makebox(1,2){$=$}}

\put(8,1){\line(0,1){1}}
\put(9,0){\line(0,1){2}}
\put(10,0){\line(0,1){2}}
\put(11,0){\line(0,1){2}}
\put(12,1){\line(0,1){1}}
\put(9,0){\line(1,0){2}}
\put(8,1){\line(1,0){4}}
\put(8,2){\line(1,0){4}}

\put(8,1){\makebox(1,1){$2$}}
\put(9,0){\makebox(1,1){$6$}}
\put(9,1){\makebox(1,1){$3$}}
\put(10,0){\makebox(1,1){$7$}}
\put(10,1){\makebox(1,1){$5$}}
\put(11,1){\makebox(1,1){$7$}}

\put(12,0){\makebox(3,1){$\leftarrow 483$}}

\put(15,0){\makebox(1,2){$=$}}

\put(16,1){\line(0,1){1}}
\put(17,0){\line(0,1){2}}
\put(18,0){\line(0,1){2}}
\put(19,0){\line(0,1){2}}
\put(20,0){\line(0,1){2}}
\put(21,1){\line(0,1){1}}
\put(22,1){\line(0,1){1}}
\put(17,0){\line(1,0){3}}
\put(16,1){\line(1,0){6}}
\put(16,2){\line(1,0){6}}

\put(16,1){\makebox(1,1){$2$}}
\put(17,0){\makebox(1,1){$4$}}
\put(17,1){\makebox(1,1){$3$}}
\put(18,0){\makebox(1,1){$6$}}
\put(18,1){\makebox(1,1){$5$}}
\put(19,0){\makebox(1,1){$8$}}
\put(19,1){\makebox(1,1){$6$}}
\put(20,1){\makebox(1,1){$7$}}
\put(21,1){\makebox(1,1){$8$}}

\end{picture}.
\end{center}

In both cases, the insertions to the top go as follows.
\setlength{\unitlength}{12pt}
\begin{center}
\begin{picture}(19,4)

\put(0,3){\line(0,1){1}}
\put(1,3){\line(0,1){1}}
\put(2,3){\line(0,1){1}}
\put(3,3){\line(0,1){1}}
\put(4,3){\line(0,1){1}}
\put(0,3){\line(1,0){4}}
\put(0,4){\line(1,0){4}}

\put(0,3){\makebox(1,1){$2$}}
\put(1,3){\makebox(1,1){$3$}}
\put(2,3){\makebox(1,1){$5$}}
\put(3,3){\makebox(1,1){$7$}}

\put(3,0){\makebox(1,1){$7$}}
\put(3,1){\makebox(1,1){$\uparrow$}}
\put(4,3){\makebox(2,1){$\rightsquigarrow$}}

\put(6,3){\line(0,1){1}}
\put(7,3){\line(0,1){1}}
\put(8,3){\line(0,1){1}}
\put(9,3){\line(0,1){1}}
\put(10,3){\line(0,1){1}}
\put(11,3){\line(0,1){1}}
\put(6,3){\line(1,0){5}}
\put(6,4){\line(1,0){5}}

\put(6,3){\makebox(1,1){$2$}}
\put(7,3){\makebox(1,1){$3$}}
\put(8,3){\makebox(1,1){$5$}}
\put(9,3){\makebox(1,1){$7$}}
\put(10,3){\makebox(1,1){$8$}}

\put(8,0){\makebox(1,1){$5$}}
\put(8,1){\makebox(1,1){$\uparrow$}}
\put(11,3){\makebox(2,1){$\rightsquigarrow$}}

\put(13,3){\line(0,1){1}}
\put(14,3){\line(0,1){1}}
\put(15,3){\line(0,1){1}}
\put(16,3){\line(0,1){1}}
\put(17,3){\line(0,1){1}}
\put(18,3){\line(0,1){1}}
\put(19,3){\line(0,1){1}}
\put(13,3){\line(1,0){6}}
\put(13,4){\line(1,0){6}}

\put(13,3){\makebox(1,1){$2$}}
\put(14,3){\makebox(1,1){$3$}}
\put(15,3){\makebox(1,1){$5$}}
\put(16,3){\makebox(1,1){$6$}}
\put(17,3){\makebox(1,1){$7$}}
\put(18,3){\makebox(1,1){$8$}}
\end{picture}.
\end{center}
\end{ex}

\begin{lem} \label{lem:acb}
Let $T$ be an increasing shifted tableau of shape $\lambda=(\lambda_{1},\lambda_{2},\ldots)$ such that $\mathfrak{row}(T)acb$ is an FPF-involution word, where $a<b<c$.
Then, $T\stackrel{\mathsf{Sp}}{\leftarrow}acb=T\stackrel{\mathsf{Sp}}{\leftarrow}cab$.
\end{lem}

\begin{proof}

Let $T_{1}$ be the first row of $T$ and $T^{\prime}$ be a portion of $T$ below the first row.

\begin{flushleft}
\textbf{Case 1:} $T_{(1,\lambda_1)}<a$.
\end{flushleft}

\setlength{\unitlength}{12pt}
\begin{center}
\begin{picture}(16,3)
\put(0,1){\makebox(8,1){$T\stackrel{\mathsf{Sp}}{\leftarrow}acb=T\stackrel{\mathsf{Sp}}{\leftarrow}cab=$}}
\put(8,2){\line(0,1){1}}
\put(9,1){\line(0,1){1}}
\put(12,2){\line(0,1){1}}
\put(13,2){\line(0,1){1}}
\put(14,2){\line(0,1){1}}
\put(9,1){\line(1,0){1}}
\put(8,2){\line(1,0){6}}
\put(8,3){\line(1,0){6}}

\put(14,1){\makebox(2,1){$\leftarrow c$}}

\put(10,0){\makebox(1,1){$\ddots$}}
\put(11,0){\makebox(1,2){$T^{\prime}$}}
\put(8,2){\makebox(4,1){$T_{1}$}}
\put(12,2){\makebox(1,1){$a$}}
\put(13,2){\makebox(1,1){$b$}}
\end{picture}.
\end{center}

\begin{flushleft}
\textbf{Case 2:} $T_{(1,\lambda_{1})}=a$.
\end{flushleft}

This case must be excluded because $\mathfrak{row}(T)a$ is not an FPF-involution word.

\begin{flushleft}
\textbf{Case 3:} $a<T_{(1,\lambda_{1})}<b$.
\end{flushleft}

Let $a^{\prime}$ be the smallest entry such that $a<a^{\prime}$.
Then,

\setlength{\unitlength}{12pt}
\begin{center}
\begin{picture}(15,3)
\put(0,1){\makebox(8,1){$T\stackrel{\mathsf{Sp}}{\leftarrow}acb=T\stackrel{\mathsf{Sp}}{\leftarrow}cab=$}}
\put(8,2){\line(0,1){1}}
\put(9,1){\line(0,1){1}}
\put(12,2){\line(0,1){1}}
\put(13,2){\line(0,1){1}}
\put(9,1){\line(1,0){1}}
\put(8,2){\line(1,0){5}}
\put(8,3){\line(1,0){5}}
\put(13,1){\makebox(2,1){$\leftarrow \Tilde{a}c$}}
\put(10,0){\makebox(1,1){$\ddots$}}
\put(11,0){\makebox(1,2){$T^{\prime}$}}
\put(8,2){\makebox(4,1){$T_{1}^{\prime}$}}
\put(12,2){\makebox(1,1){$b$}}
\end{picture},
\end{center}
where $\Tilde{a}$ is either $a+1$ or $a^{\prime}$.
  
\begin{flushleft}
\textbf{Case 4:} $b\leq T_{(1,\lambda_{1})}<c$.
\end{flushleft}

Let $a^{\prime}$ (resp. $b^{\prime}$) be the smallest entry such that $a<a^{\prime}$ (resp. $b<b^{\prime}$).

\setlength{\unitlength}{12pt}
\begin{center}
\begin{picture}(15,3)
\put(0,1){\makebox(8,1){$T\stackrel{\mathsf{Sp}}{\leftarrow}acb=T\stackrel{\mathsf{Sp}}{\leftarrow}cab=$}}
\put(8,2){\line(0,1){1}}
\put(9,1){\line(0,1){1}}
\put(12,2){\line(0,1){1}}
\put(13,2){\line(0,1){1}}
\put(9,1){\line(1,0){1}}
\put(8,2){\line(1,0){5}}
\put(8,3){\line(1,0){5}}
\put(13,1){\makebox(2,1){$\leftarrow \Tilde{a}\Tilde{b}$}}
\put(10,0){\makebox(1,1){$\ddots$}}
\put(11,0){\makebox(1,2){$T^{\prime}$}}
\put(8,2){\makebox(4,1){$T_{1}^{\prime}$}}
\put(12,2){\makebox(1,1){$c$}}
\end{picture},
\end{center}
where $\Tilde{a}$ (resp. $\Tilde{b}$) is either $a+1$ or $a^{\prime}$ (resp. either $b+1$ or $b^{\prime}$).

\begin{flushleft}
\textbf{Case 5:} $T_{(1,\lambda_{1})}=c$.
\end{flushleft}

This case must be excluded because $\mathfrak{row}(T)c$ is not an FPF-involution word.

\begin{flushleft}
\textbf{Case 6:} $T_{(1,\lambda_{1})}>c$.
\end{flushleft}

\begin{enumerate}

\item $T_{(1,1)}>a$ and $T_{(1,1)}\not\equiv a \pmod{2}$.

By Lemma~\ref{lem:T1}, the first row of $T$ is either 
\setlength{\unitlength}{12pt}
\begin{center}
\begin{picture}(19,1)
\put(0,0){\makebox(3,1){$T_{1}^{(1)}=$}}
\put(3,0){\line(0,1){1}}
\put(5,0){\line(0,1){1}}
\put(7,0){\line(0,1){1}}
\put(9,0){\line(0,1){1}}
\put(3,0){\line(1,0){6}}
\put(3,1){\line(1,0){6}}
\put(3,0){\makebox(2,1){$a+1$}}
\put(5,0){\makebox(2,1){$a+2$}}
\put(7,0){\makebox(2,1){$A$}}
\put(10,0){\makebox(1,1){$\text{or}$}}
\put(12,0){\makebox(3,1){$T_{1}^{(2)}=$}}
\put(15,0){\line(0,1){1}}
\put(17,0){\line(0,1){1}}
\put(19,0){\line(0,1){1}}
\put(15,0){\line(1,0){4}}
\put(15,1){\line(1,0){4}}
\put(15,0){\makebox(2,1){$a+1$}}
\put(17,0){\makebox(2,1){$A^{\prime}$}}
\end{picture},
\end{center}
where the leftmost entry in $A$ is greater than or equal to $a+4$ and the leftmost entry in $A^{\prime}$ is greater than or equal to $a+3$.
 
\begin{enumerate}
\item $c=a+2$.

If $T_{1}=T_{1}^{(1)}$, then

\setlength{\unitlength}{12pt}
\begin{center}
\begin{picture}(22,3)
\put(0,1){\makebox(8,1){$T\stackrel{\mathsf{Sp}}{\leftarrow}acb=T\stackrel{\mathsf{Sp}}{\leftarrow}cab=$}}
\put(8,2){\line(0,1){1}}
\put(10,1){\line(0,1){2}}
\put(12,2){\line(0,1){1}}
\put(14,2){\line(0,1){1}}
\put(16,2){\line(0,1){1}}
\put(10,1){\line(1,0){2}}
\put(8,2){\line(1,0){8}}
\put(8,3){\line(1,0){8}}
\put(12,0){\makebox(1,1){$\ddots$}}
\put(13,0){\makebox(1,2){$T^{\prime}$}}
\put(8,2){\makebox(2,1){$a+1$}}
\put(10,2){\makebox(2,1){$a+2$}}
\put(12,2){\makebox(2,1){$a+3$}}
\put(14,2){\makebox(2,1){$A$}}
\put(16,1){\makebox(6,1){$\leftarrow (a+3)(a+2)$}}
\end{picture},
\end{center}

If $T_{1}=T_{1}^{(2)}$, then the leftmost entry in $A^{\prime}$ is $a+3$.
Otherwise, all the entries in $T^{\prime}$ are greater than or equal to $a+5$ and all the entries in $A^{\prime}$ are greater than or equal to $a+4$ so that 
\begin{align*}
\mathfrak{row}(T)ac&=\mathfrak{row}(T^{\prime})T_{(1,1)}\mathfrak{row}(A^{\prime})a(a+2) \\
&\stackrel{\mathsf{Sp}}{\sim}(a+1)a(a+2)\mathfrak{row}(T^{\prime})\mathfrak{row}(A^{\prime}) \\
&\stackrel{\mathsf{Sp}}{\sim}(a+1)(a+2)(a+2)\mathfrak{row}(T^{\prime})\mathfrak{row}(A^{\prime}),
\end{align*}
which is not an FPF-involution word.

\setlength{\unitlength}{12pt}
\begin{center}
\begin{picture}(21,3)
\put(0,1){\makebox(8,1){$T\stackrel{\mathsf{Sp}}{\leftarrow}acb=T\stackrel{\mathsf{Sp}}{\leftarrow}cab=$}}
\put(8,2){\line(0,1){1}}
\put(10,1){\line(0,1){2}}
\put(12,2){\line(0,1){1}}
\put(15,2){\line(0,1){1}}
\put(10,1){\line(1,0){2}}
\put(8,2){\line(1,0){7}}
\put(8,3){\line(1,0){7}}
\put(12,0){\makebox(1,1){$\ddots$}}
\put(13,0){\makebox(1,2){$T^{\prime}$}}
\put(8,2){\makebox(2,1){$a+1$}}
\put(10,2){\makebox(2,1){$a+2$}}
\put(12,2){\makebox(3,1){$A^{\prime}$}}
\put(15,1){\makebox(6,1){$\leftarrow (a+3)(a+2)$}}
\end{picture},
\end{center}

\item $c>a+2$.

The argument similar to \textbf{Case 4} (1) in the proof of Lemma~\ref{lem:bac} shows that $T_{(1,1)}=a+1$ and $T\stackrel{\mathsf{Sp}}{\leftarrow}acb=T\stackrel{\mathsf{Sp}}{\leftarrow}cab$.

\end{enumerate} 
  
\item $T_{(1,1)}>a$ and $T_{(1,1)}\equiv a \pmod{2}$.

By Lemma~\ref{lem:even}, $a$ is an even letter.
Let us write

\setlength{\unitlength}{12pt}
\begin{center}
\begin{picture}(8,4)
\put(0,1){\makebox(2,2){$T=$}}
\put(2,3){\line(0,1){1}}
\put(3,2){\line(0,1){2}}
\put(4,1){\line(0,1){1}}
\put(8,3){\line(0,1){1}}
\put(4,1){\line(1,0){1}}
\put(3,2){\line(1,0){4}}
\put(2,3){\line(1,0){6}}
\put(2,4){\line(1,0){6}}
\put(5,0){\makebox(1,1){$\ddots$}}
\put(2,3){\makebox(1,1){$x$}}
\put(3,2){\makebox(4,1){$B$}}
\put(4,3){\makebox(4,1){$A$}}
\end{picture}.
\end{center}

\begin{enumerate}
\item $x>c+1$.

Any entry in $T$ is greater than or equal to $c+2$ so that $\mathfrak{row}(T)cab\stackrel{\mathsf{Sp}}{\sim}cab\mathfrak{row}(T)$.
Since this is an FPF-involution word, $c$ is an even letter so that $x\equiv c \pmod{2}$.

\setlength{\unitlength}{12pt}
\begin{center}
\begin{picture}(14,4)
\put(0,2){\makebox(8,1){$T\stackrel{\mathsf{Sp}}{\leftarrow}acb=T\stackrel{\mathsf{Sp}}{\leftarrow}cab=$}}
\put(8,3){\line(0,1){1}}
\put(9,2){\line(0,1){2}}
\put(10,1){\line(0,1){3}}
\put(11,3){\line(0,1){1}}
\put(14,3){\line(0,1){1}}
\put(10,1){\line(1,0){1}}
\put(9,2){\line(1,0){4}}
\put(8,3){\line(1,0){6}}
\put(8,4){\line(1,0){6}}
\put(11,0){\makebox(1,1){$\ddots$}}
\put(8,3){\makebox(1,1){$a$}}
\put(9,2){\makebox(1,1){$c$}}
\put(9,3){\makebox(1,1){$b$}}
\put(10,3){\makebox(1,1){$x$}}
\put(10,2){\makebox(3,1){$B$}}
\put(11,3){\makebox(3,1){$A$}}
\end{picture}.
\end{center}

\item $x=c+1$.

\setlength{\unitlength}{12pt}
\begin{center}
\begin{picture}(14,4)
\put(0,2){\makebox(8,1){$T\stackrel{\mathsf{Sp}}{\leftarrow}acb=T\stackrel{\mathsf{Sp}}{\leftarrow}cab=$}}
\put(8,3){\line(0,1){1}}
\put(9,2){\line(0,1){2}}
\put(11,1){\line(0,1){3}}
\put(14,3){\line(0,1){1}}
\put(11,1){\line(1,0){1}}
\put(9,2){\line(1,0){4}}
\put(8,3){\line(1,0){6}}
\put(8,4){\line(1,0){6}}
\put(12,0){\makebox(1,1){$\ddots$}}
\put(8,3){\makebox(1,1){$a$}}
\put(9,2){\makebox(2,1){$c+1$}}
\put(9,3){\makebox(2,1){$b$}}
\put(11,2){\makebox(2,1){$B$}}
\put(11,3){\makebox(3,1){$A^{\prime}$}}
\end{picture},
\end{center}
where 
\setlength{\unitlength}{12pt}
\begin{picture}(3,1)
\put(0,0){\line(0,1){1}}
\put(3,0){\line(0,1){1}}
\put(0,0){\line(1,0){3}}
\put(0,1){\line(1,0){3}}
\put(0,0){\makebox(3,1){$A^{\prime}$}}
\end{picture}
is obtained from 
\setlength{\unitlength}{12pt}
\begin{picture}(3,1)
\put(0,0){\line(0,1){1}}
\put(3,0){\line(0,1){1}}
\put(0,0){\line(1,0){3}}
\put(0,1){\line(1,0){3}}
\put(0,0){\makebox(3,1){$A$}}
\end{picture}
by the initial insertion to the top at the first column with the initial column inserting letter $c+2$.

\item $x=c$.

The tableau $T$ has the following configuration.

\setlength{\unitlength}{12pt}
\begin{center}
\begin{picture}(7,4)
\put(0,3){\line(0,1){1}}
\put(1,2){\line(0,1){2}}
\put(3,1){\line(0,1){3}}
\put(4,2){\line(0,1){2}}
\put(7,3){\line(0,1){1}}
\put(3,1){\line(1,0){1}}
\put(1,2){\line(1,0){5}}
\put(0,3){\line(1,0){7}}
\put(0,4){\line(1,0){7}}
\put(4,0){\makebox(1,1){$\ddots$}}
\put(0,3){\makebox(1,1){$c$}}
\put(1,2){\makebox(2,1){$c+2$}}
\put(1,3){\makebox(2,1){$c+1$}}
\put(3,2){\makebox(1,1){$q$}}
\put(3,3){\makebox(1,1){$p$}}
\put(4,2){\makebox(2,1){$B$}}
\put(4,3){\makebox(3,1){$A$}}
\end{picture}.
\end{center}

Otherwise, $\mathfrak{row}(T)c$ is not an FPF-involution word.
It is clear that $T_{(1,2)}=c+1$.
If $T_{(2,2)}>c+2$, then
\begin{align*}
\mathfrak{row}(T)c&\stackrel{\mathsf{Sp}}{\sim}\mathfrak{row}(T^{\prime})c(c+1)cp \mathfrak{row}(A) \\
&\stackrel{\mathsf{Sp}}{\sim}c(c+1)c\mathfrak{row}(T^{\prime})p\mathfrak{row}(A) \\
&\stackrel{\mathsf{Sp}}{\sim}(c+1)c(c+1)\mathfrak{row}(T^{\prime})p\mathfrak{row}(A),
\end{align*}
which implies $\mathfrak{row}(T)c$ is not an FPF-involution word.
Consequently, $T_{(2,2)}=c+2$.

Let $a_{1}$ (resp. $b_{1}$) be the leftmost entry of $A$ (resp. $B$).

\begin{enumerate}
\item $p=c+2$ and $q=c+3$.

\setlength{\unitlength}{12pt}
\begin{center}
\begin{picture}(20,7)
\put(0,4){\makebox(8,1){$T\stackrel{\mathsf{Sp}}{\leftarrow}acb=T\stackrel{\mathsf{Sp}}{\leftarrow}cab=$}}
\put(8,6){\line(0,1){1}}
\put(9,5){\line(0,1){2}}
\put(10,4){\line(0,1){3}}
\put(12,3){\line(0,1){4}}
\put(14,2){\line(0,1){5}}
\put(15,2){\line(0,1){5}}
\put(14,2){\line(1,0){1}}
\put(12,3){\line(1,0){4}}
\put(10,4){\line(1,0){7}}
\put(9,5){\line(1,0){9}}
\put(8,6){\line(1,0){11}}
\put(8,7){\line(1,0){12}}
\put(14,0){\makebox(1,1){$c+4$}}
\put(14,1){\makebox(1,1){$\uparrow$}}
\put(15,1){\makebox(1,1){$\ddots$}}
\put(8,6){\makebox(1,1){$a$}}
\put(9,5){\makebox(1,1){$c$}}
\put(9,6){\makebox(1,1){$b$}}
\put(10,4){\makebox(2,1){$\cdot$}}
\put(10,5){\makebox(2,1){$c+2$}}
\put(10,6){\makebox(2,1){$c+1$}}
\put(12,3){\makebox(2,1){$\cdot$}}
\put(12,4){\makebox(2,1){$\cdot$}}
\put(12,5){\makebox(2,1){$c+3$}}
\put(12,6){\makebox(2,1){$c+2$}}
\put(14,2){\makebox(1,1){$\cdot$}}
\put(14,3){\makebox(1,1){$\cdot$}}
\put(14,4){\makebox(1,1){$\cdot$}}
\put(14,5){\makebox(1,1){$b_{1}$}}
\put(14,6){\makebox(1,1){$a_{1}$}}
\end{picture}.
\end{center}

\item $p=c+2$ and $q\geq c+4$.

\setlength{\unitlength}{12pt}
\begin{center}
\begin{picture}(20,7)
\put(0,4){\makebox(8,1){$T\stackrel{\mathsf{Sp}}{\leftarrow}acb=T\stackrel{\mathsf{Sp}}{\leftarrow}cab=$}}
\put(8,6){\line(0,1){1}}
\put(9,5){\line(0,1){2}}
\put(10,4){\line(0,1){3}}
\put(12,3){\line(0,1){4}}
\put(14,2){\line(0,1){5}}
\put(15,2){\line(0,1){5}}
\put(14,2){\line(1,0){1}}
\put(12,3){\line(1,0){4}}
\put(10,4){\line(1,0){7}}
\put(9,5){\line(1,0){9}}
\put(8,6){\line(1,0){11}}
\put(8,7){\line(1,0){12}}
\put(14,0){\makebox(1,1){$q$}}
\put(14,1){\makebox(1,1){$\uparrow$}}
\put(15,1){\makebox(1,1){$\ddots$}}
\put(8,6){\makebox(1,1){$a$}}
\put(9,5){\makebox(1,1){$c$}}
\put(9,6){\makebox(1,1){$b$}}
\put(10,4){\makebox(2,1){$\cdot$}}
\put(10,5){\makebox(2,1){$c+2$}}
\put(10,6){\makebox(2,1){$c+1$}}
\put(12,3){\makebox(2,1){$\cdot$}}
\put(12,4){\makebox(2,1){$\cdot$}}
\put(12,5){\makebox(2,1){$c+3$}}
\put(12,6){\makebox(2,1){$c+2$}}
\put(14,2){\makebox(1,1){$\cdot$}}
\put(14,3){\makebox(1,1){$\cdot$}}
\put(14,4){\makebox(1,1){$\cdot$}}
\put(14,5){\makebox(1,1){$b_{1}$}}
\put(14,6){\makebox(1,1){$a_{1}$}}
\end{picture}.
\end{center}

\item $p=c+3$.

In this case, $q=c+4$.
Otherwise, $qp\stackrel{\mathsf{Sp}}{\sim}pq$ so that
\begin{align*}
\mathfrak{row}(T)c&=\cdots T_{(2,2)}q\mathfrak{row}(B)c(c+1)p\mathfrak{row}(A)c \\
&\stackrel{\mathsf{Sp}}{\sim}\cdots T_{(2,2)}q\mathfrak{row}(B)(c+1)c(c+1)p\mathfrak{row}(A) \\
&\stackrel{\mathsf{Sp}}{\sim}(c+2)(c+3)q\cdots \mathfrak{row}(B)c(c+1)p\mathfrak{row}(A) \\
&\stackrel{\mathsf{Sp}}{\sim}(c+2)(c+3)pq\cdots \mathfrak{row}(B)c(c+1)\mathfrak{row}(A),
\end{align*}
which is not an FPF-involution word.

\setlength{\unitlength}{12pt}
\begin{center}
\begin{picture}(20,7)
\put(0,4){\makebox(8,1){$T\stackrel{\mathsf{Sp}}{\leftarrow}acb=T\stackrel{\mathsf{Sp}}{\leftarrow}cab=$}}
\put(8,6){\line(0,1){1}}
\put(9,5){\line(0,1){2}}
\put(10,4){\line(0,1){3}}
\put(12,3){\line(0,1){4}}
\put(14,2){\line(0,1){5}}
\put(15,2){\line(0,1){5}}
\put(14,2){\line(1,0){1}}
\put(12,3){\line(1,0){4}}
\put(10,4){\line(1,0){7}}
\put(9,5){\line(1,0){9}}
\put(8,6){\line(1,0){11}}
\put(8,7){\line(1,0){12}}
\put(14,0){\makebox(1,1){$c+4$}}
\put(14,1){\makebox(1,1){$\uparrow$}}
\put(15,1){\makebox(1,1){$\ddots$}}
\put(8,6){\makebox(1,1){$a$}}
\put(9,5){\makebox(1,1){$c$}}
\put(9,6){\makebox(1,1){$b$}}
\put(10,4){\makebox(2,1){$\cdot$}}
\put(10,5){\makebox(2,1){$c+2$}}
\put(10,6){\makebox(2,1){$c+1$}}
\put(12,3){\makebox(2,1){$\cdot$}}
\put(12,4){\makebox(2,1){$\cdot$}}
\put(12,5){\makebox(2,1){$c+4$}}
\put(12,6){\makebox(2,1){$c+3$}}
\put(14,2){\makebox(1,1){$\cdot$}}
\put(14,3){\makebox(1,1){$\cdot$}}
\put(14,4){\makebox(1,1){$\cdot$}}
\put(14,5){\makebox(1,1){$b_{1}$}}
\put(14,6){\makebox(1,1){$a_{1}$}}
\end{picture}.
\end{center}

\item $p\geq c+4$.

\setlength{\unitlength}{12pt}
\begin{center}
\begin{picture}(19,4)
\put(0,2){\makebox(8,1){$T\stackrel{\mathsf{Sp}}{\leftarrow}acb=T\stackrel{\mathsf{Sp}}{\leftarrow}cab=$}}
\put(8,3){\line(0,1){1}}
\put(9,2){\line(0,1){2}}
\put(10,1){\line(0,1){3}}
\put(12,2){\line(0,1){2}}
\put(14,2){\line(0,1){2}}
\put(15,3){\line(0,1){1}}
\put(19,3){\line(0,1){1}}
\put(10,1){\line(1,0){2}}
\put(9,2){\line(1,0){8}}
\put(8,3){\line(1,0){11}}
\put(8,4){\line(1,0){11}}
\put(12,0){\makebox(1,1){$\ddots$}}
\put(8,3){\makebox(1,1){$a$}}
\put(9,2){\makebox(1,1){$c$}}
\put(9,3){\makebox(1,1){$b$}}
\put(10,2){\makebox(2,1){$c+2$}}
\put(10,3){\makebox(2,1){$c+1$}}
\put(12,2){\makebox(2,1){$q$}}
\put(12,3){\makebox(2,1){$c+3$}}
\put(14,3){\makebox(1,1){$p$}}
\put(14,2){\makebox(3,1){$B$}}
\put(15,3){\makebox(4,1){$A$}}
\end{picture},
\end{center}

\end{enumerate}

\item $x<c$.

\setlength{\unitlength}{12pt}
\begin{center}
\begin{picture}(15,3)
\put(0,1){\makebox(8,1){$T\stackrel{\mathsf{Sp}}{\leftarrow}ac=T\stackrel{\mathsf{Sp}}{\leftarrow}ca=$}}
\put(8,2){\line(0,1){1}}
\put(9,1){\line(0,1){2}}
\put(10,2){\line(0,1){1}}
\put(13,2){\line(0,1){1}}
\put(9,1){\line(1,0){1}}
\put(8,2){\line(1,0){5}}
\put(8,3){\line(1,0){5}}
\put(13,1){\makebox(2,1){$\leftarrow \Tilde{c}$}}
\put(10,0){\makebox(1,1){$\ddots$}}
\put(11,0){\makebox(1,2){$T^{\prime}$}}
\put(8,2){\makebox(1,1){$a$}}
\put(9,2){\makebox(1,1){$x$}}
\put(10,2){\makebox(3,1){$B$}}
\end{picture}.
\end{center}

Here, if $c^{\prime}$ with $c<c^{\prime}$ exists in $A$, then $\Tilde{c}=c^{\prime}$ and 
\setlength{\unitlength}{12pt}
\begin{picture}(2,1)
\put(0,0){\line(0,1){1}}
\put(2,0){\line(0,1){1}}
\put(0,0){\line(1,0){2}}
\put(0,1){\line(1,0){2}}
\put(0,0){\makebox(2,1){$B$}}
\end{picture} is obtained from 
\setlength{\unitlength}{12pt}
\begin{picture}(2,1)
\put(0,0){\line(0,1){1}}
\put(2,0){\line(0,1){1}}
\put(0,0){\line(1,0){2}}
\put(0,1){\line(1,0){2}}
\put(0,0){\makebox(2,1){$A$}}
\end{picture} by replacing $c^{\prime}$ by $c$.
Otherwise, $\Tilde{c}=c+1$ and 
\begin{picture}(2,1)
\put(0,0){\line(0,1){1}}
\put(2,0){\line(0,1){1}}
\put(0,0){\line(1,0){2}}
\put(0,1){\line(1,0){2}}
\put(0,0){\makebox(2,1){$A$}}
\end{picture} is unchanged.
Now, it is clear that $T\stackrel{\mathsf{Sp}}{\leftarrow}acb=T\stackrel{\mathsf{Sp}}{\leftarrow}cab$.

\end{enumerate}

\item $T_{(1,1)}\leq a$.

Let $w$ (resp. $w^{\prime}$) be the sequence of three letters bumped out from the first row of $T$, which is inserted to the second row of $T$, in the insertion 
$T\stackrel{\mathsf{Sp}}{\leftarrow}acb$ (resp. $T\stackrel{\mathsf{Sp}}{\leftarrow}cab$).
Then, we have that $w\stackrel{\mathsf{CK}}{\sim}w^{\prime}$ by a case by case analysis.
Contrary to \textbf{Case 5} of the proof of Lemma~\ref{lem:bac}, we have all three types of Coxeter-Knuth related words $w$ and $w^{\prime}$. 
The same argument as in \textbf{Case 5} of the proof of Lemma~\ref{lem:bac} shows $T\stackrel{\mathsf{Sp}}{\leftarrow}acb=T\stackrel{\mathsf{Sp}}{\leftarrow}cab$.

\end{enumerate}

\end{proof}

\begin{lem} \label{lem:braid}
Let $T$ be an increasing shifted tableau of shape $\lambda=(\lambda_{1},\lambda_{2},\ldots)$ such that $\mathfrak{row}(T)a(a+1)a$ is an FPF-involution word.
Then, $T\stackrel{\mathsf{Sp}}{\leftarrow}a(a+1)a=T\stackrel{\mathsf{Sp}}{\leftarrow}(a+1)a(a+1)$.
\end{lem}

\begin{proof}
Let $T_{1}$ be the first row of $T$ and $T^{\prime}$ be a portion of $T$ below the first row.

\begin{flushleft}
\textbf{Case 1:} $T_{(1,\lambda_{1})}<a$.
\end{flushleft}

\setlength{\unitlength}{12pt}
\begin{center}
\begin{picture}(25,3)
\put(0,1){\makebox(15,1){$T\stackrel{\mathsf{Sp}}{\leftarrow}a(a+1)a=T\stackrel{\mathsf{Sp}}{\leftarrow}(a+1)a(a+1)=$}}
\put(15,2){\line(0,1){1}}
\put(16,1){\line(0,1){1}}
\put(19,2){\line(0,1){1}}
\put(20,2){\line(0,1){1}}
\put(22,2){\line(0,1){1}}
\put(16,1){\line(1,0){1}}
\put(15,2){\line(1,0){7}}
\put(15,3){\line(1,0){7}}
\put(22,1){\makebox(3,1){$\leftarrow a+1$}}
\put(17,0){\makebox(1,1){$\ddots$}}
\put(18,0){\makebox(1,2){$T^{\prime}$}}
\put(16,2){\makebox(2,1){$T_{1}$}}
\put(19,2){\makebox(1,1){$a$}}
\put(20,2){\makebox(2,1){$a+1$}}
\end{picture}.
\end{center}

\begin{flushleft}
\textbf{Case 2:} $T_{(1,\lambda_1)}\geq a$.
\end{flushleft}

If $T_{(1,1)}\geq a$, then $a\leq T_{(1,1)}\leq a+2$.
Otherwise, $\mathfrak{row}(T)a(a+1)a$ is not an FPF-involution word.

\begin{enumerate}
\item $T_{(1,1)}=a$.

The first row of $T$ has the configuration, 
\setlength{\unitlength}{12pt}
\begin{picture}(6,1)
\put(0,0){\line(0,1){1}}
\put(1,0){\line(0,1){1}}
\put(3,0){\line(0,1){1}}
\put(6,0){\line(0,1){1}}
\put(0,0){\line(1,0){6}}
\put(0,1){\line(1,0){6}}
\put(0,0){\makebox(1,1){$a$}}
\put(1,0){\makebox(2,1){$a+1$}}
\put(3,0){\makebox(3,1){$A$}}
\end{picture}.
Otherwise, $\mathfrak{row}(T)a$ is not an FPF-involution word.

\setlength{\unitlength}{12pt}
\begin{center}
\begin{picture}(22,3)

\put(0,1){\makebox(7,1){$T\stackrel{\mathsf{Sp}}{\leftarrow}a(a+1)a=$}}

\put(7,2){\line(0,1){1}}
\put(8,1){\line(0,1){2}}
\put(10,2){\line(0,1){1}}
\put(13,2){\line(0,1){1}}
\put(8,1){\line(1,0){2}}
\put(7,2){\line(1,0){6}}
\put(7,3){\line(1,0){6}}

\put(10,0){\makebox(1,1){$\ddots$}}
\put(11,0){\makebox(1,2){$T^{\prime}$}}
\put(7,2){\makebox(1,1){$a$}}
\put(8,2){\makebox(2,1){$a+1$}}
\put(10,2){\makebox(3,1){$A$}}

\put(13,1){\makebox(9,1){$\leftarrow (a+1)(a+2)(a+1)$}}

\end{picture}.
\end{center}

\setlength{\unitlength}{12pt}
\begin{center}
\begin{picture}(24,3)

\put(0,1){\makebox(9,1){$T\stackrel{\mathsf{Sp}}{\leftarrow}(a+1)a(a+1)=$}}

\put(9,2){\line(0,1){1}}
\put(10,1){\line(0,1){2}}
\put(12,2){\line(0,1){1}}
\put(15,2){\line(0,1){1}}
\put(10,1){\line(1,0){2}}
\put(9,2){\line(1,0){6}}
\put(9,3){\line(1,0){6}}

\put(12,0){\makebox(1,1){$\ddots$}}
\put(13,0){\makebox(1,2){$T^{\prime}$}}
\put(9,2){\makebox(1,1){$a$}}
\put(10,2){\makebox(2,1){$a+1$}}
\put(12,2){\makebox(3,1){$A$}}

\put(15,1){\makebox(9,1){$\leftarrow (a+2)(a+1)(a+2)$}}
\end{picture}.
\end{center}
The same argument as in \textbf{Case 5} of the proof of Lemma~\ref{lem:bac} shows $T\stackrel{\mathsf{Sp}}{\leftarrow}a(a+1)a=T\stackrel{\mathsf{Sp}}{\leftarrow}(a+1)a(a+1)$.

\item $T_{(1,1)}=a+1$.

By Lemma~\ref{lem:T1}, we have

\setlength{\unitlength}{12pt}
\begin{center}
\begin{picture}(9,4)

\put(0,1){\makebox(2,2){$T=$}}

\put(2,3){\line(0,1){1}}
\put(4,2){\line(0,1){2}}
\put(6,1){\line(0,1){3}}
\put(9,3){\line(0,1){1}}
\put(6,1){\line(1,0){1}}
\put(4,2){\line(1,0){4}}
\put(2,3){\line(1,0){7}}
\put(2,4){\line(1,0){7}}

\put(7,0){\makebox(1,1){$\ddots$}}
\put(4,2){\makebox(2,1){$a+3$}}
\put(6,2){\makebox(2,1){$B$}}
\put(2,3){\makebox(2,1){$a+1$}}
\put(4,3){\makebox(2,1){$a+2$}}
\put(6,3){\makebox(3,1){$A$}}
\end{picture},
\end{center}
where $T_{(1,3)}\geq a+4$.
The entry $T_{(2,2)}$ must be $a+3$, since otherwise 
\begin{align*}
\mathfrak{row}(T)(a+1)&=\mathfrak{row}(T^{\prime})T_{(1,1)}T_{(1,2)}\mathfrak{row}(A)(a+1) \\
&\stackrel{\mathsf{Sp}}{\sim}\mathfrak{row}(T^{\prime})(a+1)(a+2)(a+1)\mathfrak{row}(A) \\
&\stackrel{\mathsf{Sp}}{\sim}(a+1)(a+2)(a+1)\mathfrak{row}(T^{\prime})\mathfrak{row}(A) \\
&\stackrel{\mathsf{Sp}}{\sim}(a+2)(a+1)(a+2)\mathfrak{row}(T^{\prime})\mathfrak{row}(A)
\end{align*}
so that $\mathfrak{row}(T)(a+1)$ is not an FPF-involution word.

\setlength{\unitlength}{12pt}
\begin{center}
\begin{picture}(24,4)
\put(0,1){\makebox(15,2){$T\stackrel{\mathsf{Sp}}{\leftarrow}a(a+1)a=T\stackrel{\mathsf{Sp}}{\leftarrow}(a+1)a(a+1)=$}}
\put(15,3){\line(0,1){1}}
\put(17,2){\line(0,1){2}}
\put(19,1){\line(0,1){3}}
\put(21,2){\line(0,1){2}}
\put(24,3){\line(0,1){1}}
\put(19,1){\line(1,0){2}}
\put(17,2){\line(1,0){6}}
\put(15,3){\line(1,0){9}}
\put(15,4){\line(1,0){9}}
\put(21,0){\makebox(1,1){$\ddots$}}
\put(17,2){\makebox(2,1){$a+3$}}
\put(19,2){\makebox(2,1){$a+4$}}
\put(21,2){\makebox(2,1){$B$}}
\put(15,3){\makebox(2,1){$a+1$}}
\put(17,3){\makebox(2,1){$a+2$}}
\put(19,3){\makebox(2,1){$a+3$}}
\put(21,3){\makebox(3,1){$A^{\prime}$}}
\end{picture},
\end{center}
where 
\setlength{\unitlength}{12pt}
\begin{picture}(3,1)
\put(0,0){\line(0,1){1}}
\put(3,0){\line(0,1){1}}
\put(0,0){\line(1,0){3}}
\put(0,1){\line(1,0){3}}
\put(0,0){\makebox(3,1){$A^{\prime}$}}
\end{picture}
is obtained from 
\setlength{\unitlength}{12pt}
\begin{picture}(3,1)
\put(0,0){\line(0,1){1}}
\put(3,0){\line(0,1){1}}
\put(0,0){\line(1,0){3}}
\put(0,1){\line(1,0){3}}
\put(0,0){\makebox(3,1){$A$}}
\end{picture}
by the initial insertion to the top at the first column with the initial column inserting letter $a+4$.

\item $T_{(1,1)}=a+2$.

Let us write 

\setlength{\unitlength}{12pt}
\begin{center}
\begin{picture}(9,5)

\put(0,2){\makebox(2,1){$T=$}}

\put(2,4){\line(0,1){1}}
\put(4,3){\line(0,1){2}}
\put(5,2){\line(0,1){3}}
\put(6,1){\line(0,1){4}}
\put(9,4){\line(0,1){1}}
\put(6,1){\line(1,0){1}}
\put(5,2){\line(1,0){2}}
\put(4,3){\line(1,0){4}}
\put(2,4){\line(1,0){7}}
\put(2,5){\line(1,0){7}}

\put(7,0){\makebox(1,1){$\ddots$}}
\put(2,4){\makebox(2,1){$a+2$}}
\put(4,3){\makebox(1,1){$b^{\prime}$}}
\put(4,4){\makebox(1,1){$b$}}
\put(5,2){\makebox(1,1){$r$}}
\put(5,3){\makebox(1,1){$q$}}
\put(5,4){\makebox(1,1){$p$}}

\put(6,3){\makebox(2,1){$B$}}
\put(6,4){\makebox(3,1){$A$}}
\end{picture}.
\end{center}

\begin{enumerate}
\item $b=a+3$.

Both $T\stackrel{\mathsf{Sp}}{\leftarrow}a(a+1)a$ and $T\stackrel{\mathsf{Sp}}{\leftarrow}(a+1)a(a+1)$ take the form

\setlength{\unitlength}{12pt}
\begin{center}
\begin{picture}(10,5)

\put(0,4){\line(0,1){1}}
\put(1,3){\line(0,1){2}}
\put(3,2){\line(0,1){3}}
\put(5,1){\line(0,1){4}}
\put(7,3){\line(0,1){2}}
\put(10,4){\line(0,1){1}}
\put(5,1){\line(1,0){2}}
\put(3,2){\line(1,0){4}}
\put(1,3){\line(1,0){8}}
\put(0,4){\line(1,0){10}}
\put(0,5){\line(1,0){10}}

\put(7,0){\makebox(1,1){$\ddots$}}
\put(0,4){\makebox(1,1){$a$}}
\put(1,3){\makebox(2,1){$a+2$}}
\put(1,4){\makebox(2,1){$a+1$}}
\put(3,2){\makebox(2,1){$r$}}
\put(3,3){\makebox(2,1){$b^{\prime}$}}
\put(3,4){\makebox(2,1){$a+3$}}
\put(5,3){\makebox(2,1){$q$}}
\put(5,4){\makebox(2,1){$a+4$}}

\put(7,3){\makebox(2,1){$B$}}
\put(7,4){\makebox(3,1){$A^{\prime}$}}
\end{picture},
\end{center}
where 
\setlength{\unitlength}{12pt}
\begin{picture}(3,1)
\put(0,0){\line(0,1){1}}
\put(3,0){\line(0,1){1}}
\put(0,0){\line(1,0){3}}
\put(0,1){\line(1,0){3}}
\put(0,0){\makebox(3,1){$A^{\prime}$}}
\end{picture}
is obtained from 
\setlength{\unitlength}{12pt}
\begin{picture}(3,1)
\put(0,0){\line(0,1){1}}
\put(3,0){\line(0,1){1}}
\put(0,0){\line(1,0){3}}
\put(0,1){\line(1,0){3}}
\put(0,0){\makebox(3,1){$A$}}
\end{picture}
by the initial insertion to the top at the first column with the initial column inserting letter $a+5(=p+1)$ or $p(>a+4)$.

\item $b>a+3$.

Both $T\stackrel{\mathsf{Sp}}{\leftarrow}a(a+1)a$ and $T\stackrel{\mathsf{Sp}}{\leftarrow}(a+1)a(a+1)$ take the form

\setlength{\unitlength}{12pt}
\begin{center}
\begin{picture}(10,5)

\put(0,4){\line(0,1){1}}
\put(1,3){\line(0,1){2}}
\put(3,2){\line(0,1){3}}
\put(5,1){\line(0,1){4}}
\put(6,3){\line(0,1){2}}
\put(7,4){\line(0,1){1}}
\put(10,4){\line(0,1){1}}
\put(5,1){\line(1,0){1}}
\put(3,2){\line(1,0){3}}
\put(1,3){\line(1,0){7}}
\put(0,4){\line(1,0){10}}
\put(0,5){\line(1,0){10}}

\put(6,0){\makebox(1,1){$\ddots$}}
\put(0,4){\makebox(1,1){$a$}}
\put(1,3){\makebox(2,1){$a+2$}}
\put(1,4){\makebox(2,1){$a+1$}}
\put(3,2){\makebox(2,1){$r$}}
\put(3,3){\makebox(2,1){$b^{\prime}$}}
\put(3,4){\makebox(2,1){$a+3$}}
\put(5,3){\makebox(1,1){$q$}}
\put(5,4){\makebox(1,1){$b$}}
\put(6,4){\makebox(1,1){$p$}}

\put(6,3){\makebox(2,1){$B$}}
\put(7,4){\makebox(3,1){$A$}}
\end{picture}.
\end{center}

\end{enumerate}

\item $T_{(1,1)}<a$.

Let $a^{\prime}=T_{(1,i)}$ be the smallest entry in $T_{1}$ such that $a\leq a^{\prime}$ ($i\neq 1$).

\begin{enumerate}

\item $a^{\prime}=a$.

The entry $T_{(1,i+1)}$ must be $a+1$, since otherwise $\mathfrak{row}(T)a$ is not an FPF-involution word.

\setlength{\unitlength}{12pt}
\begin{center}
\begin{picture}(21,3)

\put(0,1){\makebox(7,1){$T\stackrel{\mathsf{Sp}}{\leftarrow}a(a+1)a=$}}

\put(7,2){\line(0,1){1}}
\put(8,1){\line(0,1){1}}
\put(12,2){\line(0,1){1}}
\put(8,1){\line(1,0){1}}
\put(7,2){\line(1,0){5}}
\put(7,3){\line(1,0){5}}

\put(9,0){\makebox(1,1){$\ddots$}}
\put(10,0){\makebox(1,2){$T^{\prime}$}}
\put(9,2){\makebox(1,1){$T_{1}$}}

\put(12,1){\makebox(9,1){$\leftarrow (a+1)(a+2)(a+1)$}}

\end{picture},
\end{center}
and

\setlength{\unitlength}{12pt}
\begin{center}
\begin{picture}(23,3)

\put(0,1){\makebox(9,1){$T\stackrel{\mathsf{Sp}}{\leftarrow}(a+1)a(a+1)=$}}

\put(9,2){\line(0,1){1}}
\put(10,1){\line(0,1){1}}
\put(14,2){\line(0,1){1}}
\put(10,1){\line(1,0){1}}
\put(9,2){\line(1,0){5}}
\put(9,3){\line(1,0){5}}

\put(11,0){\makebox(1,1){$\ddots$}}
\put(12,0){\makebox(1,2){$T^{\prime}$}}
\put(11,2){\makebox(1,1){$T_{1}$}}

\put(14,1){\makebox(9,1){$\leftarrow (a+2)(a+1)(a+2)$}}

\end{picture}.
\end{center}

\item $a^{\prime}=a+1$.

The entry $T_{(1,i-1)}$ is smaller than or equal to $a-1$.
The entry $T_{(1,i+1)}$ must be $a+2$, since otherwise 
$\mathfrak{row}(T)(a+1)$ is not an FPF-involution word.

\setlength{\unitlength}{12pt}
\begin{center}
\begin{picture}(24,4)

\put(0,1){\makebox(7,1){$T\stackrel{\mathsf{Sp}}{\leftarrow}a(a+1)a=$}}

\put(7,2){\line(0,1){1}}
\put(8,1){\line(0,1){1}}
\put(10,2){\line(0,1){1}}
\put(11,2){\line(0,1){1}}
\put(13,2){\line(0,1){1}}
\put(15,2){\line(0,1){1}}
\put(8,1){\line(1,0){1}}
\put(7,2){\line(1,0){8}}
\put(7,3){\line(1,0){8}}

\put(9,0){\makebox(1,1){$\ddots$}}
\put(10,0){\makebox(1,2){$T^{\prime}$}}
\put(10,2){\makebox(1,1){$a$}}
\put(11,2){\makebox(2,1){$a+1$}}

\put(10,3){\makebox(1,1){$\small{i}$}}
\put(15,1){\makebox(9,1){$\leftarrow (a+1)(a+2)(a+1)$}}

\end{picture},
\end{center}
and

\setlength{\unitlength}{12pt}
\begin{center}
\begin{picture}(26,4)
\put(0,1){\makebox(9,1){$T\stackrel{\mathsf{Sp}}{\leftarrow}(a+1)a(a+1)=$}}
\put(9,2){\line(0,1){1}}
\put(10,1){\line(0,1){1}}
\put(12,2){\line(0,1){1}}
\put(13,2){\line(0,1){1}}
\put(15,2){\line(0,1){1}}
\put(17,2){\line(0,1){1}}
\put(10,1){\line(1,0){1}}
\put(9,2){\line(1,0){8}}
\put(9,3){\line(1,0){8}}
\put(11,0){\makebox(1,1){$\ddots$}}
\put(12,0){\makebox(1,2){$T^{\prime}$}}
\put(12,2){\makebox(1,1){$a$}}
\put(13,2){\makebox(2,1){$a+1$}}
\put(12,3){\makebox(1,1){$\small{i}$}}
\put(17,1){\makebox(9,1){$\leftarrow (a+2)(a+1)(a+2)$}}
\end{picture}.
\end{center}

\item $a^{\prime}\geq a+2$.

The entry $T_{(1,i-1)}$ is smaller than or equal to $a-1$.

\setlength{\unitlength}{12pt}
\begin{center}
\begin{picture}(20,4)
\put(0,1){\makebox(7,1){$T\stackrel{\mathsf{Sp}}{\leftarrow}a(a+1)a=$}}

\put(7,2){\line(0,1){1}}
\put(8,1){\line(0,1){1}}
\put(10,2){\line(0,1){1}}
\put(11,2){\line(0,1){1}}
\put(13,2){\line(0,1){1}}
\put(15,2){\line(0,1){1}}
\put(8,1){\line(1,0){1}}
\put(7,2){\line(1,0){8}}
\put(7,3){\line(1,0){8}}

\put(9,0){\makebox(1,1){$\ddots$}}
\put(10,0){\makebox(1,2){$T^{\prime}$}}
\put(10,2){\makebox(1,1){$a$}}
\put(11,2){\makebox(2,1){$a+1$}}

\put(10,3){\makebox(1,1){$\small{i}$}}
\put(15,1){\makebox(5,1){$\leftarrow a^{\prime}a^{\prime\prime}(a+1)$}}
\end{picture},
\end{center}
and

\setlength{\unitlength}{12pt}
\begin{center}
\begin{picture}(22,4)
\put(0,1){\makebox(9,1){$T\stackrel{\mathsf{Sp}}{\leftarrow}(a+1)a(a+1)=$}}

\put(9,2){\line(0,1){1}}
\put(10,1){\line(0,1){1}}
\put(12,2){\line(0,1){1}}
\put(13,2){\line(0,1){1}}
\put(15,2){\line(0,1){1}}
\put(17,2){\line(0,1){1}}
\put(10,1){\line(1,0){1}}
\put(9,2){\line(1,0){8}}
\put(9,3){\line(1,0){8}}

\put(11,0){\makebox(1,1){$\ddots$}}
\put(12,0){\makebox(1,2){$T^{\prime}$}}
\put(12,2){\makebox(1,1){$a$}}
\put(13,2){\makebox(2,1){$a+1$}}

\put(12,3){\makebox(1,1){$\small{i}$}}
\put(17,1){\makebox(5,1){$\leftarrow a^{\prime}(a+1)a^{\prime\prime}$}}
\end{picture},
\end{center}
where $a^{\prime\prime}=T_{(1,i+1)}$.

Since $(a+1)(a+2)(a+1)\stackrel{\mathsf{CK}}{\sim}(a+2)(a+1)(a+2)$ in cases (a) and (b) and $a^{\prime}a^{\prime\prime}(a+1)\stackrel{\mathsf{CK}}{\sim}a^{\prime}(a+1)a^{\prime\prime}$ $(a+1<a^{\prime}<a^{\prime\prime})$ in case (c), the same argument as in \textbf{Case 5} of the proof of Lemma~\ref{lem:bac} shows $T\stackrel{\mathsf{Sp}}{\leftarrow}a(a+1)a=T\stackrel{\mathsf{Sp}}{\leftarrow}(a+1)a(a+1)$.

\end{enumerate}

\end{enumerate}

\end{proof}

\section{Queer supercrystal structure for increasing factorizations of involution words}

Involutions words are orthogonal Hecke words~\cite{Mar} with minimum length associated with an element $z \in \mathfrak{I}_{\infty}=\left\{ z\in \mathfrak{S}_{\infty} \relmiddle| z^{2}=1 \right\}$.
Here, we just mention the defining properties of involution words.
For more details we refer the reader to \cite{Mar}.
Two involution words, $w$ and $w^{\prime}$ are called \emph{equivalent}, denoted by $w\stackrel{\mathsf{O}}{\sim} w^{\prime}$, if $w^{\prime}$ can be obtained from $w$ by a finite sequence of relations on consecutive letters $ab\sim ba$ ($\left\vert a-b\right\vert >1$) and $a(a+1)a\sim (a+1)a(a+1)$ and the relation with $i_{1}i_{2}\cdots i_{n}\sim i_{2}i_{1}\cdots i_{n}$.

\begin{thm}[\cite{HMP17}]
An orthogonal Hecke word is an involution word if and only if its equivalence class contains no words with equal adjacent letters.
\end{thm}

Parallel to the FPF-involution Coxeter-Knuth insertion, the involution Coxeter-Knuth insertion is also defined~\cite{Mar}.
In the involution Coxeter-Knuth insertion, the inserting word is an involution word $w$ and the procedures (1), (2), and (3) in the algorithm of the FPF-involution Coxeter-Knuth insertion are replaced by the following two.

\begin{itemize}
\item[(1)] If $a=b$, then leave $L$ unchanged and insert $a+1$ to the row below if $L$ is a row or to the next column to the right if $L$ is a column. 
\item[(2)] If $a\neq b$, then  replace $b$ by $a$ in $L$ and insert $b$ to the row blow if $L$ is a row or to the next column to the right if $L$ is a column or $b$ was on the main diagonal.
\end{itemize}
The insertion tableau is denoted by $P_{\mathsf{O}}(w)$ and the recording tableau by $Q_{\mathsf{O}}(w)$.
This algorithm is equivalent to the orthogonal Hecke insertion~\cite{HKPWZZ,Mar,PP} restricted to involution words.

The set of involution words is denoted by $\hat{\mathcal{R}}(z)$ for $z\in \mathfrak{I}_{\infty}$.
Given $w\in \hat{\mathcal{R}}(z)$ for $z\in \mathfrak{I}_{\infty}$, an increasing factorization of $w$ is a factorization $w^{1}w^{2}\cdots w^{m}$ such that $w=w^{1}w^{2}\cdots w^{m}$ with $\left\vert w \right\vert =\left\vert w^{1} \right\vert + \cdots +\left\vert w^{m} \right\vert $ and each factor $w^{i}$ is strictly increasing.
For $z\in \mathfrak{I}_{\infty}$, we denote by $\mathrm{RF}^{m}(z)$ the set of all increasing factorizations with $m$ blocks of all involution words $\hat{\mathcal{R}}(z)$.

\begin{thm} \label{thm:q-involution}
Let $z\in \mathfrak{I}_{\infty}$.
Then, the set $\mathrm{RF}^{m}(z)$ admits a $\mathfrak{q}(m)$-crystal structure.
The even Kashiwara operators are $\Tilde{e}_{i}^{F}$ and $\Tilde{f}_{i}^{F}$ ($i=1,2,\ldots,m-1$) described in Section~\ref{sec:factorization}, whereas the odd Kashiwara operators are given in Lemma~\ref{lem:f0F2} and \ref{lem:e0F2}.
\end{thm}

\begin{lem} \label{lem:f0F2}
Let $w^{1}w^{2}\cdots w^{m}\in \mathrm{RF}^{m}(z)$ be an increasing factorization of $w\in \hat{\mathcal{R}}(z)$ for $z\in \mathfrak{I}_{\infty}$.
The action of the odd Kashiwara operator $\Tilde{f}_{\Bar{1}}^{F}$ on $w^{1}w^{2}\cdots w^{m}$ is given by the following rule:

$\Tilde{f}_{\Bar{1}}^{F}$ always changes the first two factors if $\Tilde{f}_{\Bar{1}}^{F}(w^{1}w^{2}\cdots w^{m})\neq \boldsymbol{0}$. 

\begin{enumerate}

\item $\left\vert w^{1} \right\vert =0$.

$\Tilde{f}_{\Bar{1}}^{F}(w^{1}w^{2}\cdots w^{m})= \boldsymbol{0}$.

\item $\left\vert w^{1} \right\vert \geq 1$.

If $\min(\mathrm{cont}(w^{1}))=u_{1}<\min (\mathrm{cont}(w^{2}))$, then $\Tilde{f}_{\Bar{1}}^{F}(w^{1}w^{2}\cdots w^{m})=\Tilde{w}^{1}\Tilde{w}^{2}\cdots w^{m}$, where $\mathrm{cont}(\Tilde{w}^{1})=\mathrm{cont}(w^{1})\backslash \{u_{1}\}$ and $\mathrm{cont}(\Tilde{w}^{2})=\mathrm{cont}(w^{2})\cup \{u_{1}\}$.
Otherwise, $\Tilde{f}_{\Bar{1}}^{F}(w^{1}w^{2}\cdots w^{m})= \boldsymbol{0}$.
 
\end{enumerate}
\end{lem}

\begin{lem} \label{lem:e0F2}
Let $w^{1}w^{2}\cdots w^{m}\in \mathrm{RF}^{m}(z)$ be an increasing factorization of $w\in \hat{\mathcal{R}}(z)$ for $z\in \mathfrak{I}_{\infty}$.
The action of the odd Kashiwara operator $\Tilde{e}_{\Bar{1}}^{F}$ on $w^{1}w^{2}\cdots w^{m}$ is given by the following rule:

$\Tilde{e}_{\Bar{1}}^{F}$ always changes the first two factors if $\Tilde{f}_{\Bar{1}}^{F}(w^{1}w^{2}\cdots w^{m})\neq \boldsymbol{0}$. 

\begin{enumerate}

\item $\left\vert w^{2} \right\vert =0$.

$\Tilde{e}_{\Bar{1}}^{F}(w^{1}w^{2}\cdots w^{m})= \boldsymbol{0}$.

\item $\left\vert w^{2} \right\vert \geq 1$.

If $\min(\mathrm{cont}(w^{2}))=v_{1}<\min (\mathrm{cont}(w^{1}))$, then $\Tilde{e}_{\Bar{1}}^{F}(w^{1}w^{2}\cdots w^{m})=\Tilde{w}^{1}\Tilde{w}^{2}\cdots w^{m}$, where $\mathrm{cont}(\Tilde{w}^{1})=\mathrm{cont}(w^{1})\cup \{v_{1}\}$ and $\mathrm{cont}(\Tilde{w}^{2})=\mathrm{cont}(w^{2})\backslash \{v_{1}\}$.
Otherwise, $\Tilde{e}_{\Bar{1}}^{F}(w^{1}w^{2}\cdots w^{m})= \boldsymbol{0}$.
 
\end{enumerate}
\end{lem}

\begin{lem} \label{lem:const2}
Let $u^{1}u^{2}\cdots u^{m}$ and $v^{1}v^{2}\cdots v^{m}$ be two vertices in the same connected component of the $\mathfrak{q}(m)$-crystal $\mathrm{RF}^{m}(z)$ for $z\in \mathfrak{I}_{\infty}$.
Then, two words $u=u^{1}u^{2}\cdots u^{m}$ and $v=v^{1}v^{2}\cdots v^{m}$ have the same insertion tableau; $P_{\mathsf{O}}(u)=P_{\mathsf{O}}(v)$.
\end{lem}

Since proofs of Theorem~\ref{thm:q-involution}, and Lemmas~\ref{lem:f0F2}, \ref{lem:e0F2}, and \ref{lem:const2} are much the same as those of Theorems~\ref{thm:q-FPF}, and Lemmas~\ref{lem:f0F1}, \ref{lem:e0F1}, and \ref{lem:const1}, respectively, we omit them.

\subsection*{Acknowledgements}
The author would like to express his gratitude to Professor Masato Okado for helpful discussions.
This work was supported by Osaka City University Advanced Mathematical Institute (MEXT Joint Usage/Research Center on Mathematics and Theoretical Physics).

\end{document}